\documentclass{jT}

\usepackage{amssymb,amsmath,latexsym,amsthm, xy}
\usepackage[english]{babel}

\numberwithin{equation}{section}
\newtheorem{theorem}{Theorem}[section]
\newtheorem{lemma}[theorem]{Lemma}
\newtheorem{prop}[theorem]{Proposition}
\newtheorem{corollary}[theorem]{Corollary}

\theoremstyle{definition}
\newtheorem{definition}[theorem]{Definition}
\newtheorem{remark}[theorem]{Remark}

\newtheorem{question}[theorem]{Question}

\def\ints{{\mathbb Z}}
\def\rats{{\mathbb Q}}
\def\nats{{\mathbb N}}
\def\complex{{\mathbb C}}
\def\proj{{\mathbb P}}

\def\FF{{\mathbb F}}
\def\ord{{\text{ord}}}
\DeclareMathOperator{\eff}{eff}
\DeclareMathOperator{\bottom}{bot}
\def\Frac{{\text{Frac}}}
\DeclareMathOperator{\Aut}{Aut}
\def\Gal{{\text{Gal}}}
\def\Ind{{\text{Ind}}}

\def\Spec{{\mbox{Spec }}}
\def\mc#1{\mathcal{#1}}
\def\ol#1{\overline{#1}}

\def\matrix#1#2#3#4{
    \left( \begin{array}{cc} #1&#2 \\ #3&#4 \end{array} \right)}

\thanks{The author was supported by a NDSEG Graduate Research Fellowship and an NSF Postdoctoral 
Research Fellowship in the Mathematical Sciences.  Some editing was done while a guest at the Max-Planck-Institut F\"{u}r Mathematik in Bonn.
Final submission occurred while the author was supported by NSF grant DMS-1265290}

\date{\today}


\title[Fields of moduli, II]{Fields of moduli of three-point $G$-covers with cyclic $p$-Sylow, II}

\author{\sc Andrew OBUS}
\address{Andrew Obus\\
University of Virginia\\
141 Cabell Drive\\
Charlottesville, VA 22904}
\email{andrewobus@gmail.com}
\urladdr{people.virginia.edu/~aso9t/}

\begin{document}

\subjclass[2000]{14 G 20, 11 G 22, 14 H 30, 14 H 25, 14 G 25, 11 G 20, 11 S 20}

\maketitle

\begin{resume}
Nous poursuivons l'\'etude de la r\'eduction stable et des corps de modules des $G$-rev\^etements galoisiens de la droite projective sur un corps 
discr\`etement valu\'e de caract\'eristique mixte $(0, p)$, dans le cas o\`u $G$ a un $p$-sous-groupe de Sylow cyclique d'ordre $p^n$. Supposons de 
plus que le normalisateur de $P$ agit sur lui m\^eme via une involution. Sous des hypoth\`eses assez l\'eg\`eres, nous montrons que si $f : Y \to 
\proj^1$ est un $G$-rev\^etement galoisien ramifi\'e au-dessus de $3$ points, d\'efini sur $\complex$, alors les $n$-i\`emes groupes de ramification 
sup\'erieure au-dessus de $p$, en num\'erotation supŽrieure, de (la cl\^oture galoisienne de) l'extension $K/\rats$ sont triviaux, o\`u $K$ est le 
corps des modules de $f$.
\end{resume}

\begin{abstr}
We continue the examination of the stable reduction and fields of moduli of $G$-Galois covers of the projective line over a 
complete discrete valuation field of mixed characteristic $(0, p)$, where $G$ has a \emph{cyclic} $p$-Sylow subgroup $P$ of order 
$p^n$.  Suppose further that the normalizer of $P$ acts on $P$ via an involution.
Under mild assumptions, if $f: Y \to \proj^1$ is a three-point $G$-Galois cover defined over $\complex$, then the $n$th higher ramification groups above $p$ 
for the upper numbering of the (Galois closure of the) extension $K/\rats$ vanish, where $K$ is the field of moduli of $f$. 
\end{abstr}

\section{Introduction}\label{CHintro}

\subsection{Overview}\label{Soverview}
This paper continues the work of the author in \cite{Ob:fm}
about ramification of primes of $\rats$ in fields of moduli of three-point $G$-Galois covers of the
Riemann sphere.  We place bounds on the ramification of the prime $p$ when a $p$-Sylow subgroup $P$ of $G$ 
is cyclic of arbitrary order (Theorem \ref{Tmain}).  
This was done in \cite{Be:rp} for $P$ trivial, and in \cite{We:br} for $|P| = p$.  In \cite{We:br}, Wewers used a detailed 
analysis of the stable reduction of the cover to characteristic $p$, inspired by results
of Raynaud (\cite{Ra:sp}).  
However, many of the results on stable reduction in the literature (particularly, those in the second half of 
\cite{Ra:sp}) are only applicable when $|P| = p$.  

In \cite{Ob:vc}, the author generalized much of \cite{Ra:sp} to the case 
where $P$ is cyclic of any order.  
In \cite{Ob:fm}, these results were applied under the additional assumption that $G$ is \emph{$p$-solvable}
(i.e., has no nonabelian simple composition factors with order divisible by $p$) to place bounds on ramification of $p$ in the 
field of moduli of a three-point $G$-cover.  In this paper,
we drop the assumption of $p$-solvability, but we assume that the normalizer of a $p$-Sylow subgroup $P$ acts on $P$ 
via an involution.  This 
hypothesis is satisfied for many non-$p$-solvable groups (Remark \ref{Rmainaddenda} (1)), and was used as a simplifying 
assumption in \cite{BW:hu} 
(in the case where $|P| = p$) to examine the reduction of four-point $G$-covers of $\proj^1$.  It will simplify matters in 
our situation as well.

Let $f: Y \to X \cong \proj^1_{\complex}$ be a finite, connected, $G$-Galois branched cover of Riemann surfaces, branched only at
$\rats$-rational points.  Such a cover can always be defined over $\ol{\rats}$.  If there are exactly three branch points (without loss of generality, $0$, 
$1$, and $\infty$), such a cover is called a \emph{three-point cover}.
The fixed field in $\ol{\rats}$ of all elements of $\Aut(\ol{\rats}/\rats)$ fixing the isomorphism class of $f$ as a $G$-cover 
(i.e., taking into account the $G$-action) is a number field called the \emph{field of moduli} of $f$ (as a $G$-cover).
By \cite[Proposition 2.7]{CH:hu}, it is also the intersection of all fields of definition of $f$ (along with the $G$-action).
For more details, see, e.g., \cite{CH:hu} or the introduction to \cite{Ob:fm}.  

Since a branched $G$-Galois cover $f: Y \to X$ of the Riemann sphere is
given entirely in terms of algebraic data (the branch locus $C$, the Galois group $G$, 
and an element $g_i \in G$ for each $c_i \in C$ such that $\prod_i g_i = 1$ and the $g_i$ generate $G$), 
it is reasonable to try to draw inferences about the field of moduli of $f$
based on these data.  But this is a deep question, as the relation between topology of covers and their defining equations 
is given by \cite{gaga}, where the methods are non-constructive. 

\subsection{Main result}
Let $f: Y \to X = \proj^1$ be a three-point $G$-cover.  If $p \nmid |G|$, then $p$ is unramified in the field of 
moduli of $f$ (\cite{Be:rp}).
If $G$ has a $p$-Sylow group of order $p$ (thus cyclic), then $p$ is tamely ramified in the field of moduli of $f$ (\cite{We:br}).  
Furthermore, if $G$ has a \emph{cyclic} $p$-Sylow subgroup
of order $p^n$ and is $p$-solvable (i.e., has no nonabelian simple composition factors with order divisible by $p$), then
the $n$th higher ramification groups above $p$ for the upper numbering of (the Galois closure of) $K/\rats$ vanish, where $K$ is the field of 
moduli of $f$ (\cite{Ob:fm}).  Our main result extends this to many non-$p$-solvable groups.

If $H$ is a subgroup of $G$, we write $N_G(H)$
for the normalizer of $H$ in $G$ and $Z_G(H)$ for the centralizer of $H$ in $G$.

\begin{theorem}\label{Tmain}
Let $f: Y \to X$ be a three-point $G$-Galois cover of the Riemann sphere, and
suppose that a $p$-Sylow subgroup $P \leq G$ 
is cyclic of order $p^n$.  Suppose $|N_G(P)/Z_G(P)| = 2$.  Lastly, suppose that at least 
one of the three branch points has prime-to-$p$ branching index (if exactly one, then we require $p \ne 3$).
If $K/\rats$ is the field of moduli of $f$,  
then the $n$th higher ramification groups for the upper numbering of the Galois closure of $K/\rats$
vanish. 
\end{theorem}

\begin{remark}\label{Rlocaltoglobal} 
By \cite[Proposition 7.1]{Ob:fm}, it suffices to prove Theorem \ref{Tmain} for covers over $\ol{\rats_p^{ur}}$, rather than 
$\ol{\rats}$, and for higher ramification groups over $\rats_p^{ur}$.  Here, and throughout, $\rats_p^{ur}$ is the 
\emph{completion} of the 
maximal unramified extension of $\rats_p$.  This is the version we prove in \S\ref{CHmain}.  In fact, we prove even more, i.e.,
that the \emph{stable model} of $f$ can be defined over an extension $K/\rats_p^{ur}$ whose $n$th higher ramification groups for the 
upper numbering vanish above $p$.
\end{remark}

\begin{remark}\label{Rmainaddenda}
\begin{enumerate}
\item Many simple groups $G$ satisfy the hypotheses of Theorem \ref{Tmain}
(for instance, any $PSL_2(\ell)$ where $p \ne 2$ and $v_p(\ell^2 - 1) = n$).
\item If $N_G(P) = Z_G(P)$, then $G$ is $p$-solvable by a theorem of Burnside (\cite[Theorem 4, p.\ 169]{Za:tg}), and thus falls within the 
scope of \cite[Theorem 1.3]{Ob:fm}, which treats $p$-solvable groups.  Theorem \ref{Tmain} seems to be the next easiest case.    
\item
I expect Theorem \ref{Tmain} to hold even if $p=3$ or if all branching
indices are divisible by $p$.  See Question \ref{Qgap}.
\end{enumerate}
\end{remark}

\begin{remark}\label{Rmodular}
In the case that all three branch points of $f$ as in Theorem \ref{Tmain} have prime-to-$p$ branching index, we show (Proposition
\ref{Pm2tau3}) that $f$ actually has good reduction.  We use this in Corollary \ref{Cmodular} to give a proof that, when $N$ is prime, 
the modular curve $X(N)$ has good reduction at all primes not dividing $6N$.
In fact, Corollary \ref{Cmodular} is more general, and works (in a slightly weakened sense) for all $N$.
Our proof does not use the modular interpretation of $X(N)$, only that $X(N) \to X(1)$ can be given as a three-point $PSL_2(\ints/N)$-cover.
\end{remark}

As in \cite{Ob:fm}, our main technique for proving Theorem \ref{Tmain} will be an analysis of the 
\emph{stable reduction} of the $G$-cover $f$ to characteristic $p$.
The major difference between the methods of this paper and those of \cite{Ob:fm} is this paper's use of the \emph{auxiliary 
cover}.  This is a construction, introduced by Raynaud (\cite{Ra:sp}), to simplify the group-theoretical structure of a 
Galois cover of curves.  In particular, by replacing $f$ with its auxiliary cover $f^{aux}$, we obtain a Galois cover whose
field of moduli is related to that of $f$, but which now has a $p$-solvable Galois group.  

Unfortunately, the cover $f^{aux}$ in general has extra branch points, and it is not obvious where these extra branch points 
arise.  The crux of the proof of Theorem \ref{Tmain} is understanding where these extra branch points are located; this is the
content of \S\ref{Setaletails}, and is the reason why Theorem \ref{Tmain} is significantly more difficult than the analogous
theorem where $G$ is assumed to be $p$-solvable.  Our assumption that $|N_G(P)/Z_G(P)| = 2$ alleviates this difficulty 
somewhat, as it allows us to work with more explicit equations (see \S\ref{Stauequals1} and Question \ref{Qdeformation}).

\subsection{Section-by-section summary and walkthrough}
In \S\ref{Smupn}, we give some explicit results on the reduction of $\mu_{p^n}$-torsors.  
The purpose of \S\ref{Sdisks}, \S\ref{Sstable}, and \S\ref{Sdefdata} is to recall the relevant material from \cite{Ob:fm}.
In \S\ref{Sdisks} and \S\ref{Sstable}, we 
introduce stable reduction of $G$-covers, and state some of the basic properties.  
We also recall the \emph{vanishing cycles formula} from \cite{Ob:fm}, which is indispensable for the proof of 
Theorem \ref{Tmain}.
In \S\ref{Sdefdata}, we recall the properties of \emph{deformation data} (constructed in \cite{Ob:vc}), which give extra structure to the stable model. 

The new part of the paper begins with \S\ref{Scomb}, where we discuss \emph{monotonicity} of stable reduction 
(a property that becomes relevant when $v_p(|G|) > 1$), and show that it is satisfied for all covers in this paper.
In \S\ref{Saux}, we introduce the \emph{auxiliary cover} and the \emph{strong auxiliary cover}, and discuss why they
are useful in calculating the field of moduli.

In \S\ref{CHmain}, we prove our main result, Theorem \ref{Tmain}.  The proof is divided into \S\ref{Stauequals3}, 
\S\ref{Stauequals2}, and \S\ref{Stauequals1}, corresponding to the case of $3$, $2$, and $1$ branch point(s) with 
prime-to-$p$ branching index, respectively.  The proof for $1$ branch point with prime-to-$p$ index is by far the most difficult (as 
it involves the appearance of an extra branch point in the auxiliary cover),
and toward the beginning of \S\ref{Stauequals1}, we give an outline of the proof and of how it is split up over
\S\ref{Sprelimlemmas}--\S\ref{Shigherram}.  

In \S\ref{CHfurther}, we consider some questions arising from this work.  In Appendix \ref{Awildexample} we give an example of
a three-point cover with nontrivial wild monodromy (see the appendix for more details).  Appendix \ref{Awildexample} is not needed for the
rest of the paper.

For reasons that will become clear in \S\ref{Stauequals1}, the case $p = 5$ presents some complications.
The reader who is willing to assume $p > 5$ may skip Lemma \ref{Lartinschreier} (2), Remark \ref{R0c5}, Proposition
\ref{Pa0} (2), Remark \ref{R1c5}, Proposition \ref{Pa1} (2), Proposition \ref{Pinseptailloc} (2), (3), Lemma \ref{L5insep}, and Proposition
\ref{Pfstrstdef} (2c), which are among the more technical parts of the paper.

\subsection{Notation and conventions}\label{Snotations}
The letter $k$ will always represent an algebraically closed field of characteristic $p>0$.

If $H$ is a subgroup of a finite group $G$, then $N_G(H)$ is the normalizer of $H$ in $G$ and $Z_G(H)$ is the centralizer of $H$ in $G$.  
If $G$ has a cyclic $p$-Sylow subgroup $P$, and $p$ is understood, we write $m_G = |N_G(P)/Z_G(P)|$.  

If $K$ is a field, $\ol{K}$ is its algebraic closure.  We write $G_K$ for the absolute Galois group of $K$.  If $H \leq
G_K$, we write $\ol{K}^H$ for the fixed field of $H$ in $\ol{K}$.  Similarly, if $\Gamma$ is a group of automorphisms of a
ring $A$, we write $A^{\Gamma}$ for the fixed ring under $\Gamma$.

We use the standard theory of higher ramification groups for the upper and lower numbering from \cite[IV]{Se:lf}.  However, unlike in 
\cite{Se:lf}, if $L/K$ is a nontrivial $G$-Galois extension of complete discrete valuation rings with algebraically closed residue fields, 
then the \emph{conductor} of $L/K$, written $h_{L/K}$, will be for us the greatest upper jump (i.e., the greatest $i$ such that $G^i \neq \{id\}$).  
This is consistent with \cite{Ob:fm}.

If $R$ is any local ring, then $\hat R$ is the completion of $R$ with respect to its maximal ideal. 
If $R$ is any ring with a non-archimedean absolute value $| \cdot |$, then $R\{T\}$ 
is the ring of power series $\sum_{i=0}^{\infty} c_i T^i$ such that $\lim_{i \to \infty} |c_i| = 0$.  
If $R$ is a discrete valuation ring with fraction field $K$ of characteristic 0 and residue field $k$ of
characteristic $p$, we normalize the absolute value on $K$ and on any subring of $K$ so that $|p| = 1/p$.  We always normalize 
the valuation on $K$ so that $p$ has valuation $1$.

A \emph{branched cover} $f: Y \to X$ is a finite, surjective, generically \'{e}tale morphism of geometrically connected, smooth, 
proper curves. 
If $f$ is of degree $d$ and we choose an isomorphism $i: G \to \Aut(Y/X)$, then the datum $(f, i)$
is called a \emph{$G$-Galois cover} (or just a \emph{$G$-cover}, for short).  We will usually suppress the isomorphism
$i$, and speak of $f$ as a $G$-cover. 
 
Suppose $f: Y \to X$ is a $G$-cover of smooth curves, and $K$ is a field of definition for $X$.  Then the \emph{field of
moduli of $f$ relative to $K$ (as a $G$-cover)} is $\ol{K}^{\Gamma^{in}}$, where
$\Gamma^{in} = \{\sigma \in G_K | f^{\sigma} \cong f \text{ (as } G\text{-covers)}\}$ (see, e.g., \cite[\S1.1]{Ob:fm}).
If $X$ is $\proj^1$, then the \emph{field of moduli of $f$} means the field of moduli of $f$ (as a $G$-cover) relative to $\rats$.

Let $f: Y \to X$ be any morphism of schemes and assume $H$ is a finite group with $H \hookrightarrow \Aut(Y/X)$.  If
$G$ is a finite group containing $H$, then there is a map $\Ind_H^G f: \Ind_H^G Y \to X$, where $\Ind_H^G Y$ is a disjoint union of
$[G:H]$ copies of $Y$, indexed by the left cosets of $H$ in $G$.  The group $G$ acts on $\Ind_H^G Y$, and the stabilizer
of each copy of $Y$ in $\Ind_H^G Y$ is a conjugate of $H$.

The set $\nats$ is equal to $\{1, 2, 3, \ldots \}$.

\section*{Acknowledgements}
This material is mostly adapted from my PhD thesis, and I thank my advisor, David Harbater, for much help 
related to this work.  I also thank the referee for useful suggestions. 
\section{Reduction of $\mu_{p^n}$-torsors}\label{Smupn}

Let $R$ be a mixed characteristic $(0, p)$ complete discrete valuation ring with residue field $k$ and fraction field $K$.
Let $\pi$ be a uniformizer of $R$.  Recall that we normalize the valuation of $p$ (not $\pi$) to be 1.
For any scheme or algebra $S$ over $R$, write $S_K$ and $S_k$ for its base
changes to $K$ and $k$, respectively.

We state a partial converse of \cite[Lemma 3.1]{Ob:fm}, which will be used repeatedly in analyzing the stable
reduction of covers (see \S\ref{Stauequals1}):
\begin{lemma}\label{Lartinschreier}
Suppose $R$ contains the $p^n$th roots of unity.  Let $X = \Spec A$, where $A = R\{T\}$.  
Let $f: Y_K \to X_K$ be a $\mu_{p^n}$-torsor given by the equation $y^{p^n} = g$,
where $g = 1 + \sum_{i=1}^{\infty} c_iT^i$.  
Suppose that $v(c_i) > n + \frac{1}{p-1}$ for all $i > p$ divisible by $p$.
Suppose further that (at least) one of the following two conditions holds:
\begin{enumerate}
\item There exists $i$ such that $v(c_i) < \min(v(c_p), n + \frac{1}{p-1})$.
\item $v(c_p) > n - \frac{p-2}{2(p-1)}$ and there exists $c_p' \in R$ with $v(c'_p - c_p) > n + \frac{1}{p-1}$
and  $v\left(c_1 - \sqrt[p]{c'_p p^{(p-1)n+1}}\right) < n + \frac{1}{p-1}$.
\end{enumerate}
Then, even after a possible finite extension of $K$, the map $f: Y_K \to X_K$ does \emph{not} split into a union of $p^{n-1}$ connected disjoint 
$\mu_p$-torsors, such that if $Y$ is the normalization of $X$ in the total ring of fractions of $Y_K$, then $Y_k \to X_k$ is \'{e}tale.
\end{lemma}

\begin{proof}
Suppose we are in case (1).  Pick $b \in R$ such that $v(b) = \min_i(v(c_i))$.  Then $v(b) < n + \frac{1}{p-1},$ and $g = 1+bw$ with 
$w \in A \backslash \pi A$.  Let $a < n$ be the greatest integer such that $a + \frac{1}{p-1} < v(b)$.
Then $g$ has a $p^{a}$th root in $A$, given by the binomal expansion 
$$\sqrt[p^{a}]{g} = 1 + \frac{1/p^{a}}{1!} bw + \frac{(1/p^{a})((1/p^{a}) - 1)}{2!} (bw)^2 + \cdots.$$  
Since $v(b) > a + \frac{1}{p-1}$, this series converges, and is in $A$.  Furthermore, 
since the coefficients of all terms in this series of degree $\geq 2$ have valuation greater than  $v(b) - a$, the series can be written as
$\sqrt[p^{a}]{g} = 1 + \frac{b}{p^a}u$, where $u$ is congruent to $w$ (mod $\pi$).

Now, $v(b) - a = v(\frac{b}{p^a}) \leq 1 + \frac{1}{p-1}$.  Furthermore, by assumption (1), the reduction $\ol{u}$ of $u$ 
is not a $p$th power in $A/\pi$.  Then \cite[Proposition 1.6]{He:ht} shows that $\sqrt[p^{a}]{g}$ is not a $p$th power in $A$ (nor in $K$).  
If $a < n-1$, this proves that $f$ does not split into a disjoint union of $n-1$ torsors.  If $a = n-1$, then $v(\frac{b}{p^a}) < 1 + \frac{1}{p-1}$, and 
\cite[Proposition 1.6]{He:ht} shows that the torsor given by $y^p = \sqrt[p^{n-1}]{g}$ does not have \'{e}tale reduction.  This proves the lemma
in case (1).

Suppose we are in case (2) and not in case (1).  It then suffices to show that there exists $h \in A$ such that $h^{p^n}g$ satisfies (1).
Let $\eta = -\sqrt[p]{\frac{c'_p}{p^{n-1}}}$ (any $p$th root will do).  
Now, by assumption, $v(c'_p) - (n - 1) > \frac{p}{2(p-1)}$, so $v(\eta) > \frac{1}{2(p-1)}$.
Then there exists $\epsilon > 0$ such that
$$(1 + \eta T)^{p^n} \equiv 1 - p^n\sqrt[p]{\frac{c'_p}{p^{n-1}}}T - \binom{p^n}{p}\frac{c'_p}{p^{n-1}}T^p \pmod {p^{n + \frac{1}{p-1} + \epsilon}}.$$  
It is easy to show that $\binom{p^n}{p} \equiv p^{n-1} \pmod {p^n}$ for all $n \geq 1$.  So there exists $\epsilon > 0$ such that
$$(1 + \eta T)^{p^n} \equiv 1 - \sqrt[p]{c'_p p^{(p-1)n +1}}T - c'_pT^p \pmod {p^{n + \frac{1}{p-1} + \epsilon}}.$$
Using the assumption that $v(c_1 - \sqrt[p]{c'_p p^{(p-1)n+1}}) < n + \frac{1}{p-1}$, we leave it to the reader to verify that $(1+\eta T)^{p^n}g$ satisfies 
(1) (with $i  = 1$).
\end{proof}

\section{Semistable models of $\proj^1$}\label{Sdisks}
Let $R$ be a mixed characteristic $(0, p)$ complete discrete valuation ring with residue field $k$ and fraction field $K$.
If $X$ is a smooth curve over $K$, then a \emph{semistable}
model for $X$ is a relative flat curve $X_R \to \Spec R$ with $X_R \times_R K \cong X$ and semistable special fiber (i.e.,
the special fiber is reduced with only ordinary double points for singularities).  If $X_R$ is smooth, it is called a \emph{smooth model}.

\subsection{Models}
Let $X \cong \proj^1_K$.  Write $v$ for the valuation on $K$.
Let $X_R$ be a smooth model of $X$ over $R$.  
Then there is an element $T \in K(X)$ such that $K(T) \cong K(X)$ and the local
ring at the generic point of the special fiber of $X_R$ is the valuation ring of $K(T)$ corresponding to the
Gauss valuation (which restricts to $v$ on $K$).  We say that our model
corresponds to the Gauss valuation on $K(T)$, and we call $T$ a
\emph{coordinate} of $X_R$.  Conversely, if $T$ is
any rational function on $X$ such that $K(T) \cong K(X)$, there is a smooth
model $X_R$ of $X$ such that $T$ is a coordinate of $X_R$.
In simple terms, $T$ is a coordinate of $X_R$ iff, for all $a, b \in R$, the subvarieties of $X_R$ cut out by $T-a$ and $T-b$ intersect exactly when
$v(a-b) > 0$.  

Now, let $X_R'$ be a semistable model of $X$ over $R$.
The special fiber of $X_R'$ is a tree-like configuration of $\proj^1_k$'s.
Each irreducible component $\ol{W}$ of the special fiber $\ol{X}$ of $X_R'$ yields
a smooth model of $X$ by blowing down all other irreducible components of $\ol{X}$.  
If $T$ is a coordinate on the smooth model of $X$ with $\ol{W}$ as special fiber, we will say that $T$ corresponds to $\ol{W}$.

\subsection{Disks and annuli}\label{Sdna}

We give a brief overview here.  For more details, see \cite{He:dc}.

Let $X_R'$ be a semistable model for $X = \proj^1_K$.  Suppose $x$ is a smooth point
of the special fiber $\ol{X}$ of $X_R'$ on the 
irreducible component $\ol{W}$.  Let $T$ be a coordinate corresponding to $\ol{W}$ such
that $T = 0$ specializes to $x$.  Then the set of points of $X(\ol{K})$ which specialize
to $x$ is the \emph{open $p$-adic disk} $D$ given by  
$v(T) > 0$.  The ring of functions on the formal disk corresponding to $D$ is $\hat{\mc{O}}_{X, x} \cong R\{T\}$.

Now, let $x$ be an ordinary double point of $\ol{X}$, at the
intersection of components $\ol{W}$ and $\ol{W}'$.  Then the 
set of points of $X(\ol{K})$ which specialize to $x$ is an \emph{open annulus} $A$.  If $T$
is a coordinate corresponding to $\ol{W}$ such that $T=0$ specializes to
$\ol{W}'\backslash \ol{W}$, then $A$ is given by $0 <
v(T) < e$ for some $e \in v(K^{\times})$. The ring of functions on the formal annulus corresponding to $A$
is $\hat{\mc{O}}_{X, x} \cong R[[T,U]]/(TU - p^e)$.  Observe that $e$ is independent of the coordinate.  It is called the 
\emph{\'{e}paisseur} of the annulus.

Suppose we have a preferred coordinate $T$ on $X$ and a semistable model $X_R'$
of $X$ whose special fiber $\ol{X}$ contains an irreducible 
component $\ol{X}_0$ corresponding to the coordinate $T$.  
If $\ol{W}$ is any irreducible component of $\ol{X}$ other than $\ol{X}_0$, then
since $\ol{X}$ is a tree of $\proj^1$'s, 
there is a unique non-repeating sequence of consecutive, intersecting components $\ol{X}_0,
\ldots, \ol{W}$.  Let $\ol{W}'$ be the
component in this sequence that intersects $\ol{W}$.  Then the set of points in
$X(\ol{K})$ that specialize to the connected component of $\ol{W}$ in $\ol{X} \backslash \ol{W}'$
is a closed $p$-adic disk $D$.  If the established preferred coordinate (equivalently, the
preferred component $\ol{X}_0$) is clear, we will abuse language and refer to the component $\ol{W}$ as 
\emph{corresponding to the disk $D$}, and vice versa.  If $U$ is a coordinate corresponding to $\ol{W}$, 
and if $U = \infty$ does not specialize to the connected component of $\ol{W}$ in $\ol{X} \backslash \ol{W}'$, 
then the ring of functions on the formal disk corresponding to $D$ is $R\{U\}$.

\section{Stable reduction}\label{Sstable}

In \S\ref{Sstable}, $R$ is a mixed characteristic $(0, p)$ complete discrete valuation ring with residue field $k$ and fraction field $K$.
We set $X \cong \proj^1_K$, and we fix a \emph{smooth} model $X_R$ of $X$.
Let $f: Y \to X$ be a $G$-Galois cover defined over $K$, with $G$ any finite
group, such that the branch points of $f$ are defined over $K$ and their specializations 
do not collide on the special fiber of $X_R$.  Assume that $f$ is branched at at least three points.
Using the stable reduction theorem for curves (\cite[Corollary 2.7]{DM:ir}), one can show that there is a 
unique minimal finite extension $K^{st}/K$ with ring of integers $R^{st}$ such that 
$f_{K^{st}} := f \times_K K^{st}$ has a \emph{stable model} $f^{st}: Y^{st} \to X^{st}$
(which we will simply call the stable model of $f$).  This model has the properties that:

\begin{itemize}
\item The special fiber $\ol{Y}$ of $Y^{st}$ is semistable. 
\item The ramification points of $f_{K^{st}}$ specialize to
\emph{distinct} smooth points of $\ol{Y}$.
\item Any genus zero irreducible component of $\ol{Y}$ contains at least three
marked points (i.e., ramification points or points of intersection with the rest
of $\ol{Y}$).
\item $G$ acts on $Y^{st}$, and $X^{st} = Y^{st}/G$.
\end{itemize}
The field $K^{st}$ is called the minimal field of definition of the stable model of $f$. 
If we are working over a finite
extension $K'/K^{st}$ with ring of integers $R'$, we will sometimes abuse
language and call 
$f^{st} \times_{R^{st}} R'$ the stable model of $f$.  

\begin{remark}
Our definition of the stable model is the definition used in
\cite{We:br}.  This differs from the definition
in \cite{Ra:sp} in that \cite{Ra:sp} allows the ramification points to coalesce
on the special fiber.  
\end{remark}

\begin{remark}\label{Rblowup}
Note that $X^{st}$ can be naturally identified with a blowup of $X \times_R R^{st}$
centered at closed points.  Furthermore, the nodes of $\ol{Y}$ lie above nodes of the special fiber $\ol{X}$
of $X^{st}$ (\cite[Lemme 6.3.5]{Ra:ab}), and $Y^{st}$ is the normalization of $X^{st}$ in $K^{st}(Y)$.
\end{remark}

If $\ol{Y}$ is smooth, the cover $f: Y \to X$ is said to have \emph{potentially
good reduction}.   If $f$ does not have
potentially good reduction, it is said to have \emph{bad reduction}.  In any case, the special fiber $\ol{f}:
\ol{Y} \to \ol{X}$ of the stable model is called the \emph{stable reduction} of $f$.  
The strict transform of the special fiber of $X_{R^{st}}$ in $\ol{X}$ (Remark \ref{Rblowup}) is called the \emph{original
component}, and will be denoted $\ol{X}_0$.  

Each $\sigma \in G_K$ acts on $\ol{Y}$ (via its action on $Y$).  This action commutes with that of $G$ and is called the 
\emph{monodromy action}.  
Then it is known (see, for instance, \cite[Proposition 2.9]{Ob:vc}) that the extension $K^{st}/K$ is
the fixed field of the group $\Gamma^{st} \leq G_K$ consisting of those $\sigma
\in G_K$ such that $\sigma$ acts trivially on $\ol{Y}$.  Thus $K^{st}$ is clearly Galois over $K$.
Since $k$ is algebraically closed, the action of $G_K$ fixes $\ol{X}_0$ pointwise.

\subsection{The graph of the stable reduction}\label{Sgraph}
As in \cite{We:br}, we construct the (unordered) dual graph $\mc{G}$ of the stable reduction of $\ol{X}$. 
An \emph{unordered graph} $\mc{G}$ consists of a set of \emph{vertices} $V(\mc{G})$ and a nonempty set of \emph{edges} $E(\mc{G})$.  
Each edge has a \emph{source vertex} $s(e)$ and a \emph{target vertex} $t(e)$.  Each edge has an \emph{opposite
edge} $\ol{e}$, such that $s(e) = t(\ol{e})$ and $t(e) = s(\ol{e})$.  Also, $\ol{\ol{e}} = e$.

Given $f$, $\ol{f}$, $\ol{Y}$, and $\ol{X}$ as above, we construct two unordered graphs $\mc{G}$ and
$\mc{G}'$.  In our construction, $\mc{G}$ has a vertex $v$ for each irreducible
component of $\ol{X}$ and an edge $e$ for each ordered triple $(\ol{x}, \ol{W}', \ol{W}'')$, 
where $\ol{W}'$ and $\ol{W}''$ are irreducible components of $\ol{X}$ whose intersection is $\ol{x}$.  If $e$
corresponds to $(\ol{x}, \ol{W}', \ol{W}'')$, then
$s(e)$ is the vertex corresponding to $\ol{W}'$ and $t(e)$ is the vertex corresponding to
$\ol{W}''$.  The opposite edge of $e$ corresponds to $(\ol{x}, \ol{W}'', \ol{W}')$.
We denote by $\mc{G}'$ the \emph{augmented} graph of $\mc{G}$ constructed as follows: consider the set
$B_{\text{wild}}$ of branch points of $f$ with branching index divisible by $p$.  
For each $x \in B_{\text{wild}}$, we know that $x$ specializes to 
a unique irreducible component $\ol{W}_x$ of $\ol{X}$, corresponding to a vertex $A_x$ of $\mc{G}$.  
Then $V(\mc{G}')$ consists of the elements of $V(\mc{G})$ 
with an additional vertex $V_x$ for each $x \in B_{\text{wild}}$.  Also, $E(\mc{G}')$ consists of the elements of 
$E(\mc{G})$ with two additional opposite edges for each $x \in B_{\text{wild}}$, 
one with source $V_x$ and target $A_x$, and one with source $A_x$ and
target $V_x$.  We write $v_0$ for the vertex corresponding to the original component $\ol{X}_0$.  

If $v, w \in V(\mc{G}')$, then a \emph{path} from $v$ to $w$ is a sequence of nonrepeating vertices $\{v_i\}_{i=0}^{n}$ and edges $\{e_i\}_{i=0}^{n-1}$
such that $v = v_0$, $w = v_n$, $s(e_i) = v_i$, and $t(e_i) = v_{i+1}$.  Let $u \in V(\mc{G}')$ correspond to the original component
$\ol{X}_0$.   We partially order the vertices of $\mc{G}'$ such that $v \preceq w$ if there is a path from $u$ to $w$ passing through $v$.
The set of irreducible components of $\ol{X}$ inherits the partial order $\preceq$.  Furthermore, if $\ol{x}_1$ and $\ol{x}_2$ are points
of $\ol{X}$, we say that $\ol{x}_2$ \emph{lies outward from} $\ol{x}_1$ if $\ol{x}_1 \ne \ol{x}_2$ and 
there are irreducible components $\ol{X}_1 \prec \ol{X}_2$ of $\ol{X}$ such that $\ol{x}_1 \in \ol{X}_1$ and $\ol{x}_2 \in \ol{X}_2$.
Lastly, an irreducible component $\ol{W}$ of $\ol{X}$ \emph{lies outward from} a point $\ol{x} \in \ol{X}$ if there is an irreducible component 
$\ol{W}' \prec \ol{W}$ such that $\ol{x} \in \ol{W}'$.

\subsection{Inertia Groups of the Stable Reduction}

Recall that $G$ acts on $\ol{Y}$.  By \cite[Lemme 6.3.3]{Ra:ab}, we know that
the inertia groups of the action of $G$ on $\ol{Y}$ at generic points of $\ol{Y}$ are $p$-groups.  
Also, at each node of $\ol{Y}$, the inertia group is an extension of a cyclic,
prime-to-$p$ order group by a $p$-group generated by
the inertia groups of the generic points of the crossing components.
If $\ol{V}$ is an irreducible component of $\ol{Y}$, we will always write $I_{\ol{V}} \leq G$ for the inertia group of
the generic point of $\ol{V}$, and $D_{\ol{V}} \leq G$ for the decomposition group.

For the rest of this subsection, assume $G$ has a \emph{cyclic} $p$-Sylow subgroup.
In this case, the inertia groups above a generic point of an
irreducible component $\ol{W} \subset 
\ol{X}$ are conjugate cyclic groups of $p$-power order.  If they are of order
$p^i$, we call $\ol{W}$ a
\emph{$p^i$-component}.  If $i = 0$, we call $\ol{W}$ an \emph{\'{e}tale
component}, and if $i > 0$, we call $\ol{W}$ an
\emph{inseparable component}. 

As in \cite{Ra:sp}, we call irreducible component $\ol{W} \subseteq \ol{X}$ a \emph{tail} if it is not the original component and intersects exactly 
one other irreducible component of $\ol{X}$.
Otherwise, it is called an \emph{interior component}.  
A tail of $\ol{X}$ is called \emph{primitive} if it contains a branch
point other than the point at which it intersects the rest of $\ol{X}$.  
Otherwise it is called \emph{new}.  This follows \cite{We:br}.  
An inseparable tail that is a $p^i$-component will also be called a
\emph{$p^i$-tail}.  Thus one can speak of, for instance, ``new $p^i$-tails" or ``primitive \'{e}tale tails."

\begin{lemma}[\cite{Ob:vc}, Proposition 2.13] \label{Lcorrectspec}
If $x \in X$ is branched of index $p^as$, where $p \nmid s$, then $x$
specializes to a $p^a$-component.
\end{lemma}

\begin{lemma}[\cite{Ra:sp}, Proposition 2.4.8]\label{Letaletail}
If $f$ has bad reduction and $\ol{W}$ is an \'{e}tale component of $\ol{X}$, then $\ol{W}$ is a tail.
\end{lemma}
\smallskip
\begin{lemma}[\cite{Ob:vc}, Proposition 2.16, see also \cite{Ra:sp}, Remarque 3.1.8] \label{Ltailetale}
If $f$ has bad reduction and $\ol{W}$ is a $p^i$-tail of $\ol{X}$, 
then the component $\ol{W}'$ that intersects $\ol{W}$ is a $p^j$-component with
$j > i$.
\end{lemma}

\begin{definition}\label{Draminvariant} (cf.\ \cite{Ob:vc}, Definition 2.18)
Consider a component $\ol{X}_b \ne \ol{X}_0$ of $\ol{X}$.  Let $\ol{W}$ be the unique component of $\ol{X}$ such that
$\ol{W} \prec \ol{X}_b$ and $\ol{W}$ intersects $\ol{X}_b$, say at $\ol{x}_b$.
Suppose that $\ol{W}$ is a $p^j$-component and $\ol{X}_b$ is a $p^i$-component, $i < j$. 
Let $\ol{Y}_b$ be a component of $\ol{Y}$ lying above $\ol{X}_b$, and let $\ol{y}_b$ be a point
lying above $\ol{x}_b$.  Then the \emph{effective ramification invariant} $\sigma_b$ of $\ol{X}_b$ is defined as follows:
If $\ol{X}_b$ is an \'{e}tale component, then $\sigma_b$ is the conductor of higher ramification
for the extension $\hat{\mc{O}}_{\ol{Y}_b, \ol{y}_b}/\hat{\mc{O}}_{\ol{X}_b, \ol{x}_b}$.  If
$\ol{X}_b$ is a $p^i$-component ($i > 0$), then the extension $\hat{\mc{O}}_{\ol{Y}_b, \ol{y}_b}/\hat{\mc{O}}_{\ol{X}_b, \ol{x}_b}$ can be 
factored as $\hat{\mc{O}}_{\ol{X}_b, \ol{x}_b} \stackrel{\alpha}{\hookrightarrow} S \stackrel{\beta}{\hookrightarrow} 
\hat{\mc{O}}_{\ol{Y}_b, \ol{y}_b}$, where $\alpha$ is Galois and $\beta$ is purely inseparable of degree $p^i$.
Then $\sigma_b$ is the conductor of higher ramification for the extension $S/\hat{\mc{O}}_{\ol{X}_b, \ol{x}_b}$.
\end{definition}

\begin{remark}\label{Rraminvariant}
By Lemma \ref{Ltailetale}, the effective ramification invariant is defined for every tail.
\end{remark}

\begin{lemma}[\cite{Ob:vc}, Lemma 2.20]\label{Lramdenominator}
The effective ramification invariants $\sigma_b$ lie in $\frac{1}{m_G}\ints$.
\end{lemma}

\subsection{Vanishing cycles formula}\label{Svancycles}
Assume the notation of \S\ref{Sstable}.  The \emph{vanishing cycles formula} stated below will be used repeatedly:

\begin{theorem}[Vanishing cycles formula, cf., \cite{Ob:vc}, Theorem 3.14, Corollary 3.15]\label{Tvancycles}
Let $f: Y \to X \cong \proj^1$ be a three-point $G$-Galois cover with bad reduction, where $G$ has a cyclic
$p$-Sylow subgroup.  Let $B_{\text{new}}$ be an indexing set for the new \'{e}tale tails and let
$B_{\text{prim}}$ be an indexing set for the primitive \'{e}tale tails.  Let $\sigma_b$ be the ramification invariant in Definition 
\ref{Draminvariant}.
Then we have the formula 

\begin{equation}\label{Evancycles} 
1 = \sum_{b \in B_{\text{new}}} (\sigma_b - 1) + \sum_{b \in B_{\text{prim}}} \sigma_b.
\end{equation}
\end{theorem}

\begin{corollary}\label{Chalfsigma}
If $m_G=2$, and if $f$ has bad reduction, then there are at most two \'{e}tale
tails.  Furthermore, for any \'{e}tale tail $\ol{X}_b$, $\sigma_b \in \frac{1}{2}\ints$ (see \S\ref{Sstable}).
\end{corollary}

\begin{proof}
By Lemma \cite[Lemma 4.2 (i)]{Ob:vc}, each term on the right hand side of (\ref{Evancycles}) is at least 
$\frac{1}{m_G} = \frac{1}{2}$, so both parts of the corollary follow immediately.  
\end{proof}

\section{Deformation data}\label{Sdefdata}

Deformation data arise naturally from the stable reduction of covers.  Much information is lost when we pass from the stable model of a cover to its stable reduction, 
and deformation data provide a way to retain some of this information.  This process is described in detail in
\cite[\S3.2]{Ob:vc}, and we recall some facts here.

\subsection{Generalities}\label{Sgeneralities}
Let $\ol{W}$ be any connected smooth proper curve over $k$.  
Let $H$ be a finite group and $\chi$ a 1-dimensional character 
$H \to \FF_p^{\times}.$  A \emph{deformation datum} over
$\ol{W}$ of type $(H, \chi)$ is an ordered pair $(\ol{V}, \omega)$ such that: $\ol{V} \to \ol{W}$ is
an $H$-cover;
$\omega$ is a meromorphic differential form on $\ol{V}$ that is either logarithmic or
exact (i.e., $\omega = du/u$ or $du$ for 
$u \in k(\ol{V})$); and $\eta^*\omega = \chi(\eta)\omega$ for all $\eta \in H$.  If $\omega$
is logarithmic (resp.\ exact), the deformation datum is called
multiplicative (resp.\ additive).  When $\ol{V}$ is understood, we will sometimes
speak of the deformation datum $\omega$.  

If $(\ol{V}, \omega)$ is a deformation datum, and $w \in \ol{W}$ is a closed point, we
define $m_w$ to be the order of the 
prime-to-$p$ part of the ramification index of $\ol{V} \to \ol{W}$ at $w$.  Define $h_w$
to be $\ord_v(\omega) + 1$, where $v \in
\ol{V}$ is any point which maps to $w \in \ol{W}$.  This is well-defined because $\eta^*\omega$ is a nonzero scalar multiple of $\omega$ for
$\eta \in H$.

Lastly, define $\sigma_w = h_w/m_w$.  We call $w$ a \emph{critical point} of the
deformation datum $(\ol{V}, \omega)$ if 
$(h_w, m_w) \ne (1, 1)$.  Note that every deformation datum contains only a
finite number of critical points.  The
ordered pair $(h_w, m_w)$ is called the \emph{signature} of $(\ol{V}, \omega)$ (or of
$\omega$, if $\ol{V}$ is understood) at $w$, and
$\sigma_w$ is called the \emph{invariant} of the deformation datum at $w$.

\subsection{Deformation data arising from stable reduction.}\label{Sdefdatastable}
We use the notation of \S\ref{Sstable}.  Assume that a $p$-Sylow subgroup of $G$ is cyclic.  
For each irreducible component of $\ol{Y}$ lying above a
$p^r$-component of $\ol{X}$ with $r > 0$, we construct $r$ different deformation data.  The details of this construction are given in 
\cite[Construction 3.4]{Ob:vc}, and we do not give them here.  Rather, we recall the important properties.

Suppose $\ol{V}$ is an irreducible component of $\ol{Y}$ with 
nontrivial generic inertia group $I_{\ol{V}} \cong \ints/p^r \subset G$.  If $\ol{V}'$ is the smooth projective model of the function field
of $k(\ol{V})^{p^r}$, then \cite[Construction 3.4]{Ob:vc} constructs $r$ meromorphic differential forms $\omega_1, \ldots, \omega_r$ 
on $\ol{V}'$ (well defined either up to scalar multiplication by $k^{\times}$ or by $\FF_p^{\times}$, depending on whether the 
differential form is exact or logarithmic).  
Furthermore, if $H = D_{\ol{V}}/I_{\ol{V}}$, then $H$ acts faithfully on $\ol{V}'$, and $\ol{W} \cong \ol{V}/H$.
It is shown in \cite[Construction 3.4]{Ob:vc} that
$(\ol{V}', \omega_i)$ is in fact a deformation datum of type $(H, \chi)$ over $\ol{W}$ for $1 \leq i \leq r$, where $\chi$ 
is given by the conjugation action of $H$ on $I_{\ol{V}}$.  The invariant of $\sigma_i$ at a point $w \in 
\ol{W}$ will be denoted $\sigma_{i,w}$.  Since $f$ is Galois, these invariants do not depend on which component 
$\ol{V}$ above $\ol{W}$ is chosen. 
The differential forms $\omega_1, \ldots, \omega_r$ correspond, in some sense, to the successive degree $p$ extensions building a tower 
between $\ol{V}'$ and $\ol{V}$.  For this reason, we will sometimes call the deformation datum $(\ol{V}', \omega_1)$ the \emph{bottom 
deformation datum} for $\ol{V}$.  

Furthermore, for $1 \leq i \leq r$, we associate a rational number $\delta_{i}$ (see \cite[\S5.2]{Ob:fm}).  If 
$\omega_i$ is multiplicative, then $\delta_{i} = 1$.  Otherwise, $0 < \delta_{i} < 1$.  
The \emph{effective different} $\delta^{\eff}_{\ol{W}}$ 
above $\ol{W}$ is defined by $$\delta^{\eff}_{\ol{W}} = \left( \sum_{i=1}^{r-1} \delta_{i} \right) + \frac{p}
{p-1}\delta_{r}.$$  By convention, if $\ol{W}$ is an \'{e}tale component, we set $\delta^{\eff}_{\ol{W}} = 0$.

The following lemma will be very important in the main proof.
\begin{lemma}[\cite{Ob:vc}, Lemma 3.5, cf.\ \cite{We:br}, Proposition 1.7]\label{Lcritical}
Say $(\ol{V}', \omega)$ is a deformation datum arising from the stable reduction
of a cover, and let $\ol{W}$ be the component of $\ol{X}$ lying under
$\ol{V}'$.  Then a critical point $w$ of the
deformation datum on $\ol{W}$ is either a singular point of $\ol{X}$ or the
specialization of a branch point of $Y \to
X$ with ramification index divisible by $p$.  In the first case, $\sigma_w \ne
0$, and in the second case, $\sigma_w = 0$
and $\omega$ is logarithmic.
\end{lemma}

Recall that $\mc{G}'$ is the augmented dual graph of $\ol{X}$ (\S\ref{Sgraph}).  To each $e \in E(\mc{G}')$ we will associate 
the \emph{effective invariant} $\sigma^{\eff}_b$, and to each vertex of $\mc{G}$ we will associate the 
\emph{effective different} $\delta^{\eff}_v$.  

\begin{definition}[cf.\ \cite{Ob:vc}, Definition 3.10]\label{Dsigmaeff} 
\rm Let $e \in E(\mc{G}')$.
\begin{enumerate}
\item Suppose $e$ corresponds to the triplet $(w, \ol{W}, \ol{W}')$, where $\ol{W}$ is a $p^r$-component and $\ol{W}'$ is a $p^{r'}$-component
with $r \geq r'$.  Then $r \geq 1$ by Lemma \ref{Letaletail}. 
Let $\omega_i$, $1 \leq i \leq r$, be the deformation data above $\ol{W}$.  Then 
$$\sigma^{\eff}_e := \left( \sum_{i=1}^{r-1} \frac{p-1}{p^{i}}\sigma_{i,w} \right) + \frac{1}{p^{r-1}}\sigma_{r,w}.$$
Note that this is a weighted average of the $\sigma_{i,w}$'s.
\item If $s(e)$ corresponds to a $p^r$-component and $t(e)$ corresponds to a $p^{r'}$-component with $r < r'$, then
$\sigma^{\eff}_e := -\sigma^{\eff}_{\ol{e}}$.  
\item If either $s(e)$ or $t(e)$ is in $V(\mc{G}') \backslash V(\mc{G})$, then $\sigma^{\eff}_e := 0$.  
\item For all $v \in V(\mc{G})$, define $\delta^{\eff}_v = \delta^{\eff}_{\ol{W}}$, where $v$ corresponds to $\ol{W}$. 
\item For all $e \in E(\mc{G})$, define $\epsilon_e$ to be the \'{e}paisseur of the formal annulus corresponding to $e$ 
(\S\ref{Sdisks}).
\end{enumerate}
\end{definition}

\begin{remark}\label{Roort}
In the paper \cite{OW:ce}, similar ideas of deformation data are used, but the notation and method of calculation is somewhat different.
Suppose $\ol{V}$ is an irreducible component of $\ol{Y}$ as in this section such that $D_{\ol{V}} = I_{\ol{V}} = \ints/p^r$, and let $\ol{W}$ be the 
component of $\ol{X}$ lying below it.  If $\eta_{\ol{V}}$, $\eta_{\ol{W}}$ are the generic points of $\ol{V}$, $\ol{W}$, then 
$\hat{\mc{O}}_{Y^{st}, \ol{V}}/\hat{\mc{O}}_{X^{st}, \ol{W}}$ is a $\ints/p^r$-extension of complete discrete valuation rings, corresponding to
a character $\chi$ in the language of \cite[\S5]{OW:ce}.  Our $\delta^{\eff}_{\ol{W}}$ is equal to $\text{sw}(\chi)$ in \cite[\S5.3]{OW:ce}, 
and if $e \in E(\mc{G}')$ corresponds to $(w, \ol{W}, \ol{W}')$ as in Definition \ref{Dsigmaeff}, 
then our $\sigma^{\eff}_e$ is equal to $\ord_w(\text{dsw}(\chi)) + 1$ in
\cite[\S5.3]{OW:ce}.
\end{remark}

\begin{lemma}[\cite{Ob:vc}, Lemma 3.11 (i), (iii), \cite{Ob:fm}, Lemma 5.10]\label{Lsigmaeffcompatibility}
\ 
\begin{enumerate}
\item For any $e \in E(\mc{G}')$, we have $\sigma^{\eff}_e = -\sigma^{\eff}_{\ol{e}}$.
\item  If $t(e)$ corresponds to an \'{e}tale tail $\ol{X}_b$, then $\sigma^{\eff}_e = \sigma_b$.
\item If $e \in E(\mc{G})$, then $\delta^{\eff}_{s(e)} - \delta^{\eff}_{t(e)} = \sigma^{\eff}_e \epsilon_e$.
\end{enumerate}
\end{lemma}

The following lemma is very important for \S\ref{Stauequals1}:
\begin{lemma}[\cite{Ob:fm}, Lemma 5.7]\label{Lsigmaeff}
Let $e \in E(\mc{G})$ such that $s(e) \prec t(e)$.  Let $\ol{w} \in \ol{X}$ be the point corresponding to $e$.
Let $\Pi_e$ be the set of branch points of $f$ with branching index divisible by $p$ that specialize outward from $\ol{w}$.  
Let $B_e$ index the set of \'{e}tale tails $\ol{X}_b$ lying outward from $\ol{w}$.  Then 
$$\sigma^{\eff}_e - 1 = \sum_{b \in B_e} (\sigma_b - 1) - |\Pi_e|.$$
\end{lemma}

\begin{corollary}[Monotonicity of the effective different]\label{Cdifferentmonotone}
If $v, v' \in V(\mc{G})$, and $v \prec v'$, then $\delta^{\eff}_v \geq \delta^{\eff}_{v'}$.
\end{corollary}

\begin{proof}
Clearly we may assume that $v$ and $v'$ are adjacent, i.e., there is an edge $e$ such that $s(e) = v$ and $t(e) = v'$.  
Since the branch points of $f$ are assumed not to collide on the special fiber of our original smooth model $X_R$, there is at 
most one branch point of 
$f$ specializing outward from the node $\ol{x}_e$ corresponding to $e$.  That is, there is at most either one primitive tail or one 
branch point of index divisible 
by $p$ lying outward from $\ol{x}_e$.  Since $\sigma_b > 1$ for all new tails $\ol{X}_b$ (\cite[Lemma 4.2 (i)]{Ob:vc}), we see by 
Lemma \ref{Lsigmaeff} 
that $\sigma^{\eff}_e \geq 0$.  We conclude using Lemma \ref{Lsigmaeffcompatibility} (3).
\end{proof}

\section{Monotonicity} \label{Scomb}
We maintain the assumptions and notation of \S\ref{Sstable}, along with the assumption that a $p$-Sylow subgroup of $G$ is cyclic 
of order $p^n$.

\begin{lemma}\label{Lwildspec}
Let $x$ be a branch point of $f$ with branching index exactly divisible by $p^r$.  
If $x$ specializes to an irreducible component $\ol{W}$ of $\ol{X}$, then either $\ol{W}$ is the original
component, or the unique component $\ol{W}'$ such that $\ol{W}' \prec \ol{W}$ and $\ol{W}'$ intersects $\ol{W}$ is a $p^{s}$-component, for some 
$s > r$.
\end{lemma}

\begin{proof}
By Lemma \ref{Lcritical} the deformation data above $\ol{W}$ are all multiplicative, and by \cite[Proposition 5.2]{Ob:fm} they
are all identical.
Thus $\delta^{\eff}_{\ol{W}} = r + \frac{1}{p-1}$, as $\delta_{\omega_i} = 1$ for all $\omega_i$ above $\ol{W}$.  Assume $\ol{W}$ 
is not the original
component.  Then, by Corollary \ref{Cdifferentmonotone}, $\delta^{\eff}_{\ol{W}'} \geq r + \frac{1}{p-1}$.  This is impossible 
unless $\ol{W}$ is a $p^s$-component with $s \geq r$.  If $s = r$, we must have $\delta^{\eff}_{\ol{W}'} = \delta^{\eff}_{\ol{W}} 
= r + \frac{1}{p-1}$.  
By Corollary \ref{Cdifferentmonotone}, 
$\sigma^{\eff}_e = 0$ for $e$ either edge corresponding to $\{w\} = \ol{W} \cap \ol{W}'$.  Since the deformation data above 
$\ol{W}$ are identical and 
$\sigma_e^{\eff}$ is a weighted average of invariants, we have $\sigma_{i,w} = 0$ for all $\omega_i$ above $\ol{W}$.  But this contradicts Lemma 
\ref{Lcritical}.  
So $s > r$.
\end{proof}

\begin{lemma}\label{Lnonintegral}
Let $\ol{x}$ be a singular point of $\ol{X}$ such that there are no \'{e}tale tails $\ol{X}_b$ lying outward from $\ol{x}$.
If $I_{\ol{x}} \leq G$ is an inertia group above $\ol{x}$, then $m_{I_{\ol{x}}} = 1$.
\end{lemma}

\begin{proof}
We first claim that, given an inseparable component $\ol{W} \subseteq \ol{X}$, there cannot be exactly one point $\ol{w} \in \ol{W}$ such
that $m_{I_{\ol{w}}} > 1$.  To prove the claim, let $\ol{V} \in \ol{Y}$ be a component above $\ol{W}$.  Since $\ol{W}$ is inseparable,
it follows that $D_{\ol{V}}$ has a normal subgroup of order $p$ (namely, the order $p$ subgroup of $I_{\ol{V}}$).  By \cite[Corollary 2.4]{Ob:vc},
$D_{\ol{V}}$ has a quotient of the form $\ints/p^{\nu} \rtimes \ints/m_{D_{\ol{V}}}$, where the action of $\ints/m_{D_{\ol{V}}}$ on $\ints/p^{\nu}$ 
is faithful, and $\ints/p^{\nu}$ is a $p$-Sylow subgroup of $D_{\ol{V}}$.  If $m_{D_{\ol{V}}} = 1$, then $m_{I_{\ol{w}}} = 1$ for all $\ol{w} \in \ol{W}$,
so assume $m_{D_{\ol{V}}} > 1$.  Then $\ol{V} \to \ol{W}$ has a quotient $\ints/m_{D_{\ol{V}}}$-cover $\ol{V}' \to \ol{W}$, which
must be branched at at least two points, say $\ol{w}_1$ and $\ol{w}_2$.  Then
$I_{\ol{w}_1}$ and $I_{\ol{w}_2}$ are non-abelian subgroups of $\ints/p^{\nu} \rtimes \ints/m_{D_{\ol{V}}}$, meaning that
$m_{I_{\ol{w}_1}}$ and $m_{I_{\ol{w}_2}}$ are greater than 1.  This proves the claim.

Now, if $\ol{W}$ is an inseparable tail, then there is only one possible point $\ol{w} \in \ol{W}$ where $m_{I_{\ol{w}}}$ might not be $1$ (the intersection
point with the rest of $\ol{X}$), and the claim shows that we do, in fact, have $m_{I_{\ol{w}}} = 1$.  The lemma then follows by inward induction.
\end{proof}

\begin{definition}\label{Dmonotonic}
We call the stable reduction $\ol{f}$ of $f$ \emph{monotonic} 
if for every $\ol{W} \preceq \ol{W}'$, the inertia group of $\ol{W}'$ is contained in the inertia group of $\ol{W}$.  
In other words, the stable reduction is monotonic if the generic inertia does not increase as we move outward from $\ol{X}_0$ along $\ol{X}$.  
\end{definition}

In the situation of Theorem \ref{Tmain}, all covers are monotonic:
\begin{prop} \label{Pmonotonic}
If $f$ is a three-point $G$-cover of $\proj^1$, where $G$ has a cyclic $p$-Sylow subgroup of order $p^n$, and $m_G = 2$, then 
$\ol{f}$ is monotonic.
\end{prop}

\begin{proof} 
Suppose $f$ is not monotonic.  Then there exist $j \leq n$ and 
a set $\Sigma$ of $p^j$-components of $\ol{X}$ with the following properties: $\ol{X}_0 \notin \Sigma$; 
the union $\ol{U}$ of the components in $\Sigma$ (viewed as a closed subset of $\ol{X}$) is connected; and each irreducible 
component of 
$\ol{X}$ that intersects $\ol{U}$ but is not in $\Sigma$ is a $p^i$-component, $i < j$ (think of $\ol{U}$ as being a ``plateau" 
for inertia).  In particular, $j > 0$.
Note that, by Lemmas \ref{Lcorrectspec} and \ref{Lwildspec}, no branch point of $f$ specializes to $\ol{U}$ (this is the only place where we
use $\ol{X}_0 \notin \Sigma$).

Recall that $\mc{G}$ (resp.\ $\mc{G}'$) is the dual graph (resp.\ augmented dual graph) of $\ol{X}$ (\S\ref{Sgraph}).
Let $\Phi$ be the set of all edges $e \in E(\mc{G}')$ such that $s(e)$ corresponds to a component in $\Sigma$, and let $\Phi'$ be 
the set of those $e \in \Phi$ such that $t(e)$ does not correspond to a component in $\Sigma$ (in particular, the inertia of $t(e)$ has order
smaller than $p^j$).  If $e \in \Phi$ 
corresponds to a point $\ol{x} \in \ol{X}$, then we write $\sigma^{\bottom}_e$ to mean the invariant of any bottom 
deformation datum above the component corresponding to $s(e)$ at $\ol{x}$ (this is equivalent to $\sigma^{\eff, j-1}_e$ in the 
language of \cite[Definition 3.10]{Ob:vc}).  By \cite[Lemma 3.11 (i)]{Ob:vc} we have $\sigma^{\bottom}_e = - \sigma^{\bottom}_{\ol{e}}$ for
all $e \in \Phi \backslash \Phi'$.  By \cite[Lemma 3.12]{Ob:vc} (and the fact that no branch point of $f$ specializes to $\ol{U}$) 
we have, for any vertex $v$ representing a component in $\Sigma$, that 
$$\sum_{\substack{e \in E(\mc{G'}) \\ s(e) = v}} (\sigma^{\bottom}_e - 1) = \sum_{\substack{e \in E(\mc{G}) \\ s(e) = v}} (\sigma^{\bottom}_e - 1)= -2.$$  
A simple induction argument (cf.\ the proof of \cite[Corollary 1.11]{We:br} or 
\cite[Theorem 3.14]{Ob:vc}) shows that 
\begin{equation}\label{Emonotonic}
\sum_{e \in \Phi'} (\sigma^{\bottom}_e - 1) = -2.
\end{equation}

Suppose $e \in \Phi'$ corresponds to a point $\ol{x} \in \ol{X}$, and let $m_e = m_{I_{\ol{x}}}$, where $I_{\ol{x}}$ is an inertia 
group of $\ol{f}$ above $\ol{X}$.  By Lemma \ref{Lramdenominator}, we have $\sigma^{\bottom}_e \in \frac{1}{m_e}\ints$. 
Since $t(e)$ has smaller inertia order than $s(e)$, \cite[Lemma 3.11 (ii)]{Ob:vc}
shows that $\sigma^{\bottom}_e > 0$.  Clearly, $m_e \in \{1, 2\}$, and $m_e = 1$ if no \'{e}tale tails of $\ol{X}$
lie outward from $\ol{x}$ (Lemma \ref{Lnonintegral}).  

In particular, $\sigma^{\bottom}_e - 1 \geq -\frac{1}{2}$, and 
$\sigma^{\bottom}_e - 1 \geq 0$ if there are no \'{e}tale tails lying outward from $\ol{x}$.
Since $m_G = 2$, Corollary \ref{Chalfsigma} shows that there can be at most two \'{e}tale tails.  
Thus there are at most three $e \in \Phi'$ such that  
an \'{e}tale tail lies outward from the node corresponding to $e$ (at worst, the outermost $e$ preceding each 
of the \'{e}tale tails and the innermost $e$).  
So $\sigma^{\bottom}_e - 1 \geq 0$ for all but at most three edges $e \in \Phi'$.  This contradicts 
(\ref{Emonotonic}), proving the proposition.
\end{proof}

\section{The auxiliary cover}\label{Saux}
We maintain the notation of \S\ref{Sstable}.  Assume that $f: Y \to X$ is a $G$-cover 
defined over $K$ as in \S\ref{Sstable} with bad reduction, so that $\ol{X}$ is not just the original component
($G$ need not have a cyclic $p$-Sylow group).
Following \cite[\S 3.2]{Ra:sp}, 
we can construct an \emph{auxiliary cover} $f^{aux}: Y^{aux} \to X$ with (modified)
stable model 
$(f^{aux})^{st}: (Y^{aux})^{st} \to X^{st}$ and (modified) stable reduction
$\ol{f}^{aux}: \ol{Y}^{aux} \to \ol{X}$, 
defined over some finite extension $R'$ of $R$.
We will explain what ``modified" means in a remark following the construction.
The construction is almost entirely the same as in \cite[\S3.2]{Ra:sp}, and we
will not repeat the details.  
Instead, we give an overview, and we mention where our construction differs from
Raynaud's.

Let $B_{\text{\'{e}t}}$ index the \'{e}tale tails of $\ol{X}$.  
Subdividing $B_{\text{\'{e}t}}$, we index the set of primitive tails by
$B_{\text{prim}}$ and the set of new tails by 
$B_{\text{new}}$. 
We will write $\ol{X}_b$ for the tail indexed by $b \in B_{\text{\'{e}t}}$.

The construction proceeds as follows:  From $\ol{Y}$ remove all of the
components that lie above the \emph{\'{e}tale}
tails of $\ol{X}$ (as opposed to \emph{all} the tails---this is the only thing
that needs to be done differently than in \cite{Ra:sp}, where all tails are \'{e}tale).
Now, what remains of $\ol{Y}$ is possibly disconnected.  
We choose one connected component, and call it $\ol{V}$.  

For each $b \in B_{\text{prim}}$, let $a_b$ be the branch point of
$f$ specializing to $\ol{X}_b$, let 
$\ol{x}_b$ be the point where $\ol{X}_b$ intersects 
the rest of $\ol{X}$, and let $p^rm_b$ be the index of ramification above $\ol{X}_b$ at $\ol{x}_b$,
with $m_b$ prime-to-$p$.  Then $\ol{X}_b$ intersects a $p^r$-component.
At each point $\ol{v}_b$ of $\ol{V}$ above $\ol{x}_b$, we attach to $\ol{V}$ a
\emph{Katz-Gabber} cover of $\ol{X}_b$ (cf.\ \cite[Theorem 1.4.1]{Ka:lg}, \cite[Th\'{e}or\`{e}me 3.2.1]{Ra:sp}), 
branched of order $m_b$ (with inertia groups isomorphic to $\ints/m_b$) at the specialization $\ol{a}_b$ of $a_b$ and of 
order $p^{r}m_b$ (with inertia groups isomorphic to $\ints/p^r \rtimes \ints/m_b$) at $\ol{x}_b$.
We choose our Katz-Gabber cover so that above the complete local ring of $\ol{x}_b$
on $\ol{X}_b$, it is isomorphic to the original cover.  
It is the composition of a cyclic cover of order $m_b$ branched at $\ol{x}_b$ and $\ol{a}_b$ with a cyclic cover of order
$p^r$ branched at one point.  Note that if $m_b = 1$, we have eliminated the branch point $\ol{a}_b$ of the original cover.

For each $b \in B_{\text{new}}$, we carry out the same procedure, except that we
\emph{introduce} an (arbitrary) branch point 
$\ol{a}_b \ne \ol{x}_b$ of ramification index $m_b$ on the new tail $\ol{X}_b$. 

Let $\ol{f}^{aux}: \ol{Y}^{aux} \to \ol{X}$ be the cover of $k$-schemes we have just
constructed.  Let $G^{aux} \leq G$ be the decomposition group of $\ol{V}$.
As in \cite[\S3.2]{Ra:sp}, one shows that, after a possible finite extension $R'$ of $R$, 
we can lift $\ol{f}^{aux}$ to a map $(f^{aux})^{st}: (Y^{aux})^{st} \to X^{st}$ over $R'$, satisfying the following 
properties: 
\begin{enumerate}
\item Above an \'{e}tale neighborhood of the union of those components of $\ol{X}$ other than \'{e}tale tails, 
the cover $f^{st}: Y^{st} \to X^{st}$ is $\Ind_{G^{aux}}^G (f^{aux})^{st}$ (see \S\ref{Snotations}).
\item The generic fiber $f^{aux}: Y^{aux} \to X$ is a $G^{aux}$-cover branched exactly at the branch points of $f$ and at a new 
point $a_b$ of index $m_b$ for each new tail $b \in B_{\text{new}}$ (unless $m_b = 1$, as noted above).  Each $a_b$ specializes to the corresponding 
branch point $\ol{a}_b$ introduced above.  
\end{enumerate}
Keep in mind that there is some choice here in how to pick the new branch
points---for a new tail $\ol{X}_b$, depending on the choice of $\ol{a}_b$, 
we can choose $a_b$ to be \emph{any} point of $X$ that specializes
to $\ol{X}_b \backslash \ol{x}_b$.  The set of such points forms a closed $p$-adic disk (\S\ref{Sdna})

The generic fiber $f^{aux}$ of $(f^{aux})^{st}$ is called the \emph{auxiliary cover}, and $(f^{aux})^{st}$ is 
called the \emph{modified stable model} of the auxiliary cover.  The special fiber $\ol{f}^{aux}$ is called the \emph{modified 
stable reduction} of the auxiliary cover.

\begin{remark}\label{Rmodifiedstable}
Usually, the stable model of $f^{aux}$ is same as the modified stable model $(f^{aux})^{st}$.
However, it may happen that the stable model of $f^{aux}$ is a contraction of $(f^{aux})^{st}$
(or that it is not even defined, as we may have eliminated a
branch point by passing to the auxiliary cover).  This happens only if $\ol{X}$ has a primitive tail $\ol{X}_b$ for which
$m_b = 1$, and for which the Katz-Gabber cover inserted above $\ol{X}_b$ has genus zero.  Then this tail, and
possibly some components inward, would
be contracted in the stable model of $f^{aux}$.  We use the term
\emph{modified} stable model to mean that we do
not perform this contraction.  Alternatively, we can think of $(f^{aux})^{st}$ as the stable model of $f^{aux}$, if we count 
specializations of all points of $Y^{aux}$ above branch points of $f$ (as opposed to $f^{aux}$) as marked points.
\end{remark}

If we are interested in understanding the field of moduli of a $G$-cover (or more generally,
the minimal field of definition of the stable model), it
is in some sense good enough to understand the auxiliary cover, as the following
lemma shows.

\begin{lemma}\label{Laux2orig}
If the modified stable model
$(f^{aux})^{st}: (Y^{aux})^{st} \to X^{st}$ 
of the auxiliary cover $f^{aux}$ is defined over a Galois extension $K^{aux}/K_0$, then the stable model $f^{st}$ of
$f$ can also be defined over $K^{aux}$.
\end{lemma}

\begin{proof} (cf.\ \cite{We:br}, Theorem 4.5)
Take $\sigma \in \Gamma^{aux}$, the absolute Galois group of $K^{aux}$.  We must
show that $f^{\sigma} \cong f$ and that
$\sigma$ acts trivially on the stable reduction $\ol{f}: \ol{Y} \to \ol{X}$ of
$f$.  
Let $\hat{f}: \hat{Y} \to \hat{X}$ be the formal completion of $f^{st}$ at the
special fiber and let $\hat{f}^{aux}:
\hat{Y}^{aux} \to \hat{X}$ be the formal completion of $(f^{aux})^{st}$ at the
special fiber.
For each \'{e}tale tail $\ol{X}_b$ of $\ol{X}$, let $\ol{x}_b$ be
the intersection of $\ol{X}_b$ with the rest of $\ol{X}$.  Write $\mc{D}_b$
for the formal completion of $\ol{X}_b \backslash \{\ol{x}_b\}$ in $X_{R^{st}}$. 
Then $\mc{D}_b$ is a closed formal disk, which is
certainly preserved by $\sigma$.  Also, let $\mc{U}$ be the disjoint union of the formal completion of
$\ol{X} \ \backslash \bigcup_b \ol{X}_b$ with the formal completions of the $\ol{x}_b$ in $X_{R^{st}}$.  

Write $\mc{V} = \hat{Y} \times_{\hat{X}} \mc{U}$.  We know from the construction
of the auxiliary cover that 
$$\mc{V} = \Ind_{G^{aux}}^G \ \hat{Y}^{aux} \times_{\hat{X}} \mc{U}.$$
Since $\sigma$ preserves the auxiliary cover and acts trivially on its special
fiber, $\sigma$ acts as an automorphism on 
$\mc{V}$ and acts trivially on its special fiber.  By uniqueness of tame
lifting, 
$\mc{E}_b := \hat{Y} \times_{\hat{X}} \mc{D}_b$
is the unique lift of $\ol{Y} \times_{\ol{X}} (\ol{X}_b \backslash \{\ol{x}_b\})$ to
a cover of $\mc{D}_b$ (where the branching is compatible with that of $f$, if $\ol{X}_b$ is primitive).  
This means that $\sigma$ acts as an automorphism on $\hat{Y} \times_{\hat{X}} \mc{D}_b$ as well.

Define $\mc{B}_b := \mc{U} \times_{\hat{X}} \mc{D}_b$, the boundary of the disk
$\mc{D}_b$.  A $G$-cover of formal
schemes $\hat{Y} \to \hat{X}$ such that $\hat{Y} \times_{\hat{X}} \mc{U} \cong
\mc{V}$ and $\hat{Y} \times_{\hat{X}}
\mc{D}_b \cong \mc{E}_b$ is determined by a patching isomorphism
$$\varphi_b: \mc{V} \times_{\mc{U}} \mc{B}_b \stackrel{\sim}{\to} \mc{E}_b
\times_{\mc{D}_b} \mc{B}_b$$ for each $b$.
The isomorphism $\varphi_b$ is determined by its restriction
$\ol{\varphi}_b$ to the special fiber.  

Let $\ol{X}_{b, \infty}$ be the generic point of $\Spec \hat{\mc{O}}_{\ol{X}_b, \ol{x}_b}$, and define
$\ol{Y}_{b, \infty}$ (resp.\ $\ol{Y}^{aux}_{b, \infty}$) to be $\ol{Y}
\times_{\ol{X}} \ol{X}_{b, \infty}$ (resp.\ 
$\ol{Y}^{aux} \times_{\ol{X}} \ol{X}_{b, \infty}$).  Then $\ol{Y}_{b, \infty} = \Ind_{G^{aux}}^G
\ol{Y}^{aux}_{b, \infty}$.  Since $\sigma$
acts trivially on $\ol{Y}^{aux}_{b, \infty}$, it acts trivially on $\ol{Y}_{b,
\infty}$, which is the special fiber of both $\mc{V} \times_{\mc{U}} \mc{B}_b$ and $\mc{E}_b
\times_{\mc{D}_b} \mc{B}_b$.  So $\sigma$ acts trivially on $\ol{\varphi_b}$, and thus on 
$\varphi_b$.  Thus $\hat{f^{\sigma}} \cong \hat{f}$, and by Grothendieck's
Existence Theorem, $f^{\sigma} \cong f$.

Lastly, we must check that $\sigma$ acts trivially on $\ol{f}$.  This is clear
away from the \'{e}tale tails.  Now, 
for each \'{e}tale tail $\ol{X}_b$, we know $\sigma$ acts trivially on
$\ol{X}_b$, so it must act vertically on $\ol{Y}_b := \ol{Y} \times_{\ol{X}} \ol{X}_b$.  
But $\sigma$ also acts trivially on  $\ol{Y}^{aux}_{b, \infty}$.  
Since $\ol{Y}_{b, \infty}$ is induced from $\ol{Y}^{aux}_{b, \infty}$, $\sigma$ acts trivially on 
$\ol{Y}_{b, \infty}$.  Therefore, $\sigma$ acts trivially on $\ol{Y}_b$.  
\end{proof}

The auxiliary cover $f^{aux}$ is often simpler to work with than $f$ due to the following:

\begin{prop}\label{Pauxgroups}
If we assume that a $p$-Sylow subgroup of $G$ is \emph{cyclic}, then the group
$G^{aux}$ has a normal subgroup of order $p$. 
\end{prop}

\begin{proof} 
Let $\ol{S}$ be the union of all inseparable components of $\ol{X}$.  By construction, the inverse image $\ol{V}$
of $\ol{S}$ in $\ol{Y}^{aux}$ is connected, and its decomposition group is
$G^{aux}$.  By \cite[Corollary 2.12]{Ob:vc}, $G^{aux}$ has a normal subgroup of order $p$.
\end{proof}

Lastly, in the case that a $p$-Sylow subgroup of $G$ is cyclic, we make a further simplification of the auxiliary cover, as in
\cite[Remarque 3.1.8]{Ra:sp}.  
Since $G^{aux}$ has a normal subgroup of order $p$, \cite[Corollary 2.4 (i)]{Ob:vc} shows that  
the quotient $G^{str}$ of $G^{aux}$ by its maximal normal prime-to-$p$ subgroup $N$ is isomorphic to $\ints/p^n \rtimes \ints/m_{G_{aux}}$, 
where the action of $\ints/m_{G_{aux}}$ on $\ints/p^n$ is faithful.  Note that $m_{G_{aux}} | m_G$.
Then $Y^{str} := Y^{aux}/N$ is a branched $G^{str}$-cover of $X$, called the \emph{strong auxiliary cover}.  
Constructing the strong auxiliary cover is one of the key places where it is essential to assume that a $p$-Sylow
subgroup of $G$ is cyclic, as otherwise $G^{aux}$ does not necessarily have such nice group-theoretical
properties.

The branching on the generic fiber of the strong auxiliary cover is as follows:  
At each point of $X$ where the branching index of $f$ was divisible by $p$,
the branching index of $f^{str}$ is a power 
of $p$ (as $G^{str}$ has only elements of $p$-power order and of 
prime-to-$p$ order).  
At each branch point specializing to an \'{e}tale tail $b \in
B_{\text{\'{e}t}}$, the ramification index is $m_b$, where $m_b | (p-1)$ (cf.
\cite[\S3.3.2]{Ra:sp}).  

The following lemma shows that it will generally suffice to look at the strong auxiliary
cover instead of the auxiliary cover.   

\begin{lemma}\label{Lstr2orig}
Let $f: Y \to X$ be as in \S\ref{Sstable}.  Let $L$ be a field over which both the stable model of $f^{str}$ and all the branch 
points of $Y^{aux} \to Y^{str}$ are defined.  
Then the stable model $f^{st}$ of $f$ can be defined over a tame extension of $L$.
\end{lemma}

\begin{proof}
Since $Y^{str} = Y^{aux}/N$, and $p \nmid |N|$, it follows from \cite[Proposition 6.2]{Ob:fm} that $(f^{aux})^{st}$ can be 
defined over a tame extension of $L$.  By Lemma \ref{Laux2orig}, so can $f^{st}$.
\end{proof}

While the Galois group of the (strong) auxiliary cover is simpler than the
original Galois group of $f$, we generally are made to pay 
for this with the introduction of new branch points.  
Understanding where these branch points appear is key to understanding the
minimal field of definition of the stable reduction of the auxiliary cover.

\section{Proof of the Main Theorem}\label{CHmain}
In this section, we will prove Theorem \ref{Tmain}.  Let $k$ be an algebraically closed field of characteristic $p$, let $R_0 = 
W(k)$, and let $K_0 = \Frac(R_0)$.  Note that if $k \cong \ol{\FF}_p$, then $K_0 \cong \rats_p^{ur}$.
Also, for all $i > 0$, we set $K_i = K_0(\zeta_{p^i})$, where $\zeta_{p^i}$ is a primitive $p^i$th root of unity.

Let $f:Y \to X \cong \proj^1$ be a three-point $G$-cover defined over
$\ol{K_0}$, where $G$ has a cyclic $p$-Sylow subgroup of order $p^n$ and $m_G = 2$.
Since $m_G |(p-1)$ (\cite[Lemma 2.1]{Ob:vc}), we may (and do) assume throughout this section that $p \ne 2$.
We break this section up into the cases where the number $\tau$ of branch points
of $f: Y \to X=\proj^1$ with prime-to-$p$
branching index is $1$, $2$, or $3$.  By Lemma \ref{Lcorrectspec}, if $f$ has bad reduction, then $\tau$ is the number of primitive tails 
of the stable reduction.
The cases $\tau = 2$ and $\tau = 3$ are
quite easy, whereas the case $\tau = 1$ is much more involved.  This stems from the
appearance of new tails in the stable reduction of $f$ in the case $\tau = 1$.
The ideas in the proof of the $\tau = 1$ case should work as
well in the $\tau = 0$ case, but the computations will be more difficult.  See Question \ref{Qgap}.

We mention that, because any finite extension of $K/K_0$ has cohomological dimension 1, then if
$K$ is the field of moduli of a $G$-Galois cover relative to $K_0$, it is also a field of definition 
(\cite[Proposition 2.5]{CH:hu}).

\subsection{The case $\tau = 3$}\label{Stauequals3}
\begin{prop}\label{Pm2tau3}
Assume $f: Y \to X$ is a three-point $G$-cover defined over $\ol{K_0}$ where $G$
has a cyclic $p$-Sylow subgroup $P$ with 
$m_G=|N_G(P)/Z_G(P)| = 2.$  Suppose that all three branch points of $f$ have
prime-to-$p$ branching index.
Then $f$ has potentially good reduction.  Additionally, $f$ has a model defined
over $K_0$, and thus the field of moduli of $f$ relative to $K_0$ is $K_0$.
\end{prop}

\begin{proof}
Suppose $f$ has bad reduction.  We know that the stable reduction must have three primitive tails.
But this contradicts Corollary \ref{Chalfsigma}.  So $f$ has potentially good reduction.

Let $\ol{f}: \ol{Y} \to \ol{X}$ be the reduction of $f$ over $k$.  Then $f$ is tamely
ramified.  By \cite[Theorem 4.10]{Fu:hs}, if $R$ is the ring of integers of any finite extension
$K/K_0$, then there exists a unique deformation $f_R$ of $\ol{f}$ to a cover defined over $R$.
It follows that $f_{R_0}$ exists, and $f_{R_0} \otimes_{R_0} R \cong f_R$.  Thus
$f_{R_0} \otimes_{R_0} K_0$ is the
model we seek.
\end{proof}

Proposition \ref{Pm2tau3}, while an easy consequence of the vanishing cycles formula, 
gives a proof that the modular curve $X(N)$ has good reduction to
characteristic $p$ for many $p$, without relying on its modular interpretation
(see, for instance, \cite{DR:sm}).

\begin{corollary}\label{Cmodular}
Let $N \in \nats$ have prime factorization $N = \prod_{i=1}^r p_i^{a_i}$.
Let $M$ be the product of all primes that divide $p_i^2-1$ for more than one $i$.
Then the modular curve $X(N)$ has good reduction at all primes not dividing $6NM$.
\end{corollary}

\begin{proof}
The modular curve $X(N)$ can be realized (via the $j$-function) as a $PSL_2(\ints/N)$-cover 
$f: X(N) \to X(1) \cong \proj^1$, branched at three points of index $2$, $3$, and $N$, respectively.   
Let $G = PSL_2(\ints/N)$, and let $p$ be a prime dividing $|G|$ but not dividing $6NM$. 
By the Chinese remainder theorem, one sees that 
$G \cong \prod_i PSL_2(\ints/p_i^{a_i})$.  Furthermore, for each $i$, we have an exact sequence
$$1 \to P_i \to PSL_2(\ints/p_i^{a_i}) \to PSL_2(p_i) \to 1,$$ where $P_i$ is a $p_i$-group and the third map is the modulo $p_i$ projection
on matrix entries.  The order of $PSL_2(p_i)$ is $p_i(p_i^2-1)/2$, and it is well known that $PSL_2(p_i)$ contains cyclic subgroups of
order $\frac{p_i + 1}{2}$ and $\frac{p_i - 1}{2}$.  By the Schur-Zassenhaus theorem, these subgroups lift to $PSL_2(\ints/p_i^{a_i})$.
In particular, the $p$-Sylow subgroup of $PSL_2(\ints/p_i^{a_i})$ is cyclic.  

Since $p \nmid 6NM$, we have that $p$ divides the order of exactly one $PSL_2(p_i)$.  In particular, the $p$-Sylow subgroup of
$G$ is cyclic.  It is well known that $m_{PSL_2(p_i)} = 2$ for this $p_i$ (relative to the prime $p$), and thus the same is true for $m_G$.
Since $p \nmid 6N$, Proposition \ref{Pm2tau3} shows that the cover $f$, and thus $X(N)$, has good reduction to characteristic $p$.
\end{proof}

\subsection{The case $\tau = 2$}\label{Stauequals2}
If there are exactly two branch points with prime-to-$p$ branching index, then $f$ has bad reduction ($f$ cannot have good reduction because it will have a branch point with $p$ dividing the branching index).  We use the notation of \S\ref{Sstable}.  In particular, 
$f^{st}: Y^{st} \to X^{st}$ is the stable model of $f$ and $\ol{f}: \ol{Y} \to \ol{X}$ is the stable reduction.

\begin{prop}\label{Pm2tau2}
Assume $f: Y \to X$ is a three-point $G$-cover defined over $\ol{K_0}$ where $G$
has a cyclic $p$-Sylow subgroup $P$ of order $p^n$ with 
$m_G=|N_G(P)/Z_G(P)| = 2.$  Suppose that two of the three branch points of $f$
have prime-to-$p$ branching index.
Then the stable model of $f$ can be defined over a tame extension $K$ of $K_n$. 
In particular, $f$ can be defined over $K$.  
Thus the field of moduli of $f$ relative to $K_0$ is contained in a tame extension of $K_n$.
\end{prop}

\begin{proof}
We know that $\ol{X}$ must have two primitive tails, and Corollary \ref{Chalfsigma} shows that there are no new tails. 
The effective ramification invariant for each of the primitive tails is $\frac{1}{2}$ by the
vanishing cycles formula (\ref{Evancycles}).  Then the strong auxiliary
cover $f^{str}: Y^{str} \to X$ is a three-point $\ints/p^{\nu} \rtimes \ints/2$-cover, for some $\nu \leq n$.  

By \cite[Proposition 7.6]{Ob:fm}, the stable model of $f^{str}$ is defined over a tame extension $K^{str}$ of $K_{\nu} \subseteq K_n$.
Since the branch loci of $f$, $f^{aux}$, and $f^{str}$ are each the same three points, all branch points of the 
canonical map $Y^{aux} \to Y^{str}$ are ramification points of $f^{str}$.  The ramification points of $f^{str}$ specialize to 
distinct points on $\ol{Y}^{str}$, so $G_{K^{str}}$ permutes them trivially.  Thus they are defined over $K^{str}$.
By Lemma \ref{Lstr2orig}, the stable model of $f$ is defined over a tame extension $K$ of $K^{str}$.
\end{proof}

\subsection{The case $\tau = 1$}\label{Stauequals1}
Now we consider the case where only one point, say $0$, has prime-to-$p$
branching index.  As in the case $\tau = 2$, the cover $f$ has bad reduction.
The goal of this (rather lengthy) section is to prove the following proposition:

\begin{prop}\label{Pm2tau1}
Assume $f: Y \to X$ is a three-point $G$-cover defined over $\ol{K_0}$ where $G$
has a cyclic $p$-Sylow subgroup $P$ with 
$m_G=|N_G(P)/Z_G(P)| = 2$ and $p \ne 3$.  Suppose that exactly one of the three branch points of $f$
has prime-to-$p$ branching index.
Then the stable model of $f$ can be defined over a finite extension $K/K_0$ such that the $n$th higher ramification groups 
for the upper numbering for (the Galois closure of) $K/K_0$ vanish.  In particular, $f$ can be defined over such a $K$. 
Thus the $n$th higher 
ramification group for the upper numbering for the field of moduli of $f$ relative to $K_0$ vanishes.
\end{prop}

We mention that, because $m_G = 2$, the stable reduction of $f$ is monotonic (Proposition \ref{Pmonotonic}).

\subsubsection{}
We first deal with the case where there is one primitive tail $\ol{X}_b$, but no new
\'{e}tale tails.  Then the vanishing cycles formula (\ref{Evancycles}) shows that $\sigma_b = 1$.  Furthermore, we claim that 
$m_{G^{aux}} = 1$.  If this were not the
case, then the strong auxiliary cover would have Galois group $G^{str} \cong \ints/p^{\nu} \rtimes \ints/2$, for some $\nu \leq 
n$, but only one branch point with prime-to-$p$ branching index.  Then taking the quotient by $\ints/p^{\nu}$ would 
yield a contradiction.

Since we are assuming that the stable reduction of $f$ has no new tails, the
auxiliary cover $f^{aux}: Y^{aux} \to X$ is branched at either two or three points.  If it is branched at three points,
we conclude using \cite[Proposition 7.15]{Ob:fm} that the stable model of $f^{aux}$
(which is the modified stable model) is defined over some $K$ such that the $n$th higher ramification groups of the
extension $K/K_0$ vanish.  If $f^{aux}$ 
is branched at two points (without loss of generality, $0$ and $\infty$), it is a cyclic cover, and thus clearly defined 
over $K_n$.  Furthermore, the points in the fiber above $1$ 
are defined over $K_n$.  We conclude that the modified stable model of $f^{aux}$ is defined over $K_n$ (Remark 
\ref{Rmodifiedstable}).  
By \cite[IV, Corollary to Proposition 18]{Se:lf}, the $n$th higher ramification groups of $K_n/K_0$
vanish.  By Lemma \ref{Laux2orig}, Proposition \ref{Pm2tau1} is true in this case.

\subsubsection{}
We now come to the main case, where there is a new \'{e}tale tail $\ol{X}_b$ and a
primitive tail $\ol{X}_{b'}$.  We will assume for the remainder of \S\ref{Stauequals1} that $p \ne 3$ (although it is likely that the 
main result should hold in the case $p=3$, see Question \ref{Qgap}).  

Fix, once and for all, a coordinate $x$ corresponding to the smooth model $X_{R_0}$ with special fiber $\ol{X}_0$, so that
$f$ is branched at $x=0$, $x=1$, and $x= \infty$.
By the vanishing cycles formula (\ref{Evancycles}), the new tail $\ol{X}_b$ has
$\sigma_b = 3/2$ and the primitive tail $\ol{X}_{b'}$ has 
$\sigma_{b'} = 1/2$.  It is then clear that the auxiliary cover has four branch
points: at $x=0$, $x=1$, $x=\infty$, and $x=a$, where $a$ is in the disk corresponding to $\ol{X}_b$. 
Keep in mind that, by the construction of the auxiliary cover (\S\ref{Saux}), we may always replace $a$ by any other point in the disk 
corresponding to $\ol{X}_b$.  Also, the modified
stable model of the auxiliary cover is, in fact, the stable model.  The
strong auxiliary cover then has Galois group $G^{str} \cong \ints/p^{\nu}
\rtimes \ints/2$ for some $\nu \leq n$.  
Without loss of generality, we can assume that
$0$ and $a$ are branched of index $2$, and $1$ and $\infty$ are branched
of $p$-power index.  After a possible application of the transformation $x \to \frac{x}{x-1}$ of $\proj^1$, 
which interchanges $1$ and $\infty$ while fixing $0$, we may and do further assume that $a$ does not collide with $\infty$ on the 
smooth model of $X$ corresponding to the coordinate $x$ (i.e., $|a| \leq 1$).

\begin{lemma}\label{Lfullbranch}
At least one point of $f^{str}$ is branched of index $p^{\nu}$.  Such a point specializes to the original component.
\end{lemma}

\begin{proof} 
Consider the $\ints/p \rtimes \ints/2$-cover $f' := Y^{str}/Q \to X$, where $Q$ has order $p^{\nu - 1}$.  This must be branched
at at least three points, thus at $1$ or $\infty$.  If $1$ or $\infty$ is a branch point of $f'$, then its branching index in $f$
is $p^{\nu}$.  By Lemma \ref{Lcorrectspec}, any branch point of index $p^{\nu}$ specializes to a $p^{\nu}$-component of
$\ol{X}$.  By Lemma \ref{Lwildspec}, it specializes to $\ol{X}_0$.
\end{proof}

Let us fix some additional notation for \S\ref{CHmain} by writing down the equations of the cover $f^{str}: Y^{str}
\to X^{str}$.  Let $Z^{str} = Y^{str}/(\ints/p^{\nu})$.
Then $Z^{str} \to X^{str}$ is a degree 2 cover of ${\proj}^1$'s,
branched at $0$ and $a$.  Therefore, $Z^{str}$ can be given (birationally) over $\ol{K_0}$ by the
equation
\begin{equation}\label{Efstr1}
z^2 = \frac{x-a}{x}.
\end{equation}
Fix a choice of $\sqrt{1-a}$.  Since $z = \pm 1$ (resp.\ $\pm \sqrt{1-a}$) corresponds to $x = \infty$ (resp.\ $x = 1$), then
$Y^{str} \to Z^{str}$ can be given 
(birationally) over $\ol{K_0}$ by the equation
\begin{equation}\label{Efstr2}
y^{p^{\nu}} = g(z) := \left(\frac{z+1}{z-1}\right)^r \left(\frac{z + \sqrt{1-a}}{z- \sqrt{1-a}}\right)^s
\end{equation} 
for some integers $r$ and $s$, which are well-defined modulo $p^{\nu}$.  Without loss of generality, we take
$0 < r, s < p^{\nu}$.  The branching index of
$f^{str}$ at $\infty$ is $p^{\nu - v(r)}$, and at $1$ it is $p^{\nu - v(s)}$.

Write $\ol{Z}^{str}$ for $\ol{Y}^{str}/(\ints/p^{\nu})$, and 
let $\ol{Z}_b$ (resp.\ $\ol{Z}_{b'}$) be the unique irreducible component of $\ol{Z}^{str}$ above the new tail $\ol{X}_b$ (resp.\ the primitive 
tail $\ol{X}_{b'}$).

We will work over a large enough finite extension $K/K_0$ (i.e., we assume the stable model of $f^{str}$ is defined over $K$
and we replace $K$ by a finite extension whenever convenient).  
Let $e \in K$ be such that $|e|$ is the radius of the disk $\mc{D}$ corresponding to $\ol{Z}_b$.
Since $x = a$ corresponds to $z=0$, we can choose a coordinate $t$ on the disk $\mc{D}$ such that $z = et$.  If $\hat{Y}$,
$\hat{Z}$ are the formal completions of $(Y^{str})^{st}$ and $(Z^{str})^{st}$ along their special fibers, then the
torsor $\hat{Y} \times_{\hat{Z}} \mc{D} \to \mc{D}$ can be given generically, after a possible finite extension of $K$, by the equation
\begin{equation}\label{Edisk}
y^{p^{\nu}} = 1 + \frac{g'(0)}{1!}(et) + \frac{g''(0)}{2!}(et)^2 + \cdots.
\end{equation}  
Now, since $\sigma_b = \frac{3}{2}$, and since $\ol{X}_b$ intersects a $p$-component (see Lemma \ref{Lnobigjumps} below), 
we know that the generic fiber of this torsor must split into
$p^{\nu-1}$ connected components, each of which has \'{e}tale reduction and is birationally equivalent to a $\ints/p$-cover 
branched at one point with conductor $3$. 
Let $c_i = \frac{g^{(i)}(0)}{i!}e^i$.  Then 
\begin{equation}\label{Edisk2}
y^{p^{\nu}} = 1 + c_1t + c_2t^2 + \cdots.
\end{equation}

Note that we have fixed the meaning of the symbols $f$, $p$, $k$, $a$, $r$, $s$, $n$, $\nu$, $x$, $y$, $z$, $e$, $t$, $c_i$, $g(z)$, $G$, 
$\sigma_b$, $\sigma_{b'}$, $\ol{X}_b$, $\ol{X}_{b'}$, $\ol{X}_0$, $\ol{Z}_b$, and $\ol{Z}_{b'}$.  We will also use the notation $\mc{G}$ and $\mc{G}'$
for the dual graph and augmented dual graph of $\ol{X}$ (\S\ref{Sgraph}).

The idea of the proof is as follows: In order to place bounds on the higher ramification filtration of the field of moduli of $f$, it suffices by the results 
of \S\ref{Saux} to understand the minimal field of definition of the stable model of $f^{str}$.  In order to do this, we must first calculate 
the disk corresponding to the new tail $\ol{X}_b$.  Since we can take $a$ to be any value in this disk, we choose the value defined over the ``smallest"
field possible to be our $a$, and then $K_n(a, \sqrt{1-a})$ will be a field of definition of $f^{str}$.  This is done in the first large subsection,
\S\ref{Setaletails}.

Since understanding the monodromy action is enough to pin down a field of definition of the stable model of $f^{str}$, 
the goal is then to determine the monodromy action of $\Gal(\ol{K_0}/K_0)$ on $\ol{f}^{str}$. 
We use the criterion of \cite[Proposition 4.9]{Ob:fm}, which essentially says that if this action fixes the tails of $\ol{Y}^{str}$, then it fixes all of
$\ol{Y}^{str}$.  To this end, in \S\ref{Sinseparabletails} (our second large subsection), we show exactly which disks correspond to inseparable tails of
$\ol{X}$ (if there are any).  

In \S\ref{Sstablemodel}, we put all of the information from \S\ref{Setaletails} and \S\ref{Sinseparabletails} together to determine 
a field of definition for the stable model of $f$ (up to tame extension).  
Lastly, in \S\ref{Shigherram}, we show that the appropriate higher ramification groups vanish.  

We start with \S\ref{Sprelimlemmas}, where we show some basic properties of $(f^{str})^{st}$ and prove a couple
of algebraic results that will be used later.
\subsubsection{Preliminary lemmas}\label{Sprelimlemmas}

\begin{lemma}\label{Lnobigjumps}
Every \'{e}tale tail $\ol{X}_c$ of $\ol{X}$ intersects a $p$-component.
\end{lemma}

\begin{proof}
By Lemma \ref{Ltailetale}, $\ol{X}_c$ intersects an inseparable component.
If $\ol{X}_c$ intersects a $p^{\alpha}$-component, with $\alpha > 1$, then \cite[Lemma 4.2]{Ob:vc} shows that $\sigma_b
\geq p/2$.  Since $p/2 > 2$, this contradicts the vanishing cycles formula (\ref{Evancycles}).
\end{proof}

\begin{lemma}\label{Lfullwildspec}
The map $f^{str}$ is branched at $x = \infty$ of index $p^{\nu}$, and $x = \infty$ specializes to the original component $\ol{X}_0$.  
If $v(a-1) = 0$, then $f^{str}$ is branched at $x = 1$ of index $p^{\nu}$, and 
$x=1$ also specializes to $\ol{X}_0$.  
\end{lemma}

\begin{proof} 
Assume for a contradiction that $\infty$ is not branched of index $p^{\nu}$.  
Then $1$ is branched of index $p^{\nu}$, and specializes to $\ol{X}_0$ by Lemma \ref{Lwildspec}.  
Thus, the deformation data above $\ol{X}_0$ are multiplicative by Lemma \ref{Lcritical}, and identical by \cite[Proposition 5.2]{Ob:fm}.  
By assumption and by Lemma \ref{Lcorrectspec}, $\infty$ does not specialize to the original component.  Then consider
the unique point $\ol{x} \in \ol{X}_0$ such that the specialization $\ol{\infty}$ of $\infty$ lies outward from $\ol{x}$.
Since $|a| \leq 1$, there is no \'{e}tale tail lying outward from $\ol{x}$.  
If $\varepsilon \in E(\mc{G}')$ corresponds to $(\ol{x}, \ol{X}_0, \ol{W})$, for some $\ol{W}$, then
Lemma \ref{Lsigmaeff} shows that $\sigma_{\varepsilon}^{\eff} = 0$.  
But this means that $\sigma_{\varepsilon} = 0$
for each deformation datum above $\ol{X}_0$, which contradicts Lemma \ref{Lcritical}.  We have thus shown that
$\infty$ is branched of index $p^{\nu}$.  By Lemma \ref{Lwildspec}, $\infty$ specializes to $\ol{X}_0$.

Now suppose $v(a-1) = 0$.  Assume for a contradiction that $1$ does not specialize to the original
component.  Consider the unique point $\ol{x} \in \ol{X}_0$ such that the specialization $\ol{1}$ of $1$ lies outward from $\ol{x}$, and let 
$\varepsilon \in E(\mc{G}')$
correspond to $(\ol{x}, \ol{X}_0, \ol{W})$ for some $\ol{W}$.
As in the previous paragraph, $\sigma_{\varepsilon} = 0$ for each deformation datum above $\ol{X}_0$, and we get a
contradiction.
\end{proof}

\begin{corollary}\label{Corigmult}
All deformation data above the original component are multiplicative.  
In particular, $\delta^{\eff}_{\ol{X}_0} = \nu + \frac{1}{p-1}$.
\end{corollary}

\begin{proof}
By Lemma \ref{Lfullwildspec}, $x = \infty$ specializes to the original component $\ol{X}_0$.  By Lemma \ref{Lcritical}, all deformation data
above $\ol{X}_0$ are multiplicative.
\end{proof}

\begin{lemma}\label{Lk0root}
Let $c = \alpha + \frac{\beta}{\sqrt{1-a}}$, where $\alpha, \beta, a \in K$ and $\sqrt{1-a}$ means either square root.
Let $a_0 = 1 - \left(\frac{\beta}{\alpha}\right)^2$.  
If $v(c) > 0$ and $v(\alpha) = 0$, then $v(a - a_0) = v(c) + 2v(\beta)$.  Note that $a_0 \in K_0(\alpha, \beta)$. 
\end{lemma}

\begin{proof}
Solving for $a$, we find that $a = 1 - \left(\frac{\beta}{c - \alpha}\right)^2$. 
Choose $a_0 = 1 - \left(\frac{\beta}{\alpha}\right)^2$.  
Then 
$$a - a_0 = \beta^2\left(\frac{2c\alpha - c^2}{\alpha^2(\alpha - c)^2}\right).$$
Clearly, $v(a - a_0) = 2v(\beta) + v(c)$.
\end{proof}

For positive integers $q$ and $r_1, \ldots, r_n$ such that $\sum_{i} r_i = q$, define 
$$\binom{q}{r_1, \ldots, r_n} = \frac{q!}{r_1!\cdots r_n!}.$$ 
We leave the proof of the following lemma to the reader:
\begin{lemma}\label{Lbinomcoeffs}
For any prime $p$, $$v_p\left( \binom{q}{r_1, \ldots, r_n} \right) \geq \max_i\left(v_p\left( \binom{q}{r_i} \right)\right) = v_p(q) - \min_i(v_p(r_i))$$
(here $\binom{q}{r_i}$ is the standard binomial coefficient).
\end{lemma}

\subsubsection{The new (\'{e}tale) tail}\label{Setaletails}
An open $p$-adic disk is determined by its radius and any point inside.  
The disk corresponding to the new tail $\ol{X}_b$ is centered at $a$.  The following lemma determines its radius.

\begin{lemma}\label{Ltailradius}
Let $\rho$ (resp.\ $e$) be an element of $K$ such that $|\rho|$ (resp.\ $|e|$) is the radius of the disk centered at
$x=a$ corresponding to $\ol{X}_b$ (resp.\ the disk centered at $z=0$ corresponding to $\ol{Z}_b$).
\begin{enumerate}
\item 
If $v(a) = v(a-1) = 0$, then $v(\rho) = \frac{2}{3}( \nu + \frac{1}{p-1} )$ and 
$v(e) = \frac{1}{3}(\nu + \frac{1}{p-1})$.
\item 
If $v(a) > 0$, then $v(\rho) = \frac{2}{3}(\nu + \frac{1}{p-1}) + \frac{1}{3}v(a)$ and
$v(e) = \frac{1}{3}(\nu + \frac{1}{p-1} - v(a))$.
\item 
If $v(a-1) > 0$, then $v(\rho) = \frac{2}{3}(\nu + \frac{1}{p-1} + v(1-a))$ and $v(e) = \frac{1}{3}(\nu
+ \frac{1}{p-1} + v(1-a))$.
\end{enumerate}
\end{lemma}

\begin{proof}
Since $z^2 = \frac{x-a}{x}$, then for any $z$, $v(z) = \frac{1}{2}(v(x - a) - v(x))$.
Since $\ol{X}_b$ is a new tail, $x=0$ does not specialize to the corresponding disk.  
So for any $x$ in this disk, $v(x-a) > v(a)$, thus $v(x) = v(a)$.  This shows that $v(z) = \frac{1}{2}(v(x-a) - v(a))$ in this disk,
and thus $v(e) = \frac{1}{2}(v(\rho) - v(a))$.  Therefore, it suffices to prove the statements about $v(\rho)$.

Consider the path $\{v_i\}_{i=0}^{j}$, $\{e_i\}_{i=0}^{j-1}$, where $v_0$ corresponds to $\ol{X}_0$ and $v_j$ corresponds to $\ol{X}_b$ (\S\ref{Sgraph}).
Write $\epsilon_i$ (resp.\ $\sigma^{\eff}_i$, $\delta^{\eff}_i$) for $\epsilon_{e_i}$ (resp.\ $\sigma^{\eff}_{e_i}$, 
$\delta^{\eff}_{v_i}$) (see Definition \ref {Dsigmaeff}).  Then $\delta_0 = \nu + \frac{1}{p-1}$, whereas $\delta_j = 0$.

\emph{To (1):}
Suppose $v(a) = v(a-1) = 0$.  Then $\ol{X}_b$ is the only \'{e}tale tail lying outward from the point $\ol{x}_0$ corresponding to $e_0$.  
No branch points with branching index divisible by $p$ lie outward from $\ol{x}_0$, either.  
By Lemma \ref{Lsigmaeff}, $\sigma^{\eff}_i =
\frac{3}{2}$ for all $0 \leq i < j$.
By applying Lemma \ref{Lsigmaeffcompatibility} (3) to each $e_i$, $0 \leq i <
j$, we obtain $v(\rho) = \frac{2}{3}(\nu + \frac{1}{p-1})$.

\emph{To (2):}
Suppose $v(a) > 0$.  In order to separate the specializations of $x=a$ and $x=0$ on the special fiber, 
there must be a component $\ol{W}$ of $\ol{X}$ corresponding to the closed disk of radius $|a|$ and center $0$ (or
equivalently, center $a$).  Suppose $\ol{W}$ corresponds to $v_{i_0}$.  Then, for $i < i_0$, Lemma \ref{Lsigmaeff} shows that
$\sigma^{\eff}_i = 1$.  For $i \geq i_0$, Lemma \ref{Lsigmaeff} shows that $\sigma_i^{\eff} = \frac{3}{2}$.  By
construction, we have $\sum_{i=0}^{i_0-1} \epsilon_i = v(a)$.  Applying Lemma
\ref{Lsigmaeffcompatibility} (3) to each of the edges $e_0, \ldots, e_{i_0-1}$, 
we see that $\delta^{\eff}_{\ol{W}} = \nu + \frac{1}{p-1} - v(a)$.  Then, applying
Lemma \ref{Lsigmaeffcompatibility} (3) to each of the edges $e_{i_0}, \ldots, e_{j-1}$, we see that $\sum_{i=i_0}^{j-1} \epsilon_i =
\frac{2}{3} (\nu + \frac{1}{p-1} - v(a))$.  So $v(\rho) = \sum_{i=0}^{j-1} \epsilon_i = \frac{2}{3}(\nu + \frac{1}{p-1}) + \frac{1}{3}v(a)$.

\emph{To (3):}
Suppose $v(a - 1) > 0$. 
In order to separate the specializations of $x=a$ and $x=1$ on the special fiber, 
there must be a component $\ol{W}$ of $\ol{X}$ corresponding to the closed disk of radius $|1-a|$ and center $1$ (or
equivalently, center $a$).  Suppose $\ol{W}$ corresponds to $v_{i_0}$.  Then, for $i < i_0$, Lemma \ref{Lsigmaeff} shows that
$\sigma^{\eff}_i = \frac{1}{2}$.  For $i \geq i_0$, Lemma \ref{Lsigmaeff} shows that $\sigma_i^{\eff} = \frac{3}{2}$.  By
construction, we have $\sum_{i=0}^{i_0-1} \epsilon_i = v(1-a)$.  Applying Lemma
\ref{Lsigmaeffcompatibility} (3) to each of the edges $e_0, \ldots, e_{i_0-1}$, 
we see that $\delta^{\eff}_{\ol{W}} =\nu + \frac{1}{p-1} - \frac{1}{2}v(1-a)$.  
Then, applying Lemma \ref{Lsigmaeffcompatibility} to each of the edges $e_{i_0}, \ldots, e_{j-1}$, we see that $\sum_{i=i_0}^{j-1} 
\epsilon_i = \frac{2}{3} (\nu + \frac{1}{p-1} - \frac{1}{2}v(1-a))$.  So $v(\rho) = \sum_{i=0}^{j-1} \epsilon_i = 
\frac{2}{3}(\nu + \frac{1}{p-1} + v(1-a))$.  
\end{proof}

We now determine a point $a_0$ inside the disk corresponding to $\ol{X}_b$.  
It will turn out that $a_0$ (thus the disk) is uniquely determined by $p$, $r$, and $s$.  
We choose $a_0$ so that it is defined over as small an extension of $K_0$ as possible.
Our strategy will be to look at equations (\ref{Edisk}) and (\ref{Edisk2}), understand the dependence of the coefficients
$c_i = \frac{g^{(i)}(0)}{i!}e^i$ on $a$, and use Lemma \ref{Lartinschreier} to show that only for certain choices of $a$ can the torsor
given by (\ref{Edisk2}) split into $p^{\nu-1}$ copies of an Artin-Schreier cover.

The main idea of the argument is completely present when $v(a) = v(a-1) = 0$ (Proposition \ref{Panot01}, the simplest case).
Unfortunately, calculational difficulties make this idea more difficult to implement when $v(a) > 0$ or $v(a-1) >  0$ (Propositions \ref{Pa0} and \ref{Pa1}), 
and the arguments are much longer.  Furthermore, when $p=5$, we will have to use Lemma \ref{Lartinschreier} (2) instead of Lemma 
\ref{Lartinschreier} (1), which obscures the main idea even more.  Thus, on a first reading, the reader might choose to read only 
\S\ref{Setaletails}.1, as well as the statements of Propositions \ref{Pa0} and \ref{Pa1}, before moving on to \S\ref{Sinseparabletails}.

Note that our choice of $a$ does not affect which case we are in, as the radius of 
the disk corresponding to $\ol{X}_b$ is less than 1.   

\paragraph{The case $v(a) = v(a-1) = 0$}  

\begin{prop}\label{Panot01}
If $v(a) = v(a-1) = 0$, then the disk $\Delta$ corresponding to $\ol{X}_b$ contains the $K_0$-rational point 
$x = a_0$, where $a_0 = 1 - \frac{s^2}{r^2}$.
\end{prop}

\begin{proof}
Recall that we use the notation of (\ref{Edisk}) and (\ref{Edisk2}).
We know $v(e) = \frac{1}{3}(\nu + \frac{1}{p-1})$ by Lemma \ref{Ltailradius}.
Since $g^{(i)}(0)/i!$ is the coefficient of $z^i$ in the Maclaurin series expansion of $g$, and since $v(\sqrt{1-a}) = 0$, we obtain that 
$v(\frac{g^{(i)}(0)}{i!}) \geq 0$.  Since $c_i = \frac{g^{(i)}(0)}{i!}e^i$, we have $v(c_i) \geq iv(e) = \frac{i}{3}(\nu + \frac{1}{p-1})$.  It follows
that for $p | i$ (and $i \neq 0$), $v(c_i) > \nu + 
\frac{1}{p-1}$.  So for the torsor given by (\ref{Edisk2}) to split into $p^{\nu - 1}$ disjoint 
$\mu_p$-torsors with \'{e}tale reduction, Lemma
\ref{Lartinschreier} (1) says that we must have $v(c_1) \geq \nu + \frac{1}{p-1}$.  
In particular, we must have $v(g'(0)) \geq \frac{2}{3}(\nu + \frac{1}{p-1})$.  A calculation shows that 
$$g'(0) = 2r + \frac{2s}{\sqrt{1-a}}.$$

Since the branching index of $1$ is $p^{\nu}$, $s$ is a unit modulo $p$ and
$v(2s) = 0$.  Since $v(g'(0)) \geq \frac{2}{3}(\nu + \frac{1}{p-1})$, Lemma \ref{Lk0root} 
(with $c = g'(0)$, $\alpha = 2r$, $\beta = 2s$) shows that 
$a_0 = 1 - \frac{s^2}{r^2} \in K_0$ satisfies
$v(a-a_0) \geq \frac{2}{3}(\nu + \frac{1}{p-1})$.  
But by Lemma \ref{Ltailradius}, $v(\rho) = \frac{2}{3}(\nu + \frac{1}{p-1})$, where 
$|\rho|$ is the radius of $\Delta$.  So $a_0 \in \Delta$.
\end{proof}

\begin{remark}\label{Rnot01conductor3}
In fact, if $a$ is as in Proposition \ref{Panot01}, 
we have $v(c_3) = \nu + \frac{1}{p-1}$ and $v(c_i) > \nu+ \frac{1}{p-1}$ for $i \ne 3$.
By \cite[Lemma 3.1 (i)]{Ob:fm}, the torsor given by (\ref{Edisk2}) indeed splits into $p^{\nu-1}$ disjoint $\mu_p$-torsors with
\'{e}tale reduction birational to an Artin-Schreier cover branched at one point of conductor 3.  
\end{remark}

\paragraph{The case $v(a) > 0$}  

When $v(a) > 0$, the required calculations are somewhat more involved.
By Lemma \ref{Ltailradius}, we have $v(e) = \frac{1}{3}(\nu + \frac{1}{p-1} - v(a))$ in this case.

\begin{lemma}\label{Lvaint}
We have $v(a) = v(r+s) \leq \nu-1$.  In particular, $v(a) \in \ints$.
\end{lemma}

\begin{proof}
The cover $\ol{Y}^{str} \to \ol{Z}^{str}$ splits completely above the
specialization $\ol{z}$ of $z = 0$.
Recall that $t$ is a coordinate on the disk corresponding to $\ol{Z}_b$, so that $z = et$.  
Then $\ol{z}$ corresponds to the open disk $|t| < 1$, and \cite[Proposition
3.2.3 (2)]{Ra:ab} shows that this disk splits into $p^{\nu}$ disjoint copies in
$\ol{Y}^{str}$.  In particular, $g(et)$ is a $p^{\nu}$th power in $R[[t]]$.  
If $\sum \alpha_it^i$ is a power series in $R[[t]]$ that is a $p^{\nu}$th power,
the coefficient of $t$ must be divisible by $p^{\nu}$.  
So the coefficient $c_1$ of $t$ in $g(et)$, which is $g'(0)e$, has valuation at least
$\nu$, and thus $v(g'(0)) \geq \nu - v(e) = \frac{2}{3}\nu + \frac{1}{3}v(a) - \frac{1}{3(p-1)}$. 

On the other hand, $g'(0) = 2r  + \frac{2s}{\sqrt{1-a}}$, which can be written as
\begin{equation}\label{Ethisone}
g'(0) = 2r + 2s(1 + \frac{a}{2} + O(a^2)) = 2(r + s) + sa + s(O(a^2)),
\end{equation} 
where $O(a^2)$ represents terms whose valuation is at least $2v(a)$.  
If we assume for the moment that $v(a) < \nu - \frac{1}{2(p-1)}$, then we must have 
$$v(g'(0)) \geq
\frac{2}{3}\nu + \frac{1}{3}v(a) - \frac{1}{3(p-1)} > v(a).$$  Since $v(a^2) > v(a)$, this means 
$v(2(r+s) + sa) > v(a)$, so $$v(r+s) = v(sa) = v(a)$$ ($v(s) = 0$, by Lemma \ref{Lfullwildspec}).  
Since $v(r+s) \in \ints$, we have
$$v(a) = v(r+s) \leq \nu - 1.$$  

If instead, we assume that $v(a) \geq \nu - \frac{1}{2(p-1)}$, then 
$$v(g'(0)) \geq \frac{2}{3}\nu + \frac{1}{3}v(a) - \frac{1}{3(p-1)} > \nu - 1.$$  
So $v(2(r + s)) = v(r+s) \geq \nu$ by (\ref{Ethisone}).

It remains to show that we cannot have both $v(a) \geq \nu - \frac{1}{2(p-1)}$ and $v(r+s) \geq \nu$.  
Suppose, for a contradiction, that this is the case.  Then, multiplying $g(z)$ by
$\left(\frac{z-\sqrt{1-a}}{z + 1}\right)^{r+s}$, which is a $p^{\nu}$th power, we obtain the alternative equation
$$y^{p^{\nu}} = \left(\frac{z+ \sqrt{1-a}}{z+1}\right)^s\left(\frac{z- \sqrt{1-a}}{z-1}\right)^r$$ to (\ref{Efstr2}).
Consider the unique component $\ol{V}$ of $\ol{Z}^{str}$ above $\ol{W}$, the component of $\ol{X}$ corresponding to 
the disk of radius $|a|$ around $x=0$.  Then $\ol{V}$ corresponds to the coordinate
$z$.  The formal completion of $\ol{V} \backslash \{z = \pm 1\}$ in $Z^{str}$ is isomorphic to $\Spec C$ where 
$$C := R\{(z-1)^{-1}, (z+1)^{-1}\}.$$  
We have $$\frac{z+ \sqrt{1-a}}{z+1} = 1 + (\sqrt{1-a} - 1)(z+1)^{-1}.$$  Since 
$$v(\sqrt{1-a} - 1) = v(a) > \nu - 1 + \frac{1}{p-1},$$ this is a $p^{\nu - 1}$st power in $C$ (which follows from the binomial expansion).  
Likewise, $\left(\frac{z- \sqrt{1-a}}{z-1}\right)$ is a $p^{\nu -1}$st power in $C$.
So $$\left(\frac{z+ \sqrt{1-a}}{z+1}\right)^s\left(\frac{z- \sqrt{1-a}}{z-1}\right)^r$$ is a $p^{\nu -1}$st power in $C$.
But this means that there are at least $p^{\nu - 1}$ irreducible components in the inverse image of $\ol{V} \backslash 
\{z = \pm 1\}$ in $\ol{Y}^{str}$, and thus that many irreducible components of $\ol{Y}^{str}$ above $\ol{V}$.  

Now, $\ol{W}$ is not a tail, so it is not an \'{e}tale component by Lemma \ref{Letaletail}.  So $\ol{W}$ must be a $p$-component.  Furthermore, it 
cannot intersect a $p^2$-component, because the inertia groups above the intersection point would have order divisible by $p^2$, and then there could 
not be $p^{\nu-1}$ irreducible components above $\ol{V}$.   So $\ol{Y}$ has $p^{\nu-1}$ irreducible components above $\ol{V}$, each a radicial extension of degree $p$.  Associated to each is one deformation datum.  It has 
three critical points: two at the intersection of $\ol{W}$ with outward-lying components, and one at the intersection of $\ol{W}$ with inward lying 
components.  By Lemma \ref{Lsigmaeff}, the first two critical points have invariants $3/2$ and $1/2$, and \cite[p.\ 998, (2)]{We:br} shows that 
the third has invariant $-1$. Since no multiplicative deformation datum can have $-1$ for an invariant, the deformation datum must be additive.
But this contradicts \cite[Proposition 2.8]{We:mc}, proving the lemma.
\end{proof}

\begin{remark}\label{Rwsplit}
Armed with the knowledge that $v(a) = v(r+s)$, we can run the argument of the second-to-last paragraph of the proof again to see
that $g(z)$ is a $p^{v(a) - 1}$st power in $C$, and thus there are at least $p^{v(a) - 1}$ irreducible components of $\ol{Y}^{str}$ above $\ol{V}$.
In particular, $\ol{W}$ is a $p^i$-component for $i \leq \nu - v(a) + 1$.
\end{remark}

\begin{lemma}\label{L0coeffval}
The valuation $v\left(\frac{g^{(i)}(0)}{i!}\right)$ is at least $v(a) - v(i).$
\end{lemma}

\begin{proof}
Note that $\frac{g^{(i)}(0)}{i!}$ is the
coefficient of $z^i$ in the Maclaurin series expansion of $g(z)$.  
Since $v(a) \geq v(a) - v(i)$, it suffices to look modulo $a$.  By (\ref{Efstr2}), $g(z)$ is
congruent (mod $a$) to 
\begin{equation}\label{Emessy1}
\left(\frac{z+1}{z-1}\right)^{r+s} = \left(1 + \frac{2}{z-1}\right)^{r+s} = (- 1 - 2z - 2z^2 - 2z^3 - 
\cdots)^{r+s}.
\end{equation}
Expanding out the above expression gives 
\begin{equation}\label{Emessy2}
(-1)^{r+s}\left(1 + \sum_{i=1}^{\infty} \sum_{\substack{I = \{i_1, \ldots, i_q\} \subset \nats \\ A = (a_1, \ldots, a_q) \in \nats^q \\ \sum_{j=1}^q a_j \leq r+s \\ \sum_{j=1}^q a_ji_j = 
i}} 2^{|\sum_{j=1}^q a_j|} \binom{r+s}{a_1, \ldots, a_q, r+s - \sum_{j=1}^q a_j}z^i\right)
\end{equation}
(the contributions to the $z^i$ term come from taking $a_j$ different $z^{i_j}$ terms for $j = 1$ to $q$).

Now, if $\sum_{j=1}^q a_j i_j = i$, then there exists $j$ such that $v(a_j) \leq v(i)$.  By Lemma \ref{Lbinomcoeffs}, the coefficient of $z^i$ has 
valuation at least $v(r+s) - v(i)$, which is $v(a) - v(i)$, by Lemma \ref{Lvaint}.
\end{proof}

Recall that $c_i := \frac{g^{(i)}(0)}{i!}e^i$.
\begin{corollary}\label{C0coeffval}
For $i > 3$, we have $v(c_i) > \nu + \frac{1}{p-1}$, unless $p = i = 5$ and $v(a) = \nu - 1$. 
\end{corollary}

\begin{proof}
By Lemma \ref{L0coeffval}, $v(c_i) \geq iv(e) + v(a) - v(i)$.
By Lemma \ref{Ltailradius} (ii), $v(e) = \frac{1}{3}(\nu + \frac{1}{p-1} - v(a))$.
So
\begin{align}
v(c_i) \geq&\  \frac{i}{3}\left(\nu + \frac{1}{p-1} - v(a)\right) + v(a) - v(i) \\
=&\  \nu + \frac{1}{p-1} + \frac{i-3}{3}\left(\nu + \frac{1}{p-1} - v(a)\right) - v(i). \label{Erushed0}
\end{align}
Therefore, $v(c_i) > \nu + \frac{1}{p-1}$ whenever $\frac{(i-3)}{3}(\nu + \frac{1}{p-1} - v(a)) > v(i)$.  
By Lemma \ref{Lvaint}, $\nu - v(a) \geq 1$.
One checks that $\frac{(i-3)}{3}(\nu + \frac{1}{p-1} - v(a)) > v(i)$ 
always holds, except when $p = i =  5$ and $v(a) = \nu - 1$.  
This proves the corollary.
\end{proof}

\begin{remark}\label{R0c5}
In the case $p = 5$ and $v(a) = \nu - 1$, one uses (\ref{Emessy2}) to see that 
$\frac{g^{(5)}(0)}{5!}$ is congruent to $32\binom{r+s}{5} \pmod{a}$ (this is the term where $I = \{1\}$ and $A = (5)$; all 
other terms are $0$ modulo $a$).  Thus $c_5$ is equal to $32\binom{r+s}{5}e^5$ plus terms of valuation greater than $v(ae^5) = \nu 
+ \frac{13}{12} > \nu + \frac{1}{4} = \nu + \frac{1}{p-1}$.
Also, $v(c_5) = \nu + \frac{1}{12}$.  
\end{remark}

\begin{prop}\label{Pa0}
Suppose $v(a) > 0$. 
\begin{enumerate}
\item If $p > 5$ or $v(a) < \nu - 1$, then 
the disk $\Delta$ corresponding to $\ol{X}_b$ contains the $K_0$-rational point $x = a_0$, where 
$a_0 = 1 - \frac{s^2}{r^2}$.
\item If $p = 5$ and $v(a) = \nu -1$, then $\Delta$ contains the point 
$x = a_0$, where $a_0 = 1 - \left(\frac{s - \sqrt[5]{5^{4\nu +1}\binom{r+s}{5}}}{r}\right)^2$ (for all choices of 5th root).
\end{enumerate}
\end{prop}

\begin{proof}
\emph{To (1)}: By Corollary \ref{C0coeffval}, $v(c_i) > \nu + \frac{1}{p-1}$ for all $i$ such that $p | i$.
By Lemma \ref{Lartinschreier} (1), the torsor given by (\ref{Edisk2}) can split into $p^{\nu -1}$ disjoint $\mu_p$-torsors
with \'{e}tale reduction only if $v(c_1) \geq \nu + \frac{1}{p-1}$. 
In particular, we must have
$$v(g'(0)) = v(2r + \frac{2s}{\sqrt{1-a}}) \geq \nu + \frac{1}{p-1} - v(e) = 
\frac{2}{3}(\nu + \frac{1}{p-1}) + \frac{1}{3}v(a).$$  By Lemma \ref{Lk0root}, 
$a_0 = 1 - \frac{s^2}{r^2} \in K_0$ satisfies
$v(a-a_0) \geq \frac{2}{3}(\nu + \frac{1}{p-1}) + \frac{1}{3} v(a)$.  
But by Lemma \ref{Ltailradius}, $v(\rho) = \frac{2}{3}(\nu + \frac{1}{p-1}) + \frac{1}{3} v(a)$, where 
$|\rho|$ is the radius of $\Delta$.  So $a_0 \in \Delta$.

\emph{To (2)}: Let $c_5' = 32\binom{r+s}{5}e^5$.  By Remark \ref{R0c5}, $v(c_5' - c_5) > \nu + \frac{1}{4}$ and $v(c_5) > n$.
Also, Corollary \ref{C0coeffval} shows that $v(c_i) > \nu + \frac{1}{4}$ for all $i > 5$ such that $5 | i$.  
By Lemma \ref{Lartinschreier} (2), the torsor given by (\ref{Edisk2}) can split into $5^{\nu -1}$ disjoint $\mu_5$-torsors
with \'{e}tale reduction only if $$v\left(c_1 - \sqrt[5]{32 \cdot 5^{4\nu + 1}\binom{r+s}{5}e^5}\right) \geq \nu + \frac{1}{4}.$$ 

In particular, we must have
\begin{equation}\label{E05}
v\left(g'(0) - \sqrt[5]{32 \cdot 5^{4\nu + 1}\binom{r+s}{5}}\right) \geq \nu + \frac{1}{4} - v(e) = 
\nu - \frac{1}{6}.
\end{equation}  
Recall that $g'(0) = 2r + \frac{2s}{\sqrt{1-a}}$.
Since $$v\left((1 - \frac{1}{\sqrt{1-a}})\sqrt[5]{32 \cdot 5^{4\nu + 1}\binom{r+s}{5}}\right) = v(a) + \nu - \frac{1}{5} >
\nu - \frac{1}{6}$$ (as $v(a) \geq 1$), Equation (\ref{E05}) is equivalent to 
$$v\left(2r + \frac{2s - \sqrt[5]{32 \cdot 5^{4\nu + 1}\binom{r+s}{5}}}{\sqrt{1-a}}\right) \geq \nu - \frac{1}{6}.$$
By Lemma \ref{Lk0root}, we have $v(a - a_0) \geq \nu - \frac{1}{6}$, where $a_0 = 1 - \left(\frac{s - \sqrt[5]{5^{4\nu +1}\binom{r+s}{5}}}{r}\right)^2$.
But $v(\rho) = \nu - \frac{1}{6}$, so $a_0$ specializes to $\ol{X}_b$, and we can take $a = a_0$.
\end{proof}

\begin{remark}\label{R0conductor3}
As in Remark \ref{Rnot01conductor3}, one shows that if $a$ is as in Proposition \ref{Pa0}, then
the torsor given by (\ref{Edisk2}) indeed splits into $p^{\nu-1}$ disjoint $\mu_p$-torsors with
\'{e}tale reduction birational to an Artin-Schreier cover branched at one point of conductor 3.  
Also, if we are in case (1) of Proposition \ref{Pa0} and $a = 1 - \frac{s^2}{r^2}$, 
then since $v(g'(0)) > 0$, we must choose $\sqrt{1-a} = -\frac{s}{r}$.
\end{remark}

\paragraph{The case $v(a-1) > 0$}  
This case will be quite parallel to the $v(a) > 0$ case.  
We have $v(e) = \frac{1}{3}(\nu + \frac{1}{p-1} + v(1-a))$ by Lemma \ref{Ltailradius}.
We claim that $x=1$ is branched of index strictly less than $p^{\nu}$.  Indeed, if $x=1$ is branched of index $p^{\nu}$,
its specialization $\ol{1}$ would lie on $\ol{X}_0$ by Lemma \ref{Lwildspec}.  But $\ol{1}$ must be a smooth
point of $\ol{X}$.  So $\ol{1}$ corresponds to an open disk of radius $1$.   Since $v(a-1) > 0$, then $a$ would specialize to $\ol{1}$
as well, contradicting the fact that $a$ specializes to an \'{e}tale tail.  Let $p^{\nu_1} < p^{\nu}$ be the branching index of $x=1$.  
Then we know $v(s) = \nu - \nu_1$.  

Since the specializations of $1$ and $a$ cannot collide on $\ol{X}$, we must have a
component $\ol{W}$ of $\ol{X}$ corresponding to the disk of radius $|1-a|$ centered at $1$ (or equivalently, at
$a$).  

The next two lemmas play the role of Lemma \ref{Lvaint} for the case $v(a-1) > 0$.
\begin{lemma}\label{Lvsmall}
We have $v(1-a) \leq 2(\nu - 1+ \frac{1}{p-1})$.
\end{lemma}

\begin{proof}
Let $Q \leq G^{str}$ be the unique subgroup of order $p^{\nu_1}$.  Consider the cover 
$(f^{st})': (Y^{str})^{st}/Q \to X^{st}$, with generic fiber $f'$.  
Consider the path $\{v_i\}_{i=0}^{j}$, $\{e_i\}_{i=0}^{j-1}$ in $\mc{G}$, where $v_0$ corresponds to $\ol{X}_0$ and $v_j$ corresponds to $\ol{W}$.
Write $(\sigma^{\eff}_i)'$ (resp.\ $(\delta^{\eff}_i)'$) for the effective invariant at $e_i$ for $(f^{st})'$ (resp.\ the effective different at $v_i$ for $(f^{st})'$).  Then $(\delta^{\eff}_0)' = \nu - \nu_1 + \frac{1}{p-1}$.

Since no branch point of $f'$ with index divisible by $p$, and
only one branch point with index 2, specializes outward from the point corresponding to any $e_i$, Lemma 
\ref{Lsigmaeff} shows that $(\sigma_i^{\eff})' - 1$ is a sum of elements of the form 
$\sigma-1$, where $\sigma > 0$, $\sigma \in \frac{1}{2}\ints$, and $\sigma \in \ints$ for all but one term in the sum.  
Therefore, $(\sigma_i^{\eff})' - 1 \geq -\frac{1}{2}$, so $(\sigma_i^{\eff})' \geq \frac{1}{2}$.  

By Lemma \ref{Lwildspec} and monotonicity, $x=1$ specializes to a component which intersects a component which is 
inseparable for $(f^{st})'$.  In particular, since $x=1$ specializes on or outward from $\ol{W}$, it
must be the case that any component of $\ol{X}$ lying inward from $\ol{W}$ is inseparable for $(f^{st})'$.  Then, we
can apply Lemma \ref{Lsigmaeffcompatibility} to each $e_i$, $0 \leq i < j$, to show
that $$\nu - \nu_1 + \frac{1}{p-1} - (\delta^{\eff}_j)' \geq
\frac{1}{2}v(1-a).$$  Since $(\delta^{\eff}_j)' \geq 0$ and $\nu_1 \geq 1$, this yields 
$v(1-a) \leq 2(\nu - 1 + \frac{1}{p-1})$.
\end{proof}

\begin{lemma}\label{Lvcorrect}
We have $v(\sqrt{1-a}) = v(s)$.  In particular, $v(\sqrt{1-a}) \in \ints$.
\end{lemma}
\begin{proof}
As in the case $v(a) > 0$, we must have that $v(g'(0)e) \geq \nu$ (see beginning of proof of Lemma \ref{Lvaint}), so 
$$v(g'(0)) \geq \frac{2}{3}\nu - \frac{1}{3}(\frac{1}{p-1} + v(1-a)).$$  
Since $v(1-a) \leq 2(\nu - 1 + \frac{1}{p-1})$ (Lemma
\ref{Lvsmall}), we see that $$v(g'(0)) \geq \frac{2}{3} - \frac{1}{p-1} > 0.$$  Recall that 
$g'(0) = 2r + \frac{2s}{\sqrt{1-a}}.$  Since $v(2r) = 0$ ($f$ is totally ramified above $x=\infty$), it follows that
$v(\frac{2s}{\sqrt{1-a}}) = 0$.  Therefore $v(\sqrt{1-a}) = v(2s) = v(s)$.  
\end{proof}

The following lemma plays the role of Lemma \ref{L0coeffval} when $v(a-1) > 0$.
\begin{lemma}\label{L1coeffval}
The valuation $v\left( \frac{g^{(i)}(0)}{i!} \right) \geq (1-i)v(\sqrt{1-a}) - v(i)$.
\end{lemma}

\begin{proof}
Again, $\frac{g^{(i)}(0)}{i!}$ is the coefficient of $z^i$ in the Maclaurin series expansion of $g(z)$.  
Since $(1-i)v(\sqrt{1-a}) - v(i) \leq 0$, it suffices to look at coefficients modulo $R$ (as an $R$-submodule of $K$).
Then, $g(z)$ is congruent (mod $R$) to 
\begin{equation}\label{Emessy3}
\left(\frac{z+\sqrt{1-a}}{z-\sqrt{1-a}}\right)^{s} = 
\left(1 + \frac{2}{\frac{z}{\sqrt{1-a}}-1}\right)^{s} = (- 1 - 2w - 2w^2 - 2w^3 - 
\cdots)^{s},
\end{equation} 
where $w = \frac{z}{\sqrt{1-a}}$.
This is analogous to (\ref{Emessy1}), and we conclude as we do in Lemma \ref{L0coeffval} that
the coefficient of $w^i$ has valuation at least $v(s) - v(i)$.  Thus the coefficient of $z^i$ has valuation 
at least $v(s) - v(i) - iv(\sqrt{1-a})$.  By Lemma \ref{Lvcorrect}, this is $(1-i)v(\sqrt{1-a}) - v(i)$.
\end{proof}

Recall that $c_i := \frac{g^{(i)}(0)}{i!}e^i$.  Parallel to Corollary \ref{C0coeffval}, we have:
\begin{corollary}\label{C1coeffval}
For $i > 3$, we have $v(c_i) > \nu + \frac{1}{p-1}$, unless $p = i = 5$ and $v(\sqrt{1-a}) = \nu - 1$. 
\end{corollary}

\begin{proof}
By Lemma \ref{L1coeffval}, $v(c_i) \geq iv(e) + (1-i)v(\sqrt{1-a}) - v(i)$.
By Lemma \ref{Ltailradius} (3), $v(e) = \frac{1}{3}(\nu + \frac{1}{p-1} + v(1 - a))$.
So \begin{align}
v(c_i) \geq&\  \frac{i}{3}\left(\nu + \frac{1}{p-1} + v(1 - a)\right) + (1-i)v(\sqrt{1-a}) - v(i) \\
=&\  \nu + \frac{1}{p-1} + \frac{i-3}{3}\left(\nu + \frac{1}{p-1} - v(\sqrt{1-a})\right) - v(i). \label{Erushed1} 
\end{align}
Therefore, $v(c_i) > \nu + \frac{1}{p-1}$ whenever $\frac{(i-3)}{3}(\nu + \frac{1}{p-1} - v(\sqrt{1-a})) > v(i)$.  
By Lemma \ref{Lvsmall}, $\nu - v(\sqrt{1-a}) \geq 1$.
One checks that $\frac{(i-3)}{3}(\nu + \frac{1}{p-1} - v(\sqrt{1-a})) \geq v(i)$ 
always holds, except when $p = i = 5$ and $v(\sqrt{1-a}) = \nu - 1$.  
This proves the corollary.
\end{proof}

Parallel to Remark \ref{R0c5}, we have:
\begin{remark}\label{R1c5}
In the case $p = 5$ and $v(s) = v(\sqrt{1-a}) = \nu - 1$, one sees (using (\ref{Emessy2}) and (\ref{Emessy3})) that
$\frac{g^{(5)}(0)}{5!}$ is 
congruent to $\left(\sqrt{1-a}\right)^{-5} 32\binom{s}{5}$ modulo ${\left(\sqrt{1-a}\right)^{-5}sR}$ 
(viewed as an $R$-submodule of $K$).  That is, $\frac{g^{(5)}(0)}{5!}$ is equal to 
$\left(\sqrt{1-a}\right)^{-5} 32\binom{s}{5}$ 
plus terms of valuation more than $-4v(\sqrt{1-a})$, which is  $4(1-\nu)$.
So $c_5$ is equal to $\left(\sqrt{1-a}\right)^{-5} 32\binom{s}{5}e^5$ plus terms of valuation greater than 
$4 - 4\nu + 5v(e) = \nu + \frac{13}{12} > \nu + \frac{1}{4} = \nu + \frac{1}{p-1}$.
Also, $v(c_5) = \nu + \frac{1}{12}$.  
\end{remark}

The following proposition plays the role of Proposition \ref{Pa0} when $v(a-1) > 0$.
\begin{prop}\label{Pa1}
Suppose $v(a-1) > 0$. 
\begin{enumerate}
\item If $p > 5$ or $v(\sqrt{1-a}) < \nu - 1$, then 
the disk $\Delta$ corresponding to $\ol{X}_b$ contains the $K_0$-rational point $x=a_0$, where $a_0 = 1 - \frac{s^2}{r^2}$.
\item If $p = 5$ and $v(\sqrt{1-a}) = \nu -1$, then $\Delta$ contains the point $x=a_0$, where
$a_0 = 1 - \left(\frac{s - \sqrt[5]{5^{4\nu +1}\binom{s}{5}}}{r}\right)^2$.
\end{enumerate}
\end{prop}

\begin{proof}
\emph{To (1)}: By Corollary \ref{C1coeffval}, $v(c_i) > \nu + \frac{1}{p-1}$ for all $i$ such that $p | i$.
By Lemma \ref{Lartinschreier} (1), the torsor given by (\ref{Edisk2}) can split into $p^{\nu -1}$ disjoint $\mu_p$-torsors
with \'{e}tale reduction only if $v(c_1) \geq \nu + \frac{1}{p-1}$. 
In particular, we must have
$$v(g'(0)) = v(2r + \frac{2s}{\sqrt{1-a}}) \geq \nu + \frac{1}{p-1} - v(e) = 
\frac{2}{3}(\nu + \frac{1}{p-1} - v(\sqrt{1-a})).$$  By Lemma \ref{Lk0root},
$a_0 = 1 - \frac{s^2}{r^2} \in K_0$ satisfies
$v(a-a_0) \geq \frac{2}{3}(\nu + \frac{1}{p-1} + v(1-a))$.  
But by Lemma \ref{Ltailradius}, $v(\rho) = \frac{2}{3}(\nu + \frac{1}{p-1} + v(1-a))$, where 
$|\rho|$ is the radius of $\Delta$.  So $a_0 \in \Delta$.

\emph{To (2)}: Let $c_5' = \left(\sqrt{1-a}\right)^{-5} 32\binom{s}{5}e^5$.  By Remark \ref{R1c5}, $v(c_5' - c_5) > \nu + \frac{1}{4}$ and $v(c_5) > n$.
Also, Corollary \ref{C1coeffval} shows that $v(c_i) > \nu + \frac{1}{4}$ for all $i > 5$ such that $5 | i$.  
By Lemma \ref{Lartinschreier} (2), the torsor given by (\ref{Edisk2}) can split into $5^{\nu -1}$ disjoint $\mu_5$-torsors
with \'{e}tale reduction only if 
$$v\left(c_1 - \sqrt[5]{\left(\sqrt{1-a}\right)^{-5} 32 \cdot 
5^{4\nu + 1}\binom{s}{5}e^5}\right) \geq \nu + \frac{1}{4}.$$ 

In particular, we must have
$$v\left(g'(0) - \sqrt[5]{\left(\sqrt{1-a}\right)^{-5} 32 \cdot 5^{4\nu + 1}\binom{s}{5}}\right) \geq \nu + \frac{1}{4} - v(e) = 
\frac{5}{6},$$
which can be rewritten as 
$$v\left(2r + \frac{2s - \sqrt[5]{32 \cdot 5^{4\nu + 1}\binom{s}{5}}}{\sqrt{1-a}}\right) \geq \frac{5}{6}.$$
By Lemma \ref{Lk0root}, we have $v(a - a_0) \geq 2\nu - \frac{7}{6}$, where $a_0 = 1 - \left(\frac{s - \sqrt[5]{5^{4\nu +1}\binom{s}{5}}}{r}\right)^2$.
But $v(\rho) = 2\nu - \frac{7}{6}$, so $a_0 \in \Delta$.
\end{proof}

\begin{remark}\label{R1conductor3}
As in Remarks \ref{Rnot01conductor3} and \ref{R0conductor3}, one shows that if $a$ is as in Proposition \ref{Pa1}, then
the torsor given by (\ref{Edisk2}) indeed splits into $p^{\nu-1}$ disjoint $\mu_p$-torsors with
\'{e}tale reduction birational to an Artin-Schreier cover branched at one point of conductor $3$.  
Also, if we are in case (1) of Proposition \ref{Pa1} and $a = 1 - \frac{s^2}{r^2}$, 
then since $v(g'(0)) > 0$, we must choose $\sqrt{1-a} = -\frac{s}{r}$.
\end{remark}

\subsubsection{The inseparable tails}\label{Sinseparabletails}
For the rest of the proof, we replace $a$ with the $a_0$ calculated in Proposition \ref{Panot01}, \ref{Pa0}, or \ref{Pa1} of \S\ref{Setaletails},
depending on the congruency class of $a$ (see the paragraph preceding Lemma \ref{Lfullbranch}).
Maintain the notations of \S\ref{Setaletails}.  Throughout \S\ref{Sinseparabletails}, we assume that $\nu > 1$ 
(if $\nu = 1$, there can be no inseparable tails by Lemma \ref{Ltailetale}).  Our goal is to calculate exactly where the inseparable tails of
$\ol{X}$ lie.  We first place restrictions on what kinds of inseparable tails we can have 
and where they lie (Lemma \ref{Loneinsep}, Proposition \ref{Ponly0insep}), and then explicitly exhibit an inseparable tail in each allowable situation
(Proposition \ref{Pinseptailloc}).

The next lemma shows us that we are not looking for too many tails.
 
\begin{lemma}\label{Loneinsep}
Suppose $\ol{X}$ has a new inseparable $p^j$-tail $\ol{X}_c$.
\begin{enumerate}
\item The only new inseparable $p^j$-tail is $\ol{X}_c$.  Its effective ramification invariant is $\sigma_c = 2$.
\item There are two $p^j$-components $\ol{X}_{\beta}$ and $\ol{X}_{\beta'}$ of $\ol{X}$ that 
intersect $p^{j+1}$-components and have non-integral effective ramification invariant (see Definition \ref{Draminvariant}).  
We have $\sigma_{\beta} = \sigma_{\beta'} = \frac{1}{2}$.  Also, up to switching
indices $\beta$ and $\beta'$, we have $\ol{X}_0 \prec \ol{X}_{\beta} \prec \ol{X}_b$ and $\ol{X}_0 \prec \ol{X}_{\beta'} 
\prec \ol{X}_{b'}$.
\item  If $v(a) > 0$, then $j \leq \nu - v(a)$.  
If $v(a-1) > 0$, then $j \leq \nu - v(\sqrt{1-a}) - 1$.  
\end{enumerate}
\end{lemma}

\begin{proof} \emph{To (1) and (2):} 
Let $\{\ol{X}_{\alpha}$, $\alpha \in A\}$ be the set of all $p^j$-components that intersect $p^{j'}$-components with $j' > j$.  
By Lemma \ref{Ltailetale}, this includes all the $p^j$-tails.
Also, for $i > 0$, let $\Pi_i$ be the set of all branch points of $f$ (equivalently $f^{str}$) with branching index divisible by $p^i$.  As a consequence
of \cite[Proposition 3.17]{Ob:vc} (setting the ``$\alpha$" of \cite{Ob:vc} equal to our $j$) and monotonicity, 
we have the following equation relating the effective ramification invariants of the $\ol{X}_{\alpha}$
(see Definition \ref{Draminvariant}):
\begin{equation}\label{Egenvancycles}
|\Pi_{j+1}| - 2 = \sum_{\alpha \in A} (\sigma_{\alpha} - 1).
\end{equation} 

By definition, $\sigma_{\alpha} > 0$ for any $\ol{X}_{\alpha}$.  Furthermore, by \cite[Lemmas 4.1, 4.2 (i)]{Ob:vc}, any new inseparable \emph{tail} 
$\ol{X}_{\alpha}$ has invariant $\sigma_{\alpha} \geq 2$. 
By Lemmas \ref{Lnonintegral} and \ref{Lramdenominator}, any $\ol{X}_{\alpha}$ that has non-integral effective ramification invariant 
and is not a tail must lie between 
$\ol{X}_0$ and either $\ol{X}_b$ or $\ol{X}_{b'}$, and must have $\sigma_{\alpha} \geq \frac{1}{2}$.  
By monotonicity, there are at most two such components: Call them $\ol{X}_{\beta}$ and possibly
$\ol{X}_{\beta'}$.  
If there is a new inseparable $p^j$-tail $\ol{X}_c$, then the only way that (\ref{Egenvancycles}) can be satisfied is if $|\Pi_{j+1}| = 2$, if $\ol{X}_c$ is 
the only new $p^j$-tail (and $\sigma_c = 2$), and if $\ol{X}_{\beta}$ and $\ol{X}_{\beta'}$ both exist (with $\sigma_{\beta} = \sigma_{\beta'}  
= \frac{1}{2}$).  Since the conductor of a $\ints/p^a$-extension of $k[[t]]$ is always at least $p^{a-1}$ (see, e.g., \cite[Lemma 19]{Pr:lg}), 
we see that $\ol{X}_{\beta}$ (resp. $\ol{X}_{\beta'}$) intersects a $p^{j+1}$-component, because otherwise the invariant
$\sigma_{\beta}$ (resp.\ $\sigma_{\beta'}$) would be at least $\frac{p}{2}$. 

\emph{To (3):}
If $v(a) > 0$, recall that $\ol{W}$ is the component of $\ol{X}$ separating $0$ and $a$.  Assume that $\beta$ and $\beta'$ are as in
the statement of (2).
Then $\ol{W} \prec \ol{X}_{\beta}$ (and $\ol{W} \prec \ol{X}_{\beta'}$).  
Now, by Remark \ref{Rwsplit}, the order of generic inertia above $\ol{W}$ is at most $p^{\nu - v(a) + 1}$.   By monotonicity, 
$\ol{X}_{\beta}$ has less inertia than $\ol{W}$.  Since $\ol{X}_{\beta}$ is a $p^j$-component, we see that $j \leq \nu - v(a)$.

If $v(a-1) > 0$, then Lemma \ref{Lvcorrect} shows that $f^{str}$ is branched above $1$ of index $p^{\nu - v(\sqrt{1-a})}$. 
From the proof of (2), $|\Pi_{j + 1}| = 2$, which means that $1$ is branched of index at least $p^{j+1}$.  It follows that $j \leq \nu - v(\sqrt{1-a}) - 1$.
\end{proof}

Not only does Lemma \ref{Loneinsep} show us that we are not looking for too many tails, but it (in particular, part (2)) also gives information on 
the inseparable \emph{interior} components $\ol{X}_{\beta}$ and $\ol{X}_{\beta'}$ of $\ol{X}$.  The next lemma gives further information on the 
corresponding disks.

\begin{lemma}[cf.\ Lemma \ref{Ltailradius}]\label{Lnewinsepradius}
Suppose $\ol{X}$ has a new inseparable $p^j$-tail $\ol{X}_c$, and maintain the notation $\ol{X}_{\beta}$ from Lemma 
\ref{Loneinsep}.  Let $\ol{Z}_{\beta}$ be the unique component of $\ol{Z}^{str}$ above $\ol{X}_\beta$.  Let $e'$ (resp.\ $\rho'$) $\in \ol{K}$
be such that the radius
of the disk corresponding to $\ol{Z}_{\beta}$ (resp.\ $\ol{X}_{\beta}$) is $|e'|$ (resp. $|\rho'|$).

\begin{enumerate}
\item If $v(a) = v(a-1) = 0$, then $v(\rho') > \frac{2}{3}(\nu - j + \frac{1}{p-1})$ and  $v(e') > \frac{1}{3}(\nu - j + \frac{1}{p-1})$.
\item If $v(a) > 0$, then $v(\rho') > \frac{2}{3}(\nu - j + \frac{1}{p-1}) + \frac{1}{3}v(a)$ and 
$v(e') > \frac{1}{3}(\nu - j + \frac{1}{p-1} - v(a))$.
\item If $v(a-1) > 0$, then $v(\rho') > \frac{2}{3}(\nu - j + \frac{1}{p-1} - v(1-a))$ and $v(e') > \frac{1}{3}(\nu - j + \frac{1}{p-1} + v(1-a))$.
\end{enumerate}
\end{lemma}

\begin{proof}
We give only a sketch.
As in Lemma \ref{Ltailradius}, we have $v(e') = \frac{1}{2}(v(\rho') - v(a))$, so it suffices to prove the statements about $\rho'$. 

Let $Q < G^{str}$ be the unique subgroup of order $p^j$, and let $$(f^{st})': (Y^{str})^{st}/Q \to X^{st}$$ 
be the canonical map.  

Consider the path $\{v_i\}_{i=0}^{\ell}$, $\{e_i\}_{i=0}^{\ell-1}$, where $v_0$ corresponds to $\ol{X}_0$ and $v_{\ell}$ corresponds to 
$\ol{X}_{\beta}$.  Write
$(\sigma^{\eff}_i)'$ for the effective invariant for $(f^{st})'$ at $e_i$. 
The effective different for  $(f^{st})'$ at $v_0$ is $\nu - j + \frac{1}{p-1}$, whereas at $v_{\ell}$ it is 0.
Lastly, write $\Delta_i = 1$ (resp.\ $0$) if $\ol{X}_c$ lies (resp.\ does not lie) outward from the point corresponding to $e_i$.
In particular, $\Delta_{\ell - 1} = 0$, so $\Delta_i$ is not always 1.

\emph{To (1)} (cf.\ Proof of Lemma \ref{Ltailradius} (1)):
By Lemmas \ref{Lsigmaeff} and \ref{Loneinsep}, $(\sigma^{\eff}_i)' = \frac{1}{2} + \Delta_i$.  In 
particular, $\Delta_{\ell - 1} = 0$, so $(\sigma^{\eff}_i)' \leq \frac{3}{2}$ with equality not holding for some $i$.
By applying Lemma \ref{Lsigmaeffcompatibility} (3) to each $e_i$, $0 \leq i < \ell$, we obtain $v(\rho') > \frac{2}{3}(\nu - j + \frac{1}{p-1})$.

\emph{To (2):}
Recall that $\ol{W}$ is the component of $\ol{X}$ corresponding to the closed disk of radius $|a|$ and center $a$.  
Suppose $\ol{W}$ corresponds to $v_{i_0}$.  Then, for $i < i_0$, Lemma \ref{Lsigmaeff} shows that
$(\sigma^{\eff}_i)' = \Delta_i$.  For $i > i_0$, Lemma \ref{Lsigmaeff} shows that $(\sigma_i^{\eff})' = \frac{1}{2} + \Delta_i$. 
If $\Delta_i$ were equal to $1$ for all $i$, we would have $v(\rho') = \frac{2}{3}(\nu - j + \frac{1}{p-1}) + \frac{1}{3}v(a)$ (cf.\ Proof of Lemma \ref
{Ltailradius} (2)).  But since, for large enough $i$, we have $\Delta_i = 0$, we can see from
Lemma \ref{Lsigmaeffcompatibility} (3) that $v(\rho') > \frac{2}{3}(\nu - j + \frac{1}{p-1}) + \frac{1}{3}v(a)$.

\emph{To (3):} 
In this case, $\ol{W}$ is the component of $\ol{X}$ corresponding to the closed disk of radius $|1-a|$ and center $a$.  
Suppose $\ol{W}$ corresponds to $v_{i_0}$.  Then, for $i < i_0$, Lemma \ref{Lsigmaeff} shows that
$(\sigma^{\eff}_i)' = -\frac{1}{2} + \Delta_i$.  For $i > i_0$, Lemma \ref{Lsigmaeff} shows that $(\sigma_i^{\eff})' = \frac{1}{2} + \Delta_i$. 
If $\Delta_i$ were equal to $1$ for all $i$, we would have $v(\rho') = \frac{2}{3}(\nu - j + \frac{1}{p-1} + v(1-a))$ (cf.\ Proof of Lemma \ref
{Ltailradius} (3)).  But since, for large enough $i$, we have $\Delta_i = 0$, we can see from
Lemma \ref{Lsigmaeffcompatibility} (3) that $v(\rho') > \frac{2}{3}(\nu - j + \frac{1}{p-1} + v(1-a))$.
\end{proof}

Knowledge of the inertia groups above $\ol{X}_{\beta}$, $\ol{X}_{\beta'}$, as well as the disks to which they correspond, gives further
restrictions on the inseparable tails, as we see below.

\begin{prop}\label{Ponly0insep}
Suppose that $\ol{X}$ has a new inseparable $p^j$-tail.
Then $v(a) > 0$ or $v(a-1) > 0$.
Furthermore, if $v(a) > 0$, then either $j = \nu - v(a)$, or both $p = 5$ and $j = \nu - v(a) - 1$.
If $v(a-1) > 0$, then $p = 5$ and $j = \nu - v(\sqrt{1-a})-1$.
\end{prop}

\begin{proof}
Maintain the notation $\ol{X}_{\beta}$ and $\ol{Z}_{\beta}$ from Lemma \ref{Lnewinsepradius}.   
Let $e'$ be such that the radius of the disk corresponding to $\ol{Z}_{\beta}$ is $|e'|$, and let $t'$ be a coordinate on this disk.  
Let $Q < G^{str}$ be the unique subgroup of order $p^j$. 
Write $q: (Y^{str})' := Y/Q \to Z^{str}$ for the canonical map.  We can write the equation of $q$ in terms of 
$t'$ as $$y^{p^{\nu-j}} = g(e't') = 1 + \frac{g'(0)}{1!}(e't') +
\frac{g''(0)}{2!}(e't')^2 + \cdots.$$  We claim that, unless $j$ and $a$ are as in the statement of the proposition,
the right-hand side is a $p^{\nu - j}$th power in $\Frac(R\{t'\})$.  This means that $(Y^{str})^{st}/Q$ splits into $p^{\nu - j}$
irreducible components above $\ol{Z}_{\beta}$, each mapping isomorphically to $\ol{Z}_{\beta}$.  
This implies that $\ol{X}_{\beta}$ is a $p^j$-component for $f$ that does \emph{not} intersect a $p^{j+1}$-component, which contradicts 
the definition of $\ol{X}_{\beta}$.

Let $c_i'$ be the coefficient of $(t')^i$ in $g(e't')$. 
Then $c_i' = c_i(\frac{e'}{e})^i$, so $v(c_i') = v(c_i) - i(v(e) - v(e'))$, which is greater than $v(c_i) - \frac{i}{3}j$, by Lemmas \ref{Ltailradius} and 
\ref{Lnewinsepradius}.

In the case $v(a) = v(a-1) = 0$, it is clear from Lemma \ref{Lnewinsepradius} that
$v(c'_i) > \nu - j + \frac{1}{p-1}$ for $i \geq 3$.
Also, by Lemma \ref{Lartinschreier} (1), we have $v(c_i) \geq  \nu + \frac{1}{p-1}$ for $i = 1, 2$.  Then
$$v(c_i') > \nu + \frac{1}{p-1} - \frac{i}{3}j > \nu 
- j + \frac{1}{p-1}$$ for $i = 1, 2$.  Since $v(c_i') > \nu - j + \frac{1}{p-1}$ for all $i$, the binomial theorem shows that $g(e't')$ is a $p^{\nu - j}$th 
power in $\Frac(R\{t'\})$, finishing this case.

Now, suppose $v(a) > 0$.  Since $v(c_i') > v(c_i) - \frac{i}{3}j$, Equation (\ref{Erushed0}) shows that 
$$v(c_i') > \nu - j + \frac{1}{p-1} + \frac{i-3}{3} (\nu - j + \frac{1}{p-1} - v(a)) - v(i).$$ 
By Lemma \ref{Loneinsep} (3), $\nu - j - v(a) \geq 0$.  We obtain that $v(c_i') > \nu - j + \frac{1}{p-1}$ for $i \geq 3$, unless $j = \nu - v(a)$ or both
$j = \nu - v(a) - 1$ and $p = i = 5$.  Barring these possibilities, we conclude as in the case $v(a) = v(a-1) = 0$
that $v(c_i') > \nu - j  + \frac{1}{p-1}$ for $i = 1, 2$. 
Since $v(c_i') > \nu - j + \frac{1}{p-1}$ for all $i$, the binomial theorem shows that $g(e't')$ is a $p^{\nu - j}$th power in $\Frac(R\{t'\})$, 
finishing this case.

Lastly, if $v(a-1) > 0$, then since $v(c_i') > v(c_i) - \frac{i}{3}j$, Equation (\ref{Erushed1}) shows that 
$$v(c_i') > \nu - j + \frac{1}{p-1} + \frac{i-3}{3} (\nu - j + \frac{1}{p-1} - v(\sqrt{1-a})) - v(i).$$
By Lemma \ref{Loneinsep} (3),  $\nu - j - v(\sqrt{1-a}) \geq 1$.  We obtain that $v(c_i') > \nu - j + \frac{1}{p-1}$ for $i \geq 3$, unless $p=i=5$ and 
$j = \nu - v(\sqrt{1-a}) - 1$.  We conclude as in the case $v(a) > 0$.

\end{proof}

Proposition \ref{Ponly0insep} has narrowed down the possibilities for inseparable tails to the point that we can now explicitly exhibit 
every possible inseparable tail.

\begin{prop}\label{Pinseptailloc}
A new inseparable tail in fact exists in all cases allowed by Proposition \ref{Ponly0insep}.  In particular:
\begin{enumerate}
\item If $v(a) > 0$, then $\ol{X}$ has a new inseparable $p^{\nu - v(a)}$-tail
corresponding to the disk of radius $p^{-(v(a) + \frac{1}{(p-1)})}$ around $x = \frac{a}{2}$.
The two components of $\ol{Z}^{str}$ lying above correspond to the disks of radius $p^{-(\frac{1}{2(p-1)})}$ 
around $z = \pm\sqrt{-1}$.
\item If $v(a) > 0$, $p = 5$, and $v(a) < \nu - 1$, then $\ol{X}$ has a new inseparable $p^{\nu - v(a) - 1}$-tail 
corresponding to the disk of radius $p^{-(v(a) + \frac{17}{20})}$ around $x = \frac{a}{1-d^2}$, where 
$d := \pm \left(\frac{5^{v(a) + 1}}{r+s}\right)^{2/5}$ (for any choice of $5$th root).
The two components of $\ol{Z}^{str}$ lying above correspond to the disks of radius $p^{-\frac{17}{40}}$ 
around the two possible choices of $d$. 
\item If $v(a-1) > 0$, $p = 5$, and $v(\sqrt{1-a}) < \nu - 1$, then $\ol{X}$ has a new inseparable $p^{\nu - v(\sqrt{1-a}) - 1}$-tail 
corresponding to the disk of radius $p^{-(v(1-a) + \frac{17}{20})}$ around $x = \frac{a}{1-d^2}$, where 
$d := \pm 2\frac{s}{r}\left(\frac{5^{v(\sqrt{1-a}) + 1}}{s}\right)^{2/5}$ (for any choice of $5$th root).
The two components of $\ol{Z}^{str}$ lying above correspond to the disks of radius $p^{-(v(\sqrt{1-a}) + \frac{17}{40})}$ 
around the two possible choices of $d$.
\end{enumerate}
\end{prop}

Combining this with Lemma \ref{Loneinsep} (1), we obtain:
\begin{corollary}\label{Cinseptailloc}
The new inseparable tails mentioned in Proposition \ref{Pinseptailloc} are all the new inseparable tails of $\ol{X}$.
\end{corollary}

Before proving Proposition \ref{Pinseptailloc} we prove a lemma.  Recall that 
$$g(z) = \left(\frac{z+1}{z-1}\right)^r \left(\frac{z + \sqrt{1-a}}{z- \sqrt{1-a}}\right)^s.$$

\begin{lemma}\label{L5insep}
Suppose $p = 5$.

\begin{enumerate}
\item Suppose $v(a) > 0$, $v(a) < \nu - 1$, and $a = 1 - \frac{s^2}{r^2}$.  
Let $d, e'' \in R$ with $v(d) = \frac{2}{5}$ and $v(e'') = \frac{17}{40}$.  Then, if we expand
$g(d + e''t'')$ as a power series in $R\{t''\}$, we will have $g(d + e''t'') \equiv$ 
$$g(d) + (-1)^{r+s}\left(2(r+s)d^2e''t'' + 2(r+s)d(e''t'')^2 + \frac{32(r+s)}{5}(e''t'')^5\right)$$
modulo $5^{v(a) + \frac{5}{4} + \epsilon}R\{t''\}$ for some $\epsilon > 0$.  Furthermore, $$v(g(d) - (-1)^{r+s}) = v(a) + 1.$$
\item Suppose $v(a-1) > 0$, $v(\sqrt{1-a}) < \nu - 1$, and $a = 1 - \frac{s^2}{r^2}$.  
Let $d, e'' \in R$ with $v(d) = v(\sqrt{1-a}) + \frac{2}{5}$ and $v(e'') = v(\sqrt{1-a}) + \frac{17}{40}$.  Then, if we expand
$g(d + e''t'')$ as a power series in $R\{t''\}$, we will have
$g(d + e''t'') \equiv$ $$g(d) + (-1)^{r+s} \left(- 8\frac{r^3}{s^2}d^2e''t'' - 8 \frac{r^3}{s^2}d(e''t'')^2 - \frac{32r^5}{5s^4}(e''t'')^5\right)$$ 
modulo $5^{v(\sqrt{1-a}) +  \frac{5}{4} + \epsilon}R\{t''\}$ for some $\epsilon > 0$.  Furthermore, 
$$v(g(d) - (-1)^{r+s}) = v(\sqrt{1-a}) + 1.$$
\end{enumerate}
\end{lemma}

\begin{proof}
Write $\sqrt{1-a} = -\frac{s}{r}$ (Remarks \ref{R0conductor3}, \ref{R1conductor3}). 

\emph{To (1)}:
We can write
\begin{equation}\label{E501}
g(z) = \left(\frac{z+1}{z-1}\right)^{r+s}\left(\left(\frac{z+ \sqrt{1-a}}{z+1}\right)\left(\frac{z-1}{z-\sqrt{1-a}}\right)\right)^s.
\end{equation}
Recall (Lemma \ref{Lvaint}) that $v(r + s) = v(a)$.  We will calculate $g(d+e''t'')$ by first expanding
$g(z)$ as a power series, and then plugging in $z = d + e''t''$.  Note that, since $\min(v(d), v(e'')) = \frac{2}{5}$, we 
can essentially think of $z$ as having valuation $\frac{2}{5}$, and thus can ignore terms of the form $cz^i$ such that
$v(c) + \frac{2}{5}i > v(a) + \frac{5}{4}$.  Throughout, we use the notation $\sigma \approx \tau$ if, thinking of
$z$ as a generic element of $R$ having valuation $\frac{2}{5}$, we have $\sigma = \tau$ modulo terms with valuation 
$> v(a) + \frac{5}{4}$.  Since $v(a) \geq 1$, note that $a^2z \approx 0$.  This, as well as the fact that $v(a) = v(r+s)$ (Lemma \ref{Lvaint}), 
will be used repeatedly (without mention) below.

Using (\ref{Emessy1}), (\ref{Emessy2}), and simplifying, we can write
\begin{equation}\label{E502}
\left(\frac{z+1}{z-1}\right)^{r+s} \approx (-1)^{r+s}\left(1 + 2(r+s)z + \frac{8}{3}(r+s)z^3 + \frac{32}{5}(r+s)z^5\right).
\end{equation}

On the other hand, since $\sqrt{1-a} = -\frac{s}{r}$, if we let $\mu = \sqrt{1-a} - 1 = -\frac{r+s}{r}$, then $v(\mu) = v(a)$.   
We can write
\begin{align*}
\left(\left(\frac{z+ \sqrt{1-a}}{z+1}\right)\left(\frac{z-1}{z-\sqrt{1-a}}\right)\right)^s =&\  
\left((1 + \frac{\mu}{z+1})(1 + \frac{\mu}{z-1-\mu})\right)^s \\ \approx&\  1 + 2s\mu\frac{z}{z^2 -1},
\end{align*}
and thus 
\begin{equation}\label{E503}
\begin{split}
\left(\left(\frac{z+ \sqrt{1-a}}{z+1}\right)\right.&\left.\left(\frac{z-1}{z-\sqrt{1-a}}\right)\right)^s \approx \\ &  1 + 2\frac{s(r+s)}{r}(z + z^3)
\approx 1 - 2(r+s)(z+z^3).
\end{split}
\end{equation}
Combining (\ref{E501}), (\ref{E502}), and (\ref{E503}) yields
$$g(z) \approx (-1)^{r+s}\left(1 + \frac{2}{3}(r+s)z^3 + \frac{32}{5}(r+s)z^5\right).$$
Now, if we plug in $z = d + e''t''$ (and ignore all $t''$ terms where the coefficient has valuation $> v(a) + \frac{5}{4}$), 
then $z^3$ becomes $d^3 + 3d^2e''t'' + 3d(e''t'')^2$ and $z^5$ becomes $d^5 + (e''t'')^5$.  We obtain
the expression in the lemma.  

Since $v(\frac{32}{5}(r+s)d^5) = v(a) + 1$ and $v(\frac{2}{3}(r+s)d^3) = v(a) + \frac{6}{5}$, we 
see that $v(g(d) - (-1)^{r+s}) = v(a) +1$.

\emph{To (2):} 
Recall by Lemma \ref{Lvcorrect} that $v(s) = v(\sqrt{1-a})$.  
Let $w = \frac{z}{\sqrt{1-a}} = -\frac{r}{s}z$.
Again, we calculate $g(d+e''t'')$ by first expanding $g(z)$, and then plugging in $z = d + e''t''$.  
Since $\min(v(d), v(e'')) = v(\sqrt{1-a})+ \frac{2}{5}$, we 
can essentially think of $z$ as having valuation $v(\sqrt{1-a}) +\frac{2}{5}$ and $w$ as having valuation $\frac{2}{5}$. 
We will write $\sigma \approx \tau$ if, thinking of
$z$ (resp.\ $w$) as a generic element of $R$ having valuation $v(\sqrt{1-a}) + \frac{2}{5}$ (resp.\ $\frac{2}{5}$), 
we have $\sigma = \tau$ modulo terms with valuation $> v(\sqrt{1-a}) + \frac{5}{4}$.  Note also that $v(s) = v(\sqrt{1-a})$ (Lemma \ref{Lvcorrect}).

Using (\ref{Emessy2}), we have
\begin{equation}\label{E511}
\left(\frac{z+1}{z-1}\right)^r \approx (-1)^r(1 + 2rz)
\end{equation}
(note that $z^2 \approx 0$).

On the other hand, using (\ref{Emessy3}), (\ref{Emessy2}), and simplifying, we can write
\begin{equation}\label{E512}
\left(\frac{z+\sqrt{1-a}}{z-\sqrt{1-a}}\right)^{s} \approx (-1)^{s}\left(1 + 2sw + \frac{8}{3}sw^3 + \frac{32}{5}sw^5\right).
\end{equation}

Combining (\ref{E511}), (\ref{E512}), and substituting $w = -\frac{r}{s}z$ yields
\begin{equation}\label{E513}
g(z) \approx (-1)^{r+s}\left(1 - \frac{8r^3}{3s^2}z^3 -\frac{32r^5}{5s^4}z^5\right).
\end{equation}
If we plug in $z = d + e''t''$ (and ignore all $t''$ terms where the coefficient has valuation $> v(\sqrt{1-a}) + \frac{5}{4}$), 
then $z^3$ becomes $d^3 + 3d^2e''t'' + 3d(e''t'')^2$ and $z^5$ becomes $d^5 + (e''t'')^5$.  We obtain
the expression in the lemma.   

Since $v(\frac{32r^5}{5s^4}d^5) = v(\sqrt{1-a}) + 1$ and 
$v(\frac{8r^3}{3s^2}d^3) = v(\sqrt{1-a}) + \frac{6}{5}$, we 
see that $v(g(d)- (-1)^{r+s}) = v(\sqrt{1-a}) +1$.
\end{proof}

We now prove Proposition \ref{Pinseptailloc}:
\begin{proof}
Since $z^2 = \frac{x-a}{x}$, it is easy to check (as in Lemma \ref{Ltailradius}) that it suffices to prove the statements about $\ol{Z}^{str}$.

Say that we wish to show that a disk $\mc{D}$ in $\proj^1_z$ corresponds to a component of $\ol{Z}^{str}$
lying above a new inseparable $p^j$-tail $\ol{X}_c$ of $\ol{X}$.  
Let $Q < G^{str}$ be the unique subgroup of order $p^j$.
If $\hat{Y}$, $\hat{Z}$ are the formal completions of $(Y^{str})^{st}/Q$ and $(Z^{str})^{st}$
along their special fibers, then we claim that it suffices to show that the generic fiber of the torsor 
$\hat{f}: \hat{Y} \times_{\hat{Z}} \mc{D} \to \mc{D}$ splits into 
$p^{\nu - j - 1}$ $\mu_p$-torsors, each of which has \'{e}tale reduction with conductor $2$. 

We prove the claim.  Suppose $\mc{D}$ is such that $\hat{f}$ splits as desired.
Since an Artin-Schreier cover of conductor $2$ has genus $\frac{p-1}{2} > 0$, 
then \cite[Lemma 4.3]{Ob:fm} shows that $\ol{X}_c$ is contained in the stable reduction of $\ol{X}$.  
Since the proposed disks corresponding to $\ol{X}_c$ do not contain $x=0$, $1$, $a$ or $\infty$, it follows that
no branch point specializes to or outward from $\ol{X}_c$.  So either $\ol{X}_c$ is a new inseparable tail,
or there exists a new inseparable tail lying outward from $\ol{X}_c$.  In cases (2) and (3), 
no new inseparable tail can lie outward from $\ol{X}_c$, by Proposition \ref{Ponly0insep}, 
Lemma \ref{Ltailetale}, and monotonicity.  In case (1), any
new inseparable tail lying outward from $\ol{X}_c$ would have to be one of the tails in (2) or (3) (again using Proposition
\ref{Ponly0insep}, Lemma \ref{Ltailetale}, and monotonicity), 
and inspection shows that this is not the case.  Thus $\ol{X}_c$ is an 
inseparable tail, proving the claim.

It remains to show that, for the disks $\mc{D}$ in the proposition, $\hat{f}$ splits as desired.
Let $z=d$ be a center of $\mc{D}$, and enlarge $K$ (if necessary) so that $K$ contains an element $e''$ 
such that $|e''|$ is the radius of $\mc{D}$.  Then we can 
choose a coordinate $t''$ on $\mc{D}$ such that $z = d + e''t''$.  Enlarge $K$ again (if necessary) so that $g(d) \in 
(K^{\times})^{p^\nu}$.
The generic fiber of $\hat{f}$ can be given by the equation
$$y^{p^{\nu - j}} = \frac{g(d + e''t'')}{g(d)}.$$

\emph{To (1):}
Here $d = \sqrt{-1}$ (either square root) and $e'' \in K$ with $v(e'') = \frac{1}{2(p-1)}$.
Also, $j = \nu - v(a)$, so $\nu - j = v(a) = v(r+s)$, by Lemma \ref{Lvaint}.  So we may multiply $\frac{g(z)}{g(d)}$ by 
$p^{v(r+s)}$th powers without changing the generic fiber of $\hat{f}$.  By (\ref{E501}), we may assume that the generic fiber of $\hat{f}$ is given by
$$y^{p^{v(a)}} = \left(\left(\frac{z+ \sqrt{1-a}}{z+1}\right)\left(\frac{z-1}{z-\sqrt{1-a}}\right)\right)^s.$$
Let $\mu = -\frac{r+s}{r} = \sqrt{1-a} - 1$.  Then $v(\mu) = v(a)$.  Thus we can write 
\begin{equation}\label{Ebetterg}
y^{p^{v(a)}} = h(z) = \left((1 + \frac{\mu}{z+1})(1 + \frac{\mu}{z-1-\mu})\right)^s = 1 + 2s\mu\frac{z}{z^2 -1} +
O(\mu^2),
\end{equation}
where $O(\mu^2)$ signifies terms in $z$ whose coefficients have valuation at least $v(\mu^2)$.

After a possible finite extension, assume $h(d) \in (K^{\times})^{p^{v(a)}}.$
Thus we may replace $h(z)$ by $\frac{h(z)}{h(d)}$, without changing $\hat{f}$. 
Expanding $\frac{h(z)}{h(d)}$ out in terms of $t''$ gives
$$1 + \frac{h'(d)}{h(d)1!}e''t'' + \frac{h''(d)}{h(d)2!}(e''t'')^2 + \cdots.$$
For all $i > 0$, Equation (\ref{Ebetterg}) shows that 
$$v(\frac{h^{(i)}(d)}{i!h(d)}) \geq v(\mu) = v(a) = \nu - j.$$  
So for $i > 2$, $$v(\frac{h^{(i)}(d)}{i!h(d)}(e'')^i) > \nu - j + \frac{1}{p-1}.$$  
For $i = 1$, a direct calculation shows $\frac{h'(d)}{h(d)} = O(\mu^2)$, 
so $$v(\frac{h'(d)}{h(d)}e'') > 2v(\mu) = 2v(a) > \nu - j + \frac{1}{p-1}.$$  
For $i = 2$, a direct calculation shows $$h''(z) = 2s\mu\frac{2z}{(z^2-1)^3}\left(-4(1+z^2) + 2z(z^2-1)\right) + O(\mu^2).$$  Then
$v(h''(d)) = v(\mu) = \nu - j$.  Since $v(h(d)) = 0$, we have that 
$$v(\frac{h''(d)}{h(d)2!}(e'')^2) = \nu - j + \frac{1}{p-1}.$$  
By \cite[Lemma 3.1 (i)]{Ob:fm}, the generic fiber of $\hat{f}$ 
splits into $p^{v(a) - 1}$ $\mu_p$-torsors, each of which has \'{e}tale reduction with conductor 2.

\emph{To (2):} Here $d = \pm \left(\frac{5^{v(a)+1}}{r+s}\right)^{2/5}$ (any 5th root) and $e'' \in K$ with $v(e'') = \frac{17}{40}$.
Since $v(r+s) = v(a)$, then $v(d) = \frac{2}{5}$.  Also, $\nu - j = v(a) + 1$.  Then the generic fiber of $\hat{f}$ can be given
by $y^{p^{v(a) + 1}} = \frac{g(d+e''t'')}{g(d)}$.  By Lemma \ref{L5insep} (1), this is equivalent to
$$y^{p^{v(a) + 1}} \approx 1 + \left(2(r+s)d^2e''t'' + 2(r+s)d(e''t'')^2 + \frac{32(r+s)}{5}(e''t'')^5\right),$$
where ``$\approx$" means we have equality up to terms with coefficients of valuation $> v(a) + \frac{5}{4}$.  Note that, since
$v(g(d) - (-1)^{r+s}) = v(a) + 1$, dividing out by $g(d)$ is the same as dividing the coefficients of positive powers of $t''$ by 
$(-1)^{r+s}$, up to $\approx$.

The valuation of the coefficient of $(t'')^2$ is $v(a) + \frac{5}{4}$.  If $c_1''$ and $c_5''$ are the coefficients of $t''$ and $(t'')^5$, respectively,
then plugging in $d$ shows $c_5'' - \frac{(c_1'')^5}{5^{4v(a) + 5}} = 0$.  By \cite[Lemma 3.1 (ii)]{Ob:fm}, 
the special fiber of $\hat{f}$ splits into $p^{v(a)}$ $\mu_p$-torsors, each of which has \'{e}tale reduction with conductor $2$.

\emph{To (3):} Here $d = \pm 2\frac{s}{r}\left(\frac{5^{v(\sqrt{1-a})+1}}{s}\right)^{2/5}$ (any 5th root) and $e'' \in K$ with 
$v(e'') = v(\sqrt{1-a}) + \frac{17}{40}$.
Since $v(s) = v(\sqrt{1-a})$ by Lemma \ref{Lvcorrect}, then $v(d) = v(\sqrt{1-a}) + \frac{2}{5}$.  
Also, $\nu - j = v(a) + 1$.  Then the generic fiber of $\hat{f}$ can be given by $y^{p^{v(a) + 1}} = \frac{g(d+e''t'')}{g(d)}$.  
By Lemma \ref{L5insep} (2), this is equivalent to
$$y^{p^{v(a) + 1}} \approx   1 + \left(- 8\frac{r^3}{s^2}d^2e''t'' - 8 \frac{r^3}{s^2}d(e''t'')^2 - \frac{32r^5}{5s^4}(e''t'')^5\right),$$
where ``$\approx$" is as in (2).  As in (2), dividing out by $g(d)$ is the same as dividing the coefficients of positive powers of $t''$ by $(-1)^{r+s}$, up to 
$\approx$.

The valuation of the coefficient of $(t'')^2$ is $v(\sqrt{1-a}) + \frac{5}{4}$.  If $c_1''$ and $c_5''$ are the coefficients of $t''$ and $(t'')^5$, respectively,
then plugging in $d$ shows $$c_5'' - \frac{(c_1'')^5}{5^{4v(\sqrt{1-a}) + 5}} = (2^{25} - 2^5)\frac{r^5}{5s^4}(e'')^{5},$$ which
has valuation $v(\sqrt{1-a}) + \frac{25}{8} > v(\sqrt{1-a}) + \frac{5}{4}$.  
By \cite[Lemma 3.1 (i)]{Ob:fm}, the generic fiber of $\hat{f}$ 
splits into $p^{v(a) - 1}$ $\mu_p$-torsors, each of which has \'{e}tale reduction with conductor 2.
\end{proof}

\subsubsection{A field of definition of the stable model}\label{Sstablemodel}
We first determine a field of definition of $f^{str}$.  Recall that $f^{str}$ is branched at $0$, $1$, $\infty$, and $a$, and that
$K_{\nu} = K_0(\zeta_{p^{\nu}})$.

\begin{prop}\label{Pfstrdef}
The cover $f^{str}$ is defined (as a $G^{str}$-cover) over $K^{str} := K_{\nu}(a, \sqrt{1-a}) = K_{\nu}(a)$.
\end{prop}

\begin{proof}
The explicit equations (\ref{Efstr1}) and (\ref{Efstr2}) give the cover $f^{str}$.  So it is immediate that $f^{str}$ is 
defined as a mere cover (i.e., without the $G^{str}-action$) over $K_0(a, \sqrt{1-a})$.

Let $\alpha$ be a generator of $\ints/p^{\nu} \leq G^{str}$, let $\beta$ be an element of order $2$ in $G^{str}$, and let $\zeta_{p^{\nu}}$ be a 
$p^{\nu}$th root of unity.
Since $\alpha^*$ fixes $z$, Equation (\ref{Efstr2}) shows that $\alpha^*(y) = \zeta_{p^{\nu}}^iy$ for some $i \in \ints$.  
Also, Equation (\ref{Efstr1}) shows that $\beta^*(z) = -z$.  Then
we see that $\beta^*(g(z)) = g(z)^{-1}$. 
Thus $\beta^*(y) = \zeta_{p^{\nu}}^{\ell} y^{-1}$ for some $\ell \in \ints$.  This shows that the action of
$G^{str}$ is defined over $K_0(\zeta_{p^{\nu}}) = K_{\nu}$.  So $f^{str}$ is defined over $K_{\nu}(a, \sqrt{1-a})$ as a
$G^{str}$-cover.

To conclude the proof, note that either $v(1-a) = 0$ or $v(1-a) = 2v(s) \in 2\ints$ (Lemma
\ref{Lvcorrect}).  Since $v(K_{\nu}(a)^{\times}) \supset \ints$ and $p \ne 2$, it follows that $\sqrt{1-a} \in K_{\nu}(a)$.
\end{proof}

Recall that we have fixed $a = 1 - \frac{s^2}{r^2}$, unless we are in the situation
of Propositions \ref{Pa0} (2) or \ref{Pa1} (2).  In these cases, $a = 1 - \left(\frac{s - \sqrt[5]{5^{4\nu +1}\binom{r+s}{5}}}{r}\right)^2$ or
$a = 1 - \left(\frac{s - \sqrt[5]{5^{4\nu +1}\binom{s}{5}}}{r}\right)^2$, respectively (see Propositions \ref{Panot01}, \ref{Pa0}, and \ref{Pa1}).
\begin{prop}\label{Pfstrstdef}
\begin{enumerate}
\item If $v(a) = v(a-1) = 0$, then  $(f^{str})^{st}$ can be defined over a tame extension $(K^{str})^{st}$ of $K_{\nu}$.
\item Suppose $v(a) > 0$ or $v(a-1) > 0$.  
\begin{enumerate}
\item
We have $\nu > 1$.
\item
If $p > 5$, $v(a) = \nu - 1$, or $v(\sqrt{1-a}) = \nu -1$, then
$(f^{str})^{st}$ can be defined over a tame extension $(K^{str})^{st}$ of $K_{\nu}(a)(\sqrt[p]{1+u})$, where $u \in K_0(a)$ has valuation $1$. 
\item
If $p = 5$ and both $v(a)$ and $v(\sqrt{1-a})$ are less than $\nu -1$, then $(f^{str})^{st}$ can be defined over a tame extension $(K^{str})^{st}$ of 
$K_{\nu}(\sqrt[p]{\eta})(\sqrt[p]{1+u}, \sqrt[p]{1+u'}),$ where $\eta \in \rats$ has prime-to-$p$ valuation, $u \in K_0$ has valuation $1$, 
and $u' \in K_0(\sqrt[p]{\eta})$ has valuation $1$.  
\end{enumerate}
\end{enumerate}
\end{prop}

\begin{proof}
Let $K^{str} = K_{\nu}(a)$, as in Proposition \ref{Pfstrdef}.

We will use the criterion of \cite[Proposition 4.9]{Ob:fm}, which states that if $f$ has monotonic stable reduction, if 
$L/K^{str}$ is such that $G_L$ fixes a smooth point of $\ol{X}$ on each tail
of $\ol{X}$, and if $G_L$ fixes a smooth point of $\ol{Y}^{str}$ above each tail of $\ol{X}$, then $(f^{str})^{st}$ can be defined over a tame extension of 
$L$.

\emph{To (1):}
It is clear that there is a $K^{str}$-rational point specializing to each \'{e}tale tail (namely, $x=0$ and $x=a$).  
By (\ref{Efstr1}) and (\ref{Efstr2}), the fibers of $f^{str}$ above $x=0$ and $x=a$ consist of $K^{str}$-rational points.  These points are fixed by
$G_{K^{str}}$.  Now, if $v(a) = v(a-1) = 0$, there are no inseparable tails (Proposition \ref{Ponly0insep}) and $K^{str} = K_{\nu}$.
This, together with the criterion of \cite[Proposition 4.9]{Ob:fm}, proves (1).

\emph{To (2a):}
If $v(a) > 0$, then Lemma \ref{Lvaint} shows that $\nu > 1$.  If $v(a-1) > 0$, 
then Lemmas \ref{Lvsmall} and \ref{Lvcorrect} show that $\nu > 1$.

\emph{To (2b):}
Suppose we are in the situation of (2) and $v(a) > 0$.  
By Proposition \ref{Pinseptailloc} (1), there is a $p^{\nu - v(a)}$-inseparable tail $\ol{X}_c$ to which the 
$K^{str}$-rational point $x=\frac{a}{2}$ specializes.
Then each component of $\ol{Z}^{str}$ above $\ol{X}_c$ contains the specialization of one of $z = \pm\sqrt{-1}$.
Consider the cover $(Y^{str})' \to Z^{str}$, where $(Y^{str})' = Y^{str}/Q$ with $Q$ the unique
subgroup of $G^{str}$ of order $p^{\nu - v(a)}$.  Equation (\ref{Ebetterg}) shows that this cover can be given by the
equation $$y^{p^{v(a)}} = 1 + 2s\mu\frac{z}{z^2-1} + O(\mu^2),$$ where $v(\mu) = v(a)$, and where the terms on the
right-hand side are all in $K_0(a)$.  Plugging in $z = \pm \sqrt{-1}$,
we get that $y^{p^{v(a)}} = 1 + \alpha$, with $v(\alpha) = v(a)$.  By a binomial expansion (cf.\ proof of Lemma \ref{Lartinschreier}), 
$1 + \alpha$ has a $p^{v(a) -1}$st root in $K_0(a)$, which is of the form $1 + u$ with $v(u) = 1$.  So
$G_{K^{str}(\sqrt[p]{1+ u})}$ fixes the fiber above the specialization of $x = \frac{a}{2} \in \ol{X}_c$ for this cover.  Since the quotient by $Q$
is radicial above $\ol{X}_c$, it follows that $G_{K^{str}(\sqrt[p]{1+ u})}$ fixes the fiber of $\ol{Y}^{str}$ 
above the specialization of $x = \frac{a}{2} \in \ol{X}_c$ pointwise.  In particular, it fixes a point above $\ol{X}_c$.

If, instead, $v(a-1) > 0$, then there is a (not new) $p^{\nu - v(s)}$-inseparable tail
$\ol{X}_c$ containing the specialization of $x=1$.
Then each component of $\ol{Z}^{str}$ above $\ol{X}_c$ contains the specialization of one of $z = \pm\sqrt{1-a}$.
Consider the cover $(Y^{str})' \to Z^{str}$, where $(Y^{str})' = Y^{str}/Q$ and $Q$ is the unique
subgroup of $G^{str}$ of order $p^{\nu - v(s)}$.  After multiplying by $p^{v(s)}$th powers, this cover can be given by the
equation $$y^{p^{v(s)}} = \left(\frac{z+1}{z-1}\right)^r = \left(\frac{2z}{z-1} - 1\right)^r.$$  Recall that, by Lemma
\ref{Lvcorrect}, we have $v(s) = v(\sqrt{1-a})$. Plugging in $z = \pm \sqrt{1-a}$ and multiplying by $(-1)^r$, which is a 
$p^{v(s)}$th power in $K^{str}$, 
we get that $y^{p^{v(s)}} = 1 + \alpha$, with $v(\alpha) = v(s) = v(\sqrt{1-a})$.  As in the previous paragraph,
we conclude that there exists $u \in K_0(a)$ with $v(u) = 1$ such that $G_{K^{str}(\sqrt[p]{1+ u})}$ fixes a point above $\ol{X}_c$.

By Proposition \ref{Pinseptailloc} (1), these are the only inseparable tails in the situation of (2b).  Applying the criterion
of \cite[Proposition 4.9]{Ob:fm} finishes the proof of (2b).

\emph{To (2c):}
Assume we are in the situation of (2c).  By Propositions \ref{Pa0} and \ref{Pa1}, we have $a \in \rats \subset K_{\nu}$, so
$K^{str} = K_{\nu}$ and $u$ (from (2b)---the inseparable tail in that case still exists in this case) is in $K_0$.  

Suppose $v(a) > 0$.  Then there is a new inseparable $p^{\nu - v(a) - 1}$-tail $\ol{X}_{c'}$ 
containing the specialization of a $K' := K_0(\sqrt[p]{\eta})$-rational point, where $\eta \in \rats$ and $p \nmid v(\eta)$ 
(Proposition \ref{Pinseptailloc} (2)).   
Each component of $\ol{Z}^{str}$ above $\ol{X}_{c'}$ contains the specialization of the $K'$-rational point $z=d$ of Proposition
\ref{Pinseptailloc} (2).
Consider the cover $(Y^{str})' \to Z^{str}$, where $(Y^{str})' = Y^{str}/Q$ and $Q$ is the unique
subgroup of $G^{str}$ of order $p^{\nu - v(a) - 1}$.  This cover is given by the equation
$y^{p^{v(a) + 1}} = g(z)$.  Now, by Lemma \ref{L5insep} (1), $g(d) = (-1)^{r+s} + \alpha'$, with $v(\alpha') = v(a) + 1$. 
Since $(-1)^{r+s}$ is a $p^{v(a) + 1}$st power in $K_{\nu}$, we may assume $g(d) = 1 \pm \alpha'$.
By the binomial expansion, $1 \pm \alpha'$ has a $p^{v(a)}$th root in 
$K'$, which is of the form $1 + u'$ with $v(u') = 1$.  So
$G_{K'(\sqrt[p]{1+ u'})}$ fixes a point above $\ol{X}_c$ for this cover.  Since the quotient by $Q$
is radicial above $\ol{X}_c$, it follows that $G_{K'(\sqrt[p]{1+ u'})}$ fixes a point above $\ol{X}_{c'}$.

If, instead, $v(a-1) > 0$, the exact same proof (using Proposition \ref{Pinseptailloc} (3) instead of (2) and Lemma \ref{L5insep} (2)
instead of (1)) shows that we can find $\eta \in \rats$ with $p \nmid v(\eta)$ and $u' \in K' := K_0(\sqrt[p]{\eta})$ such that 
$G_{K'(\sqrt[p]{1+u'})}$ fixes a point above the new inseparable tail $\ol{X}_{c'}$.  We have now addressed all the inseparable tails 
(Proposition \ref{Pinseptailloc}), so we can apply the criterion of \cite[Lemma 4.9]{Ob:fm} to complete the proof of (2c).
\end{proof}

\begin{prop}\label{Pfstdef}
In all cases of Proposition \ref{Pfstrstdef}, the stable model $f^{st}$ of $f$ can be defined over $(K^{str})^{st}$.
\end{prop}

\begin{proof}  Since the branch loci of $f^{aux}$ and $f^{str}$ are the same, all branch points of 
$Y^{aux} \to Y^{str}$ are ramification points of $f^{str}$.  Thus their specializations do not coalesce on $\ol{f}^{str}$, and 
$G_{(K^{str})^{st}}$ must permute them trivially.  So they are each defined over $(K^{str})^{st}$.  By Lemma \ref{Lstr2orig}, $f^{st}$ is defined over
$(K^{str})^{st}$.
\end{proof}

\subsubsection{Higher ramification groups}\label{Shigherram}
In this section, we calculate the bounds on the conductors of the fields in Proposition \ref{Pfstrstdef}, considered as extensions of $K_0$.  
Recall that, for any finite Galois extension $L/K$ ($K$ a finite extension of $K_0$), we write $h_{L/K}$ for the conductor of $L/K$.

\begin{prop}\label{Pconductor}
\begin{enumerate}
\item If $L$ is a tame extension of $K_{\nu}$ ($\nu \geq 1$), then $L/K_0$ is Galois and $h_{L/K_0} = \nu - 1$.
\item 
Let $\eta \in K_0$ be such that $p \nmid v(\eta)$, let $u \in K_0$ such that $v(u) = 1$, and let
$u' \in K_0(\sqrt[p]{\eta})$ such that $v(u') = 1$.  
Let $\nu > 1$, and let $K'$ be a tame extension of $K := K_{\nu}(\sqrt[p]{\eta})(\sqrt[p]{1+u}, \sqrt[p]{1+u'})$.  
If $L$ is the Galois closure of $K'$ over $K_0$, then $h_{L/K_0} = \max(\nu - 1, \frac{p}{p-1})$.
\end{enumerate}
\end{prop}

\begin{proof}
\emph{To (1):} By \cite[IV, Corollary to Proposition 18]{Se:lf}, $h_{K_{\nu}/K_0} = \nu -1$.  
Note that any tame extension of a Galois extension of $K_0$ is, in fact, Galois over $K_0$.  So $L/K_0$ is Galois.
By \cite[Lemma 2.2]{Ob:fm}, its conductor is also $\nu -1$.

\emph{To (2):} 
Fix an algebraic closure $\ol{K_0}$ of $K_0$.
Let $K''$ be the Galois closure of $K$ over $K_0$.  Then, because any tame extension of a Galois extension 
of $K_0$ is Galois over $K_0$, we see that the compositum $K'K''$ 
is Galois over $K_0$, thus $L = K'K''$.  In particular, $L/K''$ is tame.  By \cite[Lemma 2.2]{Ob:fm}, it suffices to show 
that $h_{K''/K_0}  = \max(\nu-1, \frac{p}{p-1})$.

Let $u'_i$, $1 \leq i \leq c$, be the distinct Galois conjugates of $u'$ in 
an algebraic closure of $K_0$.  Then $K''$ is the compositum of $K_{\nu}$ with 
$$M := K_1(\sqrt[p]{\eta})(\sqrt[p]{1+u}, \sqrt[p]{1+ u'_1}, \ldots\sqrt[p]{1 + u'_c}).$$  Note that $M/K_0$ is Galois.
The conductor of a compositum is the maximum of the conductors (\cite[Lemma 2.3]{Ob:fm}), so 
$h_{K''/K_0} = \max(\nu - 1, h_{M/K_0})$.  It will suffice to show that $h_{M/K_0} = \frac{p}{p-1}$.

Note that the absolute ramification index of $K_1$ is $p-1$.
By \cite[Lemma 3.2(ii)]{Ob:ce}, $h_{K_1(\sqrt[p]{\eta})/K_1} = p$.  Since the lower numbering is invariant under subgroups,
the greatest lower jump for the higher ramification filtration of $G(K_1(\sqrt[p]{\eta})/K_0)$ is $p$. 
Then $h_{K_1(\sqrt[p]{\eta})/K_0} = \frac{p}{p-1}$ by the definition of the upper numbering.

Now, by \cite[Lemma 3.2, Remark 3.4]{Ob:ce}, 
$$h_{K_1(\sqrt[p]{\eta})(\sqrt[p]{1+u'_i})/K_1(\sqrt[p]{\eta})} < \frac{p}{p-1}(p(p-1)) - p(p-1) = p.$$
The same holds for $K_1(\sqrt[p]{\eta})(\sqrt[p]{1+u})/K_1(\sqrt[p]{\eta})$.  Thus, again using \cite[Lemma 2.3]{Ob:fm},
we see that $h_{M/K_1(\sqrt[p]{\eta})} < p$.  By \cite[Lemma 2.1]{Ob:ce} (with our $K_0$, $K_1(\sqrt[p]{\eta})$, and 
$M$ playing the roles of $K$, $L$, and $M$ in \cite{Ob:ce}), either $h_{M/K_0} = \frac{p}{p-1}$ or $h_{M/K_0} > \frac{p}{p-1}$
and $$\frac{1}{p(p-1)}(h_{M/K_1(\sqrt[p]{\eta})} - p) = h_{M/K_0} - \frac{p}{p-1}.$$
Since the left hand side is negative whereas the right hand side is positive, the second option cannot hold.
So $h_{M/K_0} = \frac{p}{p-1}$.
\end{proof}    

\proof[Proof of Proposition \ref{Pm2tau1}.]
Note that in Proposition \ref{Pfstrstdef} (2b), either $a \in K_0$ or $a \in K_0(\sqrt[p]{\eta})$, for some $\eta \in \rats$ with $p \nmid v_p(\eta)$.
Thus Proposition \ref{Pconductor} shows that the Galois closures (over $K_0$) of all of the extensions $(K^{str})^{st}$ in 
Proposition \ref{Pfstrstdef} have conductor $< \nu$ over $K_0$.  This, along with Proposition \ref{Pfstdef} and the fact that
$\nu \leq n$, completes the proof of Proposition \ref{Pm2tau1}. \qed

\proof[Proof of Theorem \ref{Tmain}]
By Proposition \ref{Pconductor} (1), the extensions in Propositions \ref{Pm2tau3} and \ref{Pm2tau2} have conductor 
$< n$ over $K_0$.  This fact, combined with Proposition \ref{Pm2tau1} and 
Remark \ref{Rlocaltoglobal}, proves Theorem \ref{Tmain}. \qed

\section{Further questions} \label{CHfurther}

\begin{question}\label{Qgap}
Does Theorem \ref{Tmain} hold even if we allow $p=3$ or no prime-to-$p$ branch points?
\end{question}

If there are no prime-to-$p$ branch points ($\tau = 0$, in the language of \S\ref{CHmain}), then $\ol{X}$ can have up to two new (\'{e}tale) tails. 
This allows for a greater proliferation of subcases (for
instance, the two new tails could branch out from the same point of the original component).
Techniques similar to those used in \S\ref{Stauequals1} should work, and we have worked out some of the easier cases (unpublished).  However, keeping track of all the possibilities will be quite tedious, and 
we do not pursue the calculation here.  

If we allow $p = 3$ (even if we assume that there is a branch point of $f$ with prime-to-$p$ branching index), then the
stable reduction of $f^{str}$ looks very different than what we determine in \S\ref{Stauequals1}.  In particular, Lemma
\ref{Lnobigjumps} does not hold.  In fact, if $\ol{X}_b$ is an \'{e}tale tail with ramification invariant $\sigma_b = \frac{3}{2}$, then it
\emph{must} intersect a $p^2$-component (because the conductor of a $\ints/3$-extension in residue characteristic $3$ cannot be $3$).
To pursue this case using the techniques of \S\ref{Stauequals1} would require good, explicit conditions characterizing
the following extensions, in analogy with Lemma \ref{Lartinschreier}:  
Say $R$ is a complete discrete valuation ring with fraction field $K$ of
characteristic $0$, residue field $k$ algebraically closed of characteristic $3$, and uniformizer $\pi$.  
Suppose further that $R$ contains
the $9$th roots of unity.  Write $A = R\{t\}$ and $L = \Frac(A)$.  We wish to characterize $\ints/9$-extensions $M$ of 
$L$ such that the normalization $B$ of $A$ in $M$ satisfies the following conditions:
\begin{itemize}
\item $\Spec B/\pi \to \Spec A/\pi$ is an \'{e}tale extension 
with conductor 3. 
\item $\Spec B/\pi$ is integral.
\end{itemize}
For a similar statement in residue characteristic $2$, see \cite[Proposition C.1]{Ob:fm}.

\begin{question}\label{Qdeformation}
What if the condition $m_G = 2$ is relaxed in Theorem \ref{Tmain}?
\end{question}

Allowing arbitrary $m_G$ will require some new techniques, as the strong
auxiliary cover is in general no longer a $p^{\nu}$-cyclic extension of
$\proj^1$, making it difficult to perform explicit computations.  
However, it turns out to be true that, if the cover has bad reduction, the deformation data above the original component 
of $\ol{X}$ are all
\emph{multiplicative}, even for arbitrary $m_G$ (this holds for $m_G=2$, by adapting Lemma \ref{Lfullbranch} to the $\tau = 0$ and $\tau = 2$ cases).  
One might hope to obtain results 
using variations on the deformation theory of torsors under multiplicative group schemes that Wewers develops in
\cite{We:def}.  In the case where $v_p(|G|) = 1$, such deformation theory leads to a conceptual understanding of where the disks corresponding
to the \'{e}tale tails of the stable reduction (as in \S\ref{Setaletails}) are located, in particular showing that such disks must contain rational points over
a small field.  It also does not rely at all on having $m_G = 2$.  
If one considers a more generalized version of Theorem \ref{Tmain} where no restrictions on $m_G$ or the branching indices are required,  
then generalizing \cite{We:def} to the case of larger cyclic $p$-Sylow subgroups might provide a more conceptual proof (in particular, with regards
to the analog of \S\ref{Setaletails}).
 
\appendix
\section{An Example of Wild Monodromy}\label{Awildexample}

In \cite[Theorem 1.1]{Ob:vc}, it is proven that if $f: Y \to X$ is a three-point $G$-cover defined over a complete discrete valuation 
field $K$ of mixed characteristic $(0,p)$, where $G$ has a cyclic $p$-Sylow subgroup of order $p^n$ and $p$ does not divide the order of the center of $G$, 
then the wild monodromy group $\Gamma_w$ of $f$  (i.e., the $p$-Sylow subgroup of $\Gal(K^{st}/K)$, notation of \S\ref{Sstable}) has
exponent dividing $p^{n-1}$.  In particular, if $n = 1$, then $\Gamma_w$ is trivial (this is \cite[Th\'{e}or\`{e}me 4.2.10]{Ra:sp}).
In this appendix, we exhibit an example showing that the wild monodromy can be nontrivial when $n > 1$.  This 
is surprisingly difficult (see Remark \ref{Rnosolvable}).  Our example is based on the calculations of \S\ref{Stauequals1}.
   
Throughout the appendix, let $G = SL_2(251)$, and let $k$ be an algebraically closed field of characteristic $p = 5$.  
Let $R_0 = W(k)$ and $K_0 = \Frac(R_0)$.  Lastly, let $K = K_0(\mu_{5^{\infty}})$ (that is, we adjoin all $5$th-power roots of unity
to $K$), and let $R$ be the valuation ring of $K$.  Note that $G$ has a cyclic $5$-Sylow subgroup of order $5^3 = 125$ and $m_G = 2$.
We normalize all valuations on $R_0$, $K_0$, or any extensions thereof so that $v(5) = 1$.

\begin{prop}\label{P251cover}
There exists a three-point $G$-cover $f: Y \to X = \proj^1_K$, defined over $K$, such that the branching indices of the three branch points are
$e_1$, $e_2$, and $e_3$, with $v_5(e_1) = 0$, $v_5(e_2) = 2$, and $v_5(e_3) = 3$.
\end{prop}

\begin{proof}
We show that such a cover can be defined over $\rats^{ab}$.  Since $\rats^{ab} \hookrightarrow K$, this will prove the proposition.

Let $\alpha = \matrix{1}{1}{0}{1} \in G$.  This has order $251$.
We claim there exists $\beta = \matrix{a}{b}{c}{d} \in G$ satisfying the following properties:
\begin{itemize}
\item The order of $\beta$ is 250.
\item The order of $\alpha\beta$ is 50.
\item The matrices $\alpha$ and $\beta$ generate $SL_2(251)$.
\end{itemize}

To prove the claim, first note that any $GL_2(251)$-conjugacy class in $G$ is determined by the trace of the matrices it contains, 
unless the trace is $\pm 2$.  In particular, the trace of a matrix determines its order if it is not $\pm 2$. 
Let $\tau$ be the trace of the matrices in some conjugacy class of order $250$, and let $\rho$ be the trace of the matrices in some conjugacy
class of order $50$.  Then $\tau$, $\rho$, $2$, and $-2$ are pairwise distinct.
Choose $a$, $b$, $c$, and $d$ in $\FF_{251}$ solving the (clearly solvable) system of equations:
\begin{align*}
a+d =& \ \tau \\
a + c+ d =& \ \rho \\
ad - bc =&\  1
\end{align*}
Since the trace of $\alpha\beta$ is $a + c + d$, these equations ensure that $\beta$ and $\alpha\beta$ have the desired orders.  Let $\ol{\alpha}$ and
$\ol{\beta}$ be the images of $\alpha$ and $\beta$ in $H := PSL_2(251)$.  Since $c \ne 0$,
one checks that $\ol{\beta}$ does not normalize the subgroup generated by $\ol{\alpha}$.  
Then, by \cite[II, Hauptsatz 8.27]{Hu:eg}, we have that $\ol{\alpha}$ and $\ol{\beta}$ generate $H$.  Furthermore, since $\beta$ is diagonalizable
over $GL_2(251)$ and has eigenvalues of order $250$, we have $\beta^{125} = -I_2$.  Since $\ol{\alpha}$ and $\ol{\beta}$ generate $H$, and 
$\beta$ generates $\ker(G \to H)$, then $\alpha$ and $\beta$ generate $G$.

Consider the triple $([\alpha], [\beta], [\alpha\beta]^{-1})$ of conjugacy classes of $G$.  By \cite[I, Theorem 5.10 and Remark afterward]{MM:ig}, 
this triple is rigid.  By \cite[I, Theorem 4.8]{MM:ig}, there exists a three-point $G$-cover of $\proj^1$, defined over $\rats^{ab}$, with branching indices
$e_1 = \ord(\alpha) = 251$, $e_3 = \ord(\beta) = 250$, and $e_2 = \ord((\alpha\beta)^{-1}) = 50$.  This completes the proof of the proposition.
\end{proof}

\begin{prop}
If $f: Y \to X = \proj^1_K$ is a cover satisfying the properties of Proposition \ref{P251cover}, then $f$ has nontrivial wild monodromy $\Gamma_w$.
\end{prop}

\begin{proof}
To fix notation, we assume $f$ is branched at $x=0$, $x=1$, and $x=\infty$ of index $e_1$, $e_2$, and $e_3$, respectively, with $v_5(e_1) = 0$, $v_5(e_2) = 2$, and $v_5(e_3) = 3$.
By \cite[Lemma 3.2]{Ob:ac}, the stable reduction of $f$ has both a primitive \'{e}tale tail and a new \'{e}tale tail. 
Construct the strong auxiliary cover $f^{str}: Y^{str} \to X$ of $f$ (\S\ref{Saux}).  
This is a four-point $G^{str}$-cover, with $G^{str} \cong \ints/125 \rtimes \ints/2$ such that the
action of $\ints/2$ is faithful.  As in (\ref{Efstr1}) and (\ref{Efstr2}), this cover is given by
\begin{align}
z^2 =& \ \frac{x-a}{x} \label{E1} \\
y^{125} =& \ g(z) := \left(\frac{z+1}{z-1}\right)^r \left(\frac{z + \sqrt{1-a}}{z- \sqrt{1-a}}\right)^s, \label{E2}
\end{align}
where $r$ and $s$ are integers satisfying $v_5(r) = 0$ and $v_5(s) = 1$.  Replacing $y$ with a prime-to-$5$ power, we can assume $s = 5$.
By Lemma \ref{Lfullwildspec}, we have $v(1-a) > 0$ in $K(a)/K$, and then
Lemma \ref{Lvcorrect} and Proposition \ref{Pa1} (1) show that we can take $a = 1 - \frac{25}{r^2}$.  In particular, $f^{str}$ is defined over $K$.
By Proposition \ref{Pinseptailloc}, the stable model $(f^{str})^{st}: (Y^{str})^{st} \to X^{st}$ of $f^{str}$ has a new inseparable 
$5$-tail $\ol{X}_c$.  We claim that there is an extension $L/K$ such that
$\Gal(L/K)$ fixes $\ol{X}_c$ pointwise, and acts nontrivially of order $5$ on the stable reduction 
$\ol{f}^{str}: \ol{Y}^{str} \to \ol{X}$ of $f^{str}$ above $\ol{X}_c$.  
Since $(f^{str})^{st}$ is a quotient of the stable model $(f^{aux})^{st}$ of the (standard) 
auxiliary cover $f^{aux}$ (\S\ref{Saux}), then $\Gal(L/K)$ will act nontrivially of order divisible by $5$
above $\ol{X}_c$ in $(f^{aux})^{st}$ as well.
Lastly, since, above an \'{e}tale neighborhood of $\ol{X}_c$, the stable model $f^{st}$ of $f$ is isomorphic to a set of disconnected copies
of $(f^{aux})^{st}$, the same holds true over a formal neighborhood $\hat{X} \subseteq X^{st}$ of $\ol{X}_c$.  That is, 
$$Y^{st} \times_{X^{st}} \hat{X} \cong \Ind_{G^{aux}}^G (Y^{aux})^{st} \times_{X^{st}} \hat{X}.$$  Since $f$ is defined over $K$,
the $\Gal(L/K)$-action 
on $Y^{st} \times_{X^{st}} \hat{X}$ is determined by the action on $(Y^{aux})^{st} \times_{X^{st}} \hat{X}$ and 
the fact that it commutes with the $G$-action.
So the action of $\Gal(L/K)$ on $\ol{Y} \times_{\ol{X}} \ol{X}_c$, thus on $\ol{Y}$, is nontrivial of order divisible by $5$.  
This means that $f$ has nontrivial wild monodromy.

It remains to prove the claim.  Let $Z^{str} = Y^{str}/(\ints/125)$, with stable model $(Z^{str})^{st}$ and stable reduction $\ol{Z}^{str}$. 
Then $z$ is a coordinate on $Z^{str}$, and by Proposition \ref{Pinseptailloc} (3), there is a component of $\ol{Z}^{str}$ above $\ol{X}_c$
containing the specialization $\ol{d}$ of $$z = d := \frac{2 \cdot 5^{7/5}}{r},$$ where we can use any choice
of $5$th root.  Since $\ol{X}_c$ is a $p$-component, there are $25$ points of $\ol{Y}^{str}$ above $\ol{d}$.  
If $g(d)$ is a $5$th power, but not a $25$th power, in $K(d) = K(\sqrt[5]{5})$, 
and if $L = K(d, \sqrt[25]{g(d)})$, then the action of $\Gal(L/K(d))$ will permute these 25 points in orbits of order $5$, and we will be done 
(it turns out that $\Gal(K(d)/K)$ fixes $\ol{d}$, even though it clearly does not fix $d$).
This follows from Lemma \ref{Lpthpower} below.
\end{proof}

\begin{lemma}\label{Lpthpower}
Let $d = \frac{2 \cdot 5^{7/5}}{r}$, where $r$ is a prime-to-$p$ integer and we choose any $p$th root of $5$.  Let $g$ be the rational function in
(\ref{E2}), with $s=5$ and $a = 1 - \frac{25}{r^2}$.  Then $g(d)$ is a $5$th power, but not a $25$th power,
in $K(\sqrt[5]{5})$.
\end{lemma}
 
\begin{proof}
Fix a $5$th root of $5$ in $\ol{K}$, which we will denote by either $\sqrt[5]{5}$ or $5^{1/5}$.  
We first note that $g(d) \in K_0(\sqrt[5]{5})$.  
By (\ref{E513}), we have 
$$g(d) = \pm\left(1 - \frac{8r^3}{75}d^3 - \frac{32r^5}{5^5}d^5\right) + o(5^{9/4}),$$
where the $o$ represents terms of valuation greater than $\frac{9}{4}$.
Upon plugging in $d$ and simplifying, this gives
$$g(d) = \pm(1 - 3\cdot 5^{11/5} - 4 \cdot 5^2) + o(5^{9/4}).$$
Using the binomal theorem, we see that $g(d)$ has a $5$th root $\eta$ in $K_0(\sqrt[5]{5})$, and
$$\eta = \pm(1 - 3\cdot 5^{6/5} - 20) + o(5^{5/4}).$$  We wish to show that $\eta$ is not a $5$th power in $K(\sqrt[5]{5})$.

Now, since $K(\sqrt[5]{5})/K_0(\sqrt[5]{5})$ is abelian, any subextension is Galois.  So if $\eta$ is a $5$th power in $K(\sqrt[5]{5})$,
then taking a $5$th root of $\eta$ must generate a Galois extension of $K_0(\sqrt[5]{5})$.  This is clearly not the case unless 
$\eta$ is already a $5$th power in $K_0(\sqrt[5]{5})$, so it suffices to show that $\eta$ is not a $5$th power in $K_0(\sqrt[5]{5})$.
Since $-1$ is a $5$th power in $K_0$, we may assume that $\eta = 19 + 3 \cdot 5^{6/5} + o(5^{5/4})$.

Suppose that $\theta \in K_0(\sqrt[5]{5})$ such that $\theta^5 = \eta$, and write 
$$\theta = \alpha + \beta \cdot 5^{1/5} + \gamma \cdot 5^{2/5} + \delta \cdot 5^{3/5} + \epsilon \cdot 5^{4/5},$$
where $\alpha$, $\beta$, $\gamma$, $\delta$, and $\epsilon$ are in $K_0$.  Comparing valuations, we see that 
$\theta \in R_0[\sqrt[5]{5}]$.  Equating coefficients of $1$ and $5^{6/5}$ gives the equations
\begin{align*}
\alpha^5 + 5\beta^5 &\equiv 19 \pmod{25}\\
\alpha^4\beta &\equiv 3 \pmod{5}
\end{align*}
The second equation yields $\alpha \equiv 4 \pmod{5}$, and then the first equation yields $\beta \equiv 3 \pmod{5}$.  But then
$\alpha^5 + 5\beta^5 \equiv 24 + 5 \cdot 18 \equiv 14 \not \equiv 19 \pmod{25}$.  So $\epsilon$ cannot exist, and we are done. 
\end{proof}

\begin{remark}\label{Rnosolvable}
The example above is quite complicated, and is not generalizable in any meaningful way (for instance, it depends critically on having $p=5$). 
One hopes for easier examples, but they are difficult to come by.  
For instance, results of \cite[\S7.1]{Ob:fm} show that no examples of three-point $G$-covers with
nontrivial wild monodromy can exist when $G$ is $p$-solvable and $m_G > 1$.  So if one wants to find an easier example where $p$ does not
divide the order of the center of $G$, one needs to look either at a group that is not $p$-solvable, or at a group where $m_G = 1$. 
By Burnside's theorem (see, e.g., \cite[Lemma 2.2]{Ob:vc}),
having $m_G = 1$ implies that $G$ is of the form $G \cong H \rtimes \ints/p^n$, where the action of $\ints/p^n$ on $H$ is faithful.
\end{remark}

\end{document}